\documentclass[12pt,twoside,reqno,a4paper]{article}

\usepackage[utf8]{inputenc}
\usepackage{amsmath,amsthm}
\usepackage{lmodern}

\usepackage{monoidal}
\usepackage{tikz-cd}
\usepackage{mathtools}
\usepackage{ragged2e}
\usepackage{nicematrix}
\usepackage{microtype}
\usepackage[ngerman, british]{babel}
\usetikzlibrary{arrows, arrows.meta, babel, positioning, calc}
\tikzcdset{arrow style=tikz, diagrams={>=stealth}}

\newtheorem{lemma}{Lemma}[section]
\newtheorem{proposition}[lemma]{Proposition}
\newtheorem{remark}[lemma]{Remark}
\newtheorem{definition}[lemma]{Definition}
\newtheorem{example}[lemma]{Example}
\newtheorem{corollary}[lemma]{Corollary}
\newtheorem{theorem}[lemma]{Theorem}

\usepackage{graphicx}
\usepackage{url}

\usepackage{ifpdf}
\ifpdf
\pdfoutput=1
\usepackage[bitstream-charter]{mathdesign}
\usepackage[pdfusetitle,colorlinks=true,linktoc=all]{hyperref}
\else
\usepackage{times}
\usepackage[dvips,ps2pdf,linktoc=all]{hyperref}

\fi

\ifdefined\hypersetup
\hypersetup{
	pdfkeywords={}, linkcolor=black, citecolor=black, filecolor=black, urlcolor=black,
}
\fi

\usepackage{fancyhdr}

\usepackage{blindtext}

\usepackage{fancyhdr}
\SetMathAlphabet{\mathcal}{normal}{OMS}{cmsy}{m}{n}
\usepackage[T1]{fontenc}

\title{Generalised 6j symbols over the category of $G$-graded vector spaces}
\author{Fabio Lischka\\FAU Erlangen-Nürnberg\\Email: \href{mailto:fabio.lischka@fau.de}{fabio.lischka@fau.de}}
\date{7th October 2021}

\begin{document}

\maketitle

\begin{abstract}
Any choice of a spherical fusion category defines an invariant of oriented closed 3-manifolds, which is computed by choosing a triangulation of the manifold and considering a state sum model that assigns a 6j symbol to every tetrahedron in this triangulation. This approach has been generalized to oriented closed 3-manifolds with defect data by Catherine Meusburger. In a recent paper, she constructed a family of invariants for such manifolds parametrised by the choice of certain spherical fusion categories, bimodule categories, finite bimodule functors and module natural transformations. Meusburger defined generalised 6j symbols for these objects, and introduces a state sum model that assigns a generalised 6j symbol to every tetrahedron in the triangulation of a manifold with defect data, where the type of 6j symbol used depends on what defect data occur within the tetrahedron. The present work provides non-trivial examples of suitable bimodule categories, bimodule functors and module natural transformation, all over categories of $G$-graded vector spaces. Our main result is the description of module functors in terms of matrices, which allows us to classify these functors when $G$ is a finite cyclic group. Furthermore, we calculate the generalised 6j symbols for categories of $G$-graded vector spaces, (bi-)module categories over such categories and module functors.
\end{abstract}

\setcounter{tocdepth}{2}
\tableofcontents

\section{Introduction}
In \cite{barrett_westbury}, Barrett and Westbury construct a family of invariants of closed oriented 3-manifolds, parametrised by spherical fusion categories, which generalise the Turaev-Viro invariant.
Given a spherical fusion category $\mathcal C$, calculating the corresponding invariant of an oriented closed 3-manifold $M$ involves
choosing a triangulation of $M$ and considering a state sum model that assigns a 6j symbol to every tetrahedron in this triangulation. These 6j symbols encode the associativity constraint of the category $\mathcal C$.

Meusburger has shown that this approach can be generalised to oriented closed 3-manifolds with defect data \cite{bm_tv+}.
The defect data of dimension 2, 1 and 0 are called defect areas, defect lines and defect vertices, respectively, and their union is a subset $D\subset M$.
The invariants defined in \cite{bm_tv+} are then parametrised by the choice of
a spherical fusion category for every connected component of $M\backslash D$, a bimodule category for every defect area, a finite bimodule functor for every defect line and a module natural transformation for every defect vertex. Meusburger defines generalised 6j symbols for these objects, and the state sum model she introduces involves assigning a generalised 6j symbol to every tetrahedron in the triangulation of $M$, where the type of 6j symbol used depends on what defect data occur within the tetrahedron.

In order to calculate this generalised state sum model and study its properties on examples, it is necessary to investigate non-trivial examples of spherical fusion categories, their bimodule categories as well as bimodule functors and calculate their 6j symbols. 
In this thesis, we will examine these structures in a simple example with applications.
Specifically, we only consider the category $\mathcal C=\vectgo$ of $G$-graded vector spaces as an example of a spherical fusion category, where $G$ is a group and $\omega$ is a 3-cocycle of $G$. For this choice of spherical category, the Turaev-Viro invariants from \cite{barrett_westbury} reduce to Dijkgraf-Witten theory, which was introduced in \cite{dw90}. The generalised state sum model from \cite{bm_tv+} should thus give rise to a generalised Dijkgraf-Witten theory for 3-manifolds with defects. 

Our main result is the explicit description of $\vectgo$-module functors in terms of matrices, which allows us to classify these functors in the case $G=\mathbb Z/n \mathbb Z$ for $n\neq 0$. Furthermore, we calculate the generalised 6j symbols for $\vectgo$, for $\monos$-bimodule categories and for $\vectgo$-module functors.

This thesis is structured as follows: In Section \ref{sec:prelim}, we fix our notation for spherical fusion categories and recall their basic properties. Section \ref{sec:vecgo} provides some background on the category $\vectgo$ of $G$-graded vector spaces. In Section \ref{sec:modc}, we recall the general definition of (bi-)module categories, (bi-)module functors and module natural transformations as well as the classification of semisimple $\vectgo$-module categories from \cite{egno} and use this classification to describe the semisimple $\monos$-bimodule categories. In Section \ref{sec:modf}, we introduce an explicit description of $\vectgo$-module functors, discuss a number of examples and give a classification of $\vectgo$-module functors for $G=\mathbb Z/n\mathbb Z$. In Section \ref{sec:6j} we introduce generalised 6j symbols following \cite{bm_tv+} and then calculate the generalised 6j symbols for $\vectgo$, for $\monos$-bimodule categories and $\vectgo$-module functors.

\clearpage
\section{Preliminaries} \label{sec:prelim}
Throughout this thesis, we fix an algebraically closed field $\mathbb F$.
Let $G$ be a group and let $1$ denote its neutral element.
The natural numbers $\mathbb N$ include $0$.
We write the Kronecker delta as
\[ \delta_{i, j} := \left\{ \begin{array}{ll} 1, & i=j \\ 0, & i\neq j \end{array} \right. . \]

\subsection{Group Cohomology}
We recall the definition of group cohomology from \cite[Application 6.5.5]{weibel}, see also \cite[Definition 2.3.1]{homalg}. 
Let $M$ be a $\mathbb Z[G]$-module with structure map $\opl: \mathbb Z[G] \times M \to M$. We use multiplicative notation for the group structure of $M$. For $n\geq 0$, the set of \emph{$n$-cochains with coefficients in $M$} is defined as
\[ C^n(G, M) := \Map(G^{\times n}, M). \]
The set $C^n(G, M)$ is an abelian group with the pointwise multiplication:
\[ (\eta \cdot \mu)(g_1, \dots, g_n) := \eta(g_1, \dots, g_n) \, \mu(g_1, \dots, g_n) \qquad (\eta, \mu\in C^n(G, M), g_i \in G) \]
The group homomorphisms
\[ d^n: \ C^n(G, M) \to C^{n+1}(G, M) \]
\vspace{-2\baselineskip}
\begin{multline*} d^n(\eta)(g_1, \dots, g_{n+1}) := (g_1 \opl \eta(g_2, \dots, g_{n+1}))\, \eta^{-1}(g_1 g_2, g_3, \dots, g_{n+1}) \\ \eta(g_1, g_2 g_3, g_4, \dots, g_{n+1})\, \dots \, \eta^{(-1)^n}(g_1, \dots, g_n g_{n+1})\, \eta^{(-1)^{n+1}}(g_1, \dots, g_n) \qquad (g_i\in G) \end{multline*}
satisfy $d^{n+1} \circ d^n = 0$, which makes
\[\begin{tikzcd}
	0 \ar[r] & C^0(G, M) \ar[r, "d^0"] & C^1(G, M) \ar[r, "d^1"] & C^2(G, M) \ar[r, "d^2"] & \dots
\end{tikzcd} \]
a cochain complex. We will write $d$ instead of $d^n$ whenever no confusion is possible. 
Denote by $Z^n(G, M) := \ker(d^n) \subset C^n(G, M)$ the subgroup of \emph{$n$-cocycles} and by $B^n(G, M) := \image (d^{n-1})\subset C^n(G, M)$ the subgroup of \emph{$n$-coboundaries}. Then the \emph{$n$-th group cohomology of $G$ with coefficients in $M$} is defined as the abelian group
\[ H^n(G, M) := Z^n(G, M) / B^n(G, M) . \]

Throughout this thesis, the module $M$ will always be either $M=\units$ (with multiplication as group structure and $G$ operating trivially) or $M=\Map(X, \units)$ for some $G$-set $X$ with multiplication in $M$ defined pointwise and $G$ acting on $M$ by
\[ (g \opl f)(x) = f(g^{-1} \opl x). \]
If $M=\Map(X, \units)$, we write $\Psi(g_1, \dots, g_n, x)$ instead of $\Psi(g_1, \dots, g_n)(x)$ for $\Psi\in C^n(G, \Map(X, \units))$ by an abuse of notation.

An $n$-cochain $\eta$ is called \emph{normalized} if $\eta(g_1, \dots, g_n) = 1$ whenever at least one of $g_1, \dots, g_n$ is $1$. By \cite[Lemma 6.1]{emNorm}, every cocycle is cohomologous to a normalized one. In fact, a more general statement is proved there:

\begin{lemma} \label{lem:norm_cochain}
	\cite[Lemma 6.1]{emNorm}
	Every $n$-cochain $\eta$ with $\dif \eta$ normalized is cohomologous to a normalized $n$-cochain $\eta'$.
\end{lemma}

For a subgroup $H \subset G$ and a $\mathbb Z[H]$-module $M$, we call the $\mathbb Z[G]$-module of $\mathbb Z[H]$-linear maps $\mathbb Z[G] \to M$
\[ \Coind_H^G(M) := \Hom_{\mathbb Z[H]}(\mathbb Z[G], M) \]
the \emph{coinduced $\mathbb Z[G]$-module}. We note that $\Coind_H^G(M)$ is isomorphic to $\Map(G/H, M)$ with group structure defined pointwise and $G$ acting on $\Map(G/H, M)$ by
\[ (g\opl f)(g'H) := f(g^{-1}g'H) \qquad (g, g'\in G, f\in \Map(G/H, M)). \]
The group cohomology of the subgroup $H$ can be expressed in terms of the group cohomology of $G$ using Shapiro's lemma:

\begin{lemma}[Shapiro's Lemma] \label{lem:shapiro}
	\cite[Lemma 6.3.2]{weibel}
	The cochain map
	\[ f^n: C^n(G, \Map(G/H, M)) \to C^n(H, M): \eta \mapsto \tilde{\eta}, \quad \tilde \eta (h_1, \dots, h_n) := \eta(h_1, \dots, h_n, H) \]
	induces an isomorphism
	\[ H^{\bullet}(G, \Coind_H^G (M)) \cong H^{\bullet} (H, M). \]
\end{lemma}

\subsection{Tensor categories} \label{ssec:tens_cat}
Unless otherwise specified, we denote the associativity constraint of a monoidal category $\mathcal C$ by $a$, the tensor unit $e$ and the left and right unit constraints $l$ and $r$, respectively. We denote the opposite category $\mathcal C^{op}$ and the monoidal category with reversed tensor product $\mathcal C^{rev}$, i.e., $\otimes_{\mathcal C^{rev}} := \otimes_{\mathcal C} \tau$, where $\tau$ is the flip functor exchanging the two factors of $\mathcal C \times \mathcal C$. Frequently, the unit constraints $l$ and $r$ will be suppressed.

We now recall the definition of duals in monoidal categories from \cite[Definition 2.10.1 and 2.10.2]{egno}, see also \cite[Definition 2.1.1]{tenscat}.
An object $X^*\in\mathcal C$ is a \emph{right dual} of $X$ if there exist morphisms
\[ \eval_X^R: X^* \otimes X \to e, \quad \coev_X^R: e \to X \otimes X^* \]
such that the composites
\[ \begin{tikzcd}[column sep=large]
	X \ar[r, "\coev_X^R \otimes \ident_X"] & (X \otimes X^*) \otimes X \ar[r, "a_{X, X^*, X}"] & X \otimes (X^* \otimes X) \ar[r, "\ident_X \otimes \eval_X^R"] & X,
\end{tikzcd} \]
\[ \begin{tikzcd}[column sep=large]
	X^* \ar[r, "\ident_{X^*} \otimes \coev_X^R"] & X^* \otimes (X\otimes X^*) \ar[r, "a^{-1}_{X^*, X, X^*}"] & (X^* \otimes X)\otimes X^* \ar[r, "\eval_X^R\otimes \ident_{X^*}"] & X^*
\end{tikzcd} \]
are the identity morphisms. These conditions are known as the snake identities. Likewise, an object $\ldual X\in\mathcal C$ is a \emph{left dual} of $X$ if there exist morphisms
\[ \eval_X^L: X\otimes \ldual X \to e, \quad \coev_X^L: e \to \ldual X\otimes X \]
such that the composites
\[ \begin{tikzcd}[column sep=large]
	X \ar[r, "\ident_X \otimes \coev_X^L"] & X \otimes (\ldual X \otimes X) \ar[r, "a^{-1}_{X, \ldual X, X}"] & (X \otimes \ldual X) \otimes X \ar[r, "\eval_X^L \otimes \ident_X"] & X,
\end{tikzcd} \]
\[ \begin{tikzcd}[column sep=large]
	\ldual X \ar[r, "\coev_X^L\otimes \ident_{\ldual X}"] & (\ldual X \otimes X)\otimes \ldual X \ar[r, "a_{\ldual X, X, \ldual X}"] & \ldual X \otimes (X\otimes \ldual X) \ar[r, "\ident_{\ldual X}\otimes \eval_X^L"] & \ldual X
\end{tikzcd} \]
are the identity morphisms. A monoidal category $\mathcal C$ is called \emph{right rigid} if every object in $\mathcal C$ has a right dual, \emph{left rigid} if every object has a left dual and \emph{rigid} if it is both right rigid and left rigid.

\begin{remark}
	Note that the left duals defined here correspond to the right duals in \cite{egno} and vice versa. Our notation follows \cite{tenscat}.
\end{remark}

We now introduce pivotal categories as described in \cite[Section 4.7]{egno}, see also \cite[Section 2.2]{tenscat}.
A \emph{pivotal structure} on a right rigid monoidal category $\mathcal C$ is a monoidal natural isomorphism $b: \Ident_{\mathcal C} \to **$, where $*: \mathcal C\to \mathcal C^{op, rev}$ denotes the duality functor. The pair $(\mathcal C, b)$ is called a \emph{pivotal category}.
A pivotal category is left rigid with $\ldual X = X^*$ and left evaluation and coevaluation
\begin{equation} \label{eq:left_eval} 
\begin{tikzcd}[column sep=large]
	\eval_X^L \esccol X\otimes X^* \ar[r, "b_X\otimes\ident"] & X^{**} \otimes X^* \ar[r, "\eval_{X^*}^R"] & e
\end{tikzcd} \end{equation}
\begin{equation} \label{eq:left_coev}
\begin{tikzcd}[column sep=large]
	\coev_X^L \esccol e \ar[r, "\coev_{X^*}^R"] & X^* \otimes X^{**} \ar[r, "\ident\otimes b_X^{-1}"] & X^*\otimes X.
\end{tikzcd} \end{equation}
Let $f: X \to X$ be an endomorphism in a pivotal category $\mathcal C$. The \emph{left trace} and \emph{right trace} of $f$ are defined as
\[ \begin{tikzcd}
	\trace_L(f) =\ e \ar[r, "\coev_X^L"] & X^* \otimes X \ar[r, "\ident_{X^*} \otimes f"] & X^* \otimes X \ar[r, "\eval_X^R"] & e\ \in \End_{\mathcal C}(e)
\end{tikzcd} \]
\[ \begin{tikzcd}
	\trace_R(f) =\ e \ar[r, "\coev_X^R"] & X \otimes X^* \ar[r, "f \otimes\ident_{X^*}"] & X\otimes X^* \ar[r, "\eval_X^L"] & e\ \in \End_{\mathcal C}(e),
\end{tikzcd} \]
respectively, where $\eval_X^L$ and $\coev_X^L$ are defined by (\ref{eq:left_eval}) and (\ref{eq:left_coev}). In particular, $\dim_{L/R}(X) := \trace_{L/R}(\ident_X)$ for $X\in\mathcal C$. $(\mathcal C, b)$ is called \emph{spherical} if $\trace_L = \trace_R$.

Let $\mathcal A$ be an abelian category and let $X\in \mathcal A$. Suppose $X \cong \bigoplus_{i\in I} X_i$ for a finite set $I$ and certain $X_i\in \mathcal A, i\in I$. Then there exists a family $(\iota_i, \pi_i)_{i\in I}$ of morphisms $\iota_i: X_i \to X, \pi_i: X \to X_i$ that satisfies
\[ \pi_i \circ \iota_j = \delta_{i, j}\, \ident_{X_i}\ , \qquad \sum_{i\in I} \iota_i \circ \pi_i = \ident_X \qquad (i, j\in I). \]
We call such morphisms $(\iota_i, \pi_i)_{i\in I}$ \emph{inclusions and projections for $X$}. Let additionally $Y\in \mathcal A$ with inclusions and projections $(\iota'_j, \pi'_j)_{j\in J}$. Then every morphism $f: X \to Y$ is uniquely determined by the morphisms $\pi'_j \circ f \circ \iota_i$ for $i\in I, j\in J$. Moreover, if $F, G: \mathcal A\to \mathcal B$ are additive functors, every morphism $f: FX \to GY$ is uniquely determined by the morphisms $G\pi'_j \circ f \circ F\iota_i$ for $i\in I, j\in J$. We will frequently make use of these facts when defining morphisms.

The following definitions are taken from \cite{egno}:

Let $\mathcal A_i, i\in I$ be additive categories. Their direct sum $\mathcal A = \bigoplus_{i\in I} \mathcal A_i$ is the additive category whose objects are direct sums
\[ \bigoplus_{i\in I} X_i, \qquad (X_i \in \mathcal A_i \text{ almost all } 0), \]
and whose morphisms are
\[ \Hom_{\mathcal A}\left(\bigoplus_{i\in I} X_i, \bigoplus_{i\in I} Y_i\right) := \bigoplus_{i\in I} \Hom_{\mathcal A_i}(X_i, Y_i). \]
$\mathcal A$ is abelian if and only if all $\mathcal A_i$ is abelian for all $i\in I$.

An additive category $\mathcal A$ is called $\mathbb F$-\emph{linear} if $\Hom_{\mathcal A}(X, Y)$ is equipped with a structure of a vector space over $\mathbb F$ for all $X, Y\in \mathcal A$ such that composition is $\mathbb F$-bilinear. A functor $F: \mathcal A \to \mathcal B$ between $\mathbb F$-linear categories is called \emph{$\mathbb F$-linear} if it induces $\mathbb F$-linear maps $\Hom_{\mathcal A}(X, Y) \to \Hom_{\mathcal B}(FX, FY)$ for all $X, Y\in\mathcal A$. An $\mathbb F$-linear abelian category $\mathcal A$ is called \emph{locally finite} over $\mathbb F$ if
\begin{itemize}
	\item the $\mathbb F$-vector space $\Hom_{\mathcal A}(X, Y)$ is finite dimensional for all $X, Y\in \mathcal A$ and
	\item every object of $\mathcal A$ is of finite length.
\end{itemize}
A locally finite category $\mathcal A$ is called \emph{finite} over $\mathbb F$ if, additionally,
\begin{itemize}
	\item every simple object of $\mathcal A$ has a projective cover and
	\item there are only finitely many isomorphism classes of simple objects.
\end{itemize}

An abelian category $\mathcal A$ is called \emph{semisimple} if every object of $\mathcal A$ is semisimple.
Let $\mathcal C$ be a rigid monoidal category which is locally finite over $\mathbb F$. Then $\mathcal C$ is called a \emph{multitensor category} if the functor $\otimes: \mathcal C \times \mathcal C \to \mathcal C$ is bilinear on morphisms. It is called a \emph{tensor category} if, additionally, $e$ is simple.
A \emph{multifusion category} is a multitensor category that is also finite and semisimple. If $e$ is simple, it is called a \emph{fusion category}.

\begin{lemma}[Schur's Lemma] \label{lem:schur}
	\cite[Proposition 1.8.4]{egno}
	Let $\mathcal A$ be a locally finite category over $\mathbb F$ and let $x, y\in \mathcal A$ be simple objects. Then
	\[ \Hom_{\mathcal A}(x, y) \cong \left\{ \begin{array}{ll} 0, & x\neq y \\ \mathbb F, & x = y. \end{array} \right. \]
\end{lemma}

Note that Schur' Lemma only holds when $\mathbb F$ is algebraically closed.

If a tensor category $\mathcal C$ is equipped with a pivotal structure, $\End_{\mathcal C}(e)\cong \mathbb F$ by Schur's Lemma and hence the traces can be interpreted as elements of $\mathbb F$.

\begin{corollary} \label{cor:trace_mult}
	Let $\mathcal C$ be a pivotal category which is locally finite over $\mathbb F$ and let $x\in\mathcal C$ be simple. Then, for every morphism $h: x\to x$, we have
	\[ \trace_{L/R}(h)\, \ident_x = \dim_{L/R}(x)\, h. \]
\end{corollary}
\begin{proof}
	Since $h=\lambda\, \ident_x$ for some $\lambda\in\mathbb F$, this follows from the linearity of the trace.
\end{proof}

\begin{corollary} \label{cor:mor_matr}
	Let $\mathcal C$ be a semisimple category which is locally finite over $\mathbb F$ and let $x\in \mathcal C$ be a simple object. For $n, m\in \mathbb N$, we have an isomorphism of $\mathbb F$-vector spaces
	\[ \Hom_{\mathcal C}(x^{\oplus n}, x^{\oplus m}) \cong \Mat(m\times n, \mathbb F). \]
\end{corollary}
\begin{proof}
	Choose inclusions and projections
	\begin{align*} &\iota_i: x \to x^{\oplus n}, \quad \pi_i: x^{\oplus n} \to x \qquad (i=1, \dots, n) \\
		&\iota'_j: x \to x^{\oplus m}, \quad \pi'_j: x^{\oplus m} \to x \qquad (j=1, \dots, m)
	\end{align*}
	for $x^{\oplus n}$ and $x^{\oplus m}$, respectively. Every morphism $f:x^{\oplus n} \to x^{\oplus m}$ is uniquely determined by the morphisms
	\[ \pi'_i \circ f \circ \iota_j \in \Hom_{\mathcal C}(x, x) \cong \mathbb F \qquad (i=1, \dots, m, j=1, \dots, n). \]
	Define a matrix $A\in \Mat(m\times n, \mathbb F)$ by setting its entry in the $i$-th row and $j$-th column to
	$A^i_j := \pi'_i \circ f \circ \iota_j$.
	It is easy to check that this construction defines an $\mathbb F$-linear isomorphism $\Hom_{\mathcal C}(x^{\oplus n}, x^{\oplus m}) \cong \Mat(m\times n, \mathbb F)$.
\end{proof}

\subsection{The Deligne product of abelian categories}

\begin{definition} \label{def:deli_prod}
	\cite[Definition 1]{franco13}
	Let $\mathcal C$ and $\mathcal D$ be locally finite categories over $\mathbb F$. Their \emph{Deligne product} is an $\mathbb F$-linear abelian category $\mathcal C \boxtimes \mathcal D$ together with an $\mathbb F$-bilinear functor $\boxtimes: \mathcal C \times \mathcal D \to \mathcal C \boxtimes \mathcal D$ which is right exact in each variable and induces an equivalence of categories
	\[ \mathrm{Rex}(\mathcal C, \mathcal D; \mathcal E) \cong \mathrm{Rex}(\mathcal C \boxtimes \mathcal D; \mathcal E) \]
	for every $\mathbb F$-linear abelian category $\mathcal E$. Here $\mathrm{Rex}(\mathcal C_1, \dots, \mathcal C_n; \mathcal E)$ denotes the category of all $\mathbb F$-multilinear functors $\mathcal C_1 \times \dots \times \mathcal C_n \to \mathcal E$ that are right exact in each variable.
\end{definition}

Note that the Deligne product is unique up to unique equivalence of categories \cite[Proposition 1.11.2 (ii)]{egno}.

We recall the following facts about the Deligne product:
\begin{enumerate}
	\item The Deligne product of locally finite categories over $\mathbb F$ exists and is a locally finite category over $\mathbb F$ \cite[Proposition 1.11.2 (i)]{egno}.
	\item If $\mathcal C$ and $\mathcal D$ are multitensor categories, so is $\mathcal C \boxtimes \mathcal D$ \cite[Corollary 4.6.2]{egno}.
\end{enumerate}

\subsection{Properties of semisimple categories}
\newcommand{\simples}{\mathcal S(I_{\mathcal A})}

Let $\mathcal A$ be a semisimple category which is locally finite over $\mathbb F$ and let $\mathcal B$ be an $\mathbb F$-linear category. It is well-known that $\mathbb F$-linear functors $F: \mathcal A\to \mathcal B$ and natural transformations between such functors are determined by their behaviour on simple objects. Our aim in this section is to state this fact more precisely.

For an additive category $\mathcal A$ let $I_{\mathcal A}$ be a set of representatives for the simple objects of $\mathcal A$ and denote by $\simples$ the full subcategory of $\mathcal A$ with objects $I_{\mathcal A}$.

\newcommand{\dirSum}{\bigoplus_{j\in J_X} y_j^{(X)}}

\begin{lemma} \label{lem:fun_simp}
	Let $\mathcal A$ and $\mathcal B$ be $\mathbb F$-linear categories and suppose $\mathcal A$ is semisimple and locally finite over $\mathbb F$. Then every family of objects $(B_x)_{x\in I_{\mathcal A}}$ in $\mathcal B$ defines an $\mathbb F$-linear functor $F: \mathcal A \to \mathcal B$ (unique up to natural isomorphism) with $F(x) = B_x$ for all $x\in I_{\mathcal A}$. Conversely, every $\mathbb F$-linear functor $\mathcal A \to \mathcal B$ is naturally isomorphic to a functor of this form.
\end{lemma}
\begin{proof}
	\begin{enumerate}
	\item Uniqueness: Let $F, G: \mathcal A \to \mathcal B$ be $\mathbb F$-linear functors with $F(x) = G(x) = B_x$ for all $x\in I_{\mathcal A}$. We prove that $F$ and $G$ are naturally isomorphic. As $\mathcal A$ is semisimple, every object $X \in \mathcal A$ admits a decomposition $X \cong \dirSum$ into simple objects $y_j^{(X)}\in I_{\mathcal A}$. Let
	\[ \iota_j^{(X)}: y_j^{(X)} \to X, \quad \pi_j^{(X)}: X \to y_j^{(X)} \qquad (j\in J_X) \]
	be inclusions and projections for $X$.
	Define a morphism $\alpha_X: F(X) \to G(X)$ by
	\[ G\pi_j^{(X)} \circ \alpha_X \circ F\iota_k^{(X)} = \delta_{j, k}\, \ident_{B_{y_k^{(X)}}}:\ B_{y_k^{(X)}} = F(y_k^{(X)})\, \to\, G(y_j^{(X)}) = B_{y_j^{(X)}}. \]
	It is straightforward to check that $\alpha_X$ is an isomorphism.
	It remains to show that this defines a natural transformation $\alpha: F\Rightarrow G$, i.e., that the diagram
	\[ \begin{tikzcd}
		FX \ar[r, "\alpha_X"] \ar[d, "Fh"] & GX \ar[d, "Gh"] \\
		FY \ar[r, "\alpha_Y"] & GY
	\end{tikzcd} \]
	commutes for every morphism $h: X \to Y$ in $\mathcal A$ or, equivalently, that
	\begin{equation} \label{eq:simp_fun_target} G\pi_j^{Y} \circ \alpha_Y \circ Fh \circ F\iota_k^{(X)} = G\pi_j^{Y} \circ Gh \circ \alpha_X \circ F\iota_k^{(X)} \end{equation}
	for all $j\in J_X, k\in J_Y$.
	By the definition of $\alpha$, (\ref{eq:simp_fun_target}) is equivalent to
	\begin{equation} \label{eq:simp_fun_target_2} F(\pi_j^{(Y)} \circ h \circ \iota_k^{(X)}) = G(\pi_j^{(Y)} \circ h \circ \iota_k^{(X)}). \end{equation}
	As $\pi_j^{(Y)} \circ h \circ \iota_k^{(X)}: y_k^{(X)} \to y_j^{(Y)}$ is a morphism between simple objects from $\mathcal A$, by Schur's Lemma \ref{lem:schur} it is either $0$ or a multiple of the identity. The functors $F$ and $G$ are $\mathbb F$-linear, so they agree on $\pi_j^{(Y)} \circ h \circ \iota_k^{(X)}$. Hence $F$ and $G$ are naturally isomorphic.
	
	This also implies that every functor $F: \mathcal A\to \mathcal B$ can be obtained, up to natural isomorphism, from the construction described in the remainder of the proof: Choose $B_x:= F(x)$ for $x\in I_{\mathcal A}$ and let $G: \mathcal A \to \mathcal B$ be the functor with $G(x) = B_x$ constructed below. Then $F\cong G$ by the above.
	
	\item Existence: Let $(B_x)_{x\in I_{\mathcal A}}$ be a family of objects in $\mathcal B$. Define a functor $F:\mathcal A \to \mathcal B$ on objects $X\in\mathcal A$ by
	\[ F(X) := \bigoplus_{j\in J_X} B_{y_j^{(X)}} = \bigoplus_{j\in J_X} F(y_j^{(X)}) \]
	and on morphisms as follows:
	\begin{enumerate}
		\item Any morphism $h: x\to y$ between simple objects $x, y\in I_{\mathcal A}$ is either $0$ or a multiple of the identity. Set $F(0) := 0$ and $F(\lambda \cdot \ident_x) := \lambda \cdot \ident_{Fx}$.
		\item For a morphism $h: X \to Y$ between arbitrary objects $X, Y\in \mathcal A$, define $F(h)$ by
		\[ F\pi_j^{(Y)} \circ F(h) \circ F\iota_k^{(X)} = F(\pi^{(Y)}_j \circ h \circ \iota^{(X)}_k). \]
		The morphism $F(\pi^{(Y)}_j \circ h \circ \iota^{(X)}_k)$ is already defined because $\pi^{(Y)}_j \circ h \circ \iota^{(X)}_k:\ y_k^{(X)} \to y_j^{(Y)}$ is a morphism between objects from $I_{\mathcal A}$.
	\end{enumerate}

	It is straightforward to check that $F$ is an $\mathbb F$-linear functor. \qedhere
	\end{enumerate}
\end{proof}

\begin{lemma} \label{lem:nat_simp}
	Let $\mathcal A$ and $\mathcal B$ be additive categories, suppose $\mathcal A$ is semisimple and let $F, G: \mathcal A \to \mathcal B$ be additive functors. Then restriction of natural transformations is a bijection
	\[ \mathrm{Nat}(F, G) \cong \mathrm{Nat}(F|_{\simples}, G|_{\simples}). \]
	A natural transformation $\eta: F\Rightarrow G$ is a natural isomorphism if and only if its image under the bijection is a natural isomorphism.
\end{lemma}
\begin{proof}
	First we show that any natural transformation $\eta: F \Rightarrow G$ is determined by the family $(\eta_{x})_{x\in I_\mathcal A}$. Every $X \in \mathcal A$ admits a decomposition $X \cong \bigoplus_{j\in J} x_j$ with $x_j\in I_{\mathcal A}$ for $j\in J$, as well as inclusions and projections $(\iota_j, \pi_j)_j$. Due to the naturality of $\eta$,
	\[\begin{tikzcd}
		F(X) \ar[r, "\eta_X"] & G(X) \\
		F(x_j) \ar[u, "F\iota_j"] \ar[r, "\eta_{x_j}"] & G(x_j) \ar[u, "G\iota_j"]
	\end{tikzcd}\]
	commutes for all $j\in J$, and thus 
	\[ G\pi_j \circ \eta_X \circ F\iota_k = G\pi_j \circ G\iota_k \circ \eta_{x_j} = \delta_{j, k}\, \eta_{x_j}. \]
	Hence $\eta_X$ is determined by the family $(\eta_{x})_{x\in I_\mathcal A}$.
	
	It remains to show that any natural transformation $\eta: F|_{\simples} \Rightarrow G|_{\simples}$ can be extended to a natural transformation $\tilde{\eta}: F \Rightarrow G$. Let $X\in\mathcal A$ be arbitrary with decomposition as well as inclusions and projections as above. We define $\tilde \eta_X$ by
	\[ G\pi_j \circ \tilde\eta_X \circ F\iota_k = \delta_{j, k}\, \eta_{x_j}:\ F(x_k) \to G(x_j) \qquad (j, k\in J). \]
	
	We now show that $\tilde{\eta}$ is a natural transformation. Let $h: X \to Y$ be an arbitrary $\mathcal A$-morphism. There exists a decomposition $Y \cong \bigoplus_{k\in K} y_k$ of $Y$ into objects $y_k\in I_{\mathcal A}$. Let $(\iota'_k, \pi'_k)_k$ denote inclusions and projections for $Y$. Then we have
	\[ G\pi'_k \circ Gh \circ \tilde{\eta}_X \circ F\iota_j = G(\pi'_k \circ h \circ \iota_j) \circ \eta_{x_j} \overset{(*)} = \eta_{y_k} \circ F(\pi'_k \circ h \circ \iota_j) = G\pi'_k \circ \tilde{\eta}_Y \circ Fh \circ F\iota_j \]
	for all $j\in J, k\in K$, where we used the naturality of $\eta: F|_{\simples} \Rightarrow G|_{\simples}$ in $(*)$. It follows that
	\[ Gh \circ \tilde{\eta}_X = \tilde{\eta}_Y \circ Fh. \]
	
	If $\eta$ is a natural isomorphism, it is straightforward to check that $\tilde\eta_X$ as defined above is an isomorphism.
	
\end{proof}

\begin{corollary} \label{cor:nat_simp2}
	Let $\mathcal A$ and $\mathcal B$ be $\mathbb F$-linear categories.
	Suppose $\mathcal A$ is locally finite and semisimple and let $F, G: \mathcal A \to \mathcal B$ be $\mathbb F$-linear functors.
	Then every family of morphisms
	\[ \eta_x: Fx \to Gx, \quad x\in I_{\mathcal A}, \]
	extends uniquely to a natural transformation $\eta: F \Rightarrow G$.
\end{corollary}
\begin{proof}
	By Lemma \ref{lem:nat_simp}, every natural transformation $\eta: F\Rightarrow G$ is uniquely determined by the component morphisms $\eta_x: Fx \to Gx,\, x\in I_{\mathcal A}$. But every family of morphisms
	\[ \eta_x: Fx \to Gx, \quad x\in I_{\mathcal A}, \]
	is a natural transformation $\eta: F|_{\simples} \Rightarrow G|_{\simples}$, as can be verified easily using Schur's Lemma \ref{lem:schur}.
\end{proof}

We now prove the well-known fact that every object in a semisimple, abelian, locally finite category is projective, see e.g. \cite[Section 7.5]{egno}.

\begin{lemma}\label{lem:ssimp_epi_split}
	Let $\mathcal A$ be a semisimple category which is locally finite over $\mathbb F$. Then every epimorphism in $\mathcal A$ splits.
\end{lemma}
\begin{proof}
	Let $f:X\to Y$ be an epimorphism in $\mathcal A$. Assume w.l.o.g. that $X = x^{\oplus n}, Y = x^{\oplus m}$ for some simple object $x\in \mathcal A$. This is justified by the following:
	\begin{enumerate} \item Every morphism $f: X\to Y$ in $\mathcal A$ can be decomposed into a direct sum of certain morphisms \mbox{$f_x: x^{\oplus n_x} \to x^{\oplus m_x}$} with $x\in \mathcal A$ simple and $n_x, m_x\in \mathbb N$,
		\item $f_1 \oplus \ldots \oplus f_n$ is an epimorphism iff $f_i$ is an epimorphism for $i=1, \dots, n$.
	\end{enumerate} 
	By Corollary \ref{cor:mor_matr}, morphisms $X \to Y$ are in bijection with $m\times n$-matrices with entries in $\mathbb F$. Write $A\in\mathrm{Mat}(m\times n, \mathbb F)$ for the matrix associated to $f$. Since $f$ is an epimorphism, $A$ has full rank and $m\leq n$, so $A$ has the right inverse
	$B:=A^T (AA^T)^{-1}$. The morphism $g: Y \to X$ corresponding to $B$ is right inverse to $f$, so $f$ splits.
\end{proof}

\begin{lemma} \label{lem:proj_split}
	\cite[Theorem I.4.7]{hs97}
	Let $\mathcal A$ be an abelian category. Then an object $X\in\mathcal A$ is projective if and only if every epimorphism $A \to X$ in $\mathcal A$ splits.
\end{lemma}

\subsection{Graphical calculus for spherical fusion categories}
\label{ssec:graph_fusion}
We use the graphical calculus for spherical fusion categories as described in \cite[Chapter XIV]{kassel}, see also \cite[Appendix A]{schaumann}.

Let $\mathcal C$ be a monoidal category. The diagrams introduced in the following each represent a single morphism in $\mathcal C$. For any object $C\in \mathcal C$, the identity on $C$ is represented by a vertical line labelled by $C$.
A morphism $f:C \to D$ in $\mathcal C$ can be represented by a circle labelled by $f$. All lines  labelled with the tensor unit $e$ are omitted.
Composition of morphisms is represented by vertical composition of diagrams from top to bottom, the tensor product of morphisms by horizontal composition of diagrams from left to right. The position of brackets is suppressed for both operations. Identity morphisms are included in the diagram in such a way that the domain and codomain of every morphism are clear from the diagram:

\begin{center}
\begin{tikzpicture}
	
	\draw[very thick] (0, 0) -- node[right=1pt, near start] {$C$} (0, -1) node[morphism, label=0:{$f$}, label=180:{$f\ \eqdiag \ $}] {} -- node[right=1pt, near end] {$D$} (0, -2);
\end{tikzpicture}
, \qquad
\begin{tikzpicture}
	\draw[very thick] (3, 0) -- node[right=1pt, near start] {$C$} (3, -1) node[morphism, label=0:{$\ident_C$}] {} -- node[right=1pt, near end] {$C$} (3, -2);
\end{tikzpicture}
$= \ \ $
\begin{tikzpicture}
	\draw[very thick] (5, 0) -- node[right=1pt, very near start] {$C$} (5, -2);	
\end{tikzpicture}
\end{center}

Here, the symbol $\eqdiag$ indicates that a diagram represents a certain algebraic expression. Lines labelled with $e$ are omitted in the diagrams, as well as the associativity constraints and unit constraints of $\mathcal C$, which is justified by the coherence theorem.
If $\mathcal C$ is right rigid or left rigid, we represent the evaluations and coevaluations by the following diagrams:
\begin{center}
\begin{tabular}{c|c|c|c}
	$\eval_C^R$ & $\coev_C^R$ & $\eval_C^L$ & $\coev_C^L$ \\
	\hline
	\tikz{
		\draw[very thick] (0, 0) node[right=1pt] {$C$} .. controls (0, -1) and (-1, -1) .. (-1, 0) node[left=1pt] {$C^*$};
	}
	& \tikz{
		\draw[very thick] (0, 0) node[left=1pt] {$C$} .. controls (0, 1) and (1, 1) .. (1, 0) node[right=1pt] {$C^*$};
	}
	& \tikz{
		\draw[very thick] (0, 0) node[left=1pt] {$C$} .. controls (0, -1) and (1, -1) .. (1, 0) node[right=1pt] {$\prescript{*}{}{C}$};
	}
	& \tikz{
		\draw[very thick] (0, 0) node[right=1pt] {$C$} .. controls (0, 1) and (-1, 1) .. (-1, 0) node[left=1pt] {$\prescript{*}{}{C}$};
	}
\end{tabular}
\end{center}
If $\mathcal C$ is pivotal, all vertical lines in the diagram are oriented and a line labelled with $C$ represents $\ident_C$ if oriented downwards and $\ident_{C^*}$ if oriented upwards. If the orientation of a line is not indicated in the diagram, it is oriented downward by default. We do not distinguish $C$ and $C^{**}$ and omit the pivotal structure in diagrams. The evaluations and coevaluations then read
\begin{center}
	\begin{tabular}{c|c|c|c}
		$\eval_C^R$ & $\coev_C^R$ & $\eval_C^L$ & $\coev_C^L$ \\
		\hline
		\tikz{
			\draw[very thick, ->] (0, 0) node[right=1pt] {$C$} .. controls (0, -1) and (-1, -1) .. (-1, 0) node[left=1pt] {$C$};
		}
		& \tikz{
			\draw[very thick, <-] (0, 0) node[left=1pt] {$C$} .. controls (0, 1) and (1, 1) .. (1, 0) node[right=1pt] {$C$};
		}
		& \tikz{
			\draw[very thick, ->] (0, 0) node[left=1pt] {$C$} .. controls (0, -1) and (1, -1) .. (1, 0) node[right=1pt] {$C$};
		}
		& \tikz{
			\draw[very thick, <-] (0, 0) node[right=1pt] {$C$.} .. controls (0, 1) and (-1, 1) .. (-1, 0) node[left=1pt] {$C$};
		}
	\end{tabular}
\end{center}

A pivotal category $\mathcal C$ is spherical if and only if for all morphisms $f:C\to C$ in $\mathcal C$

\begin{center}
	\begin{tikzpicture}
		\draw[very thick, ->] (-1,0) node[left=1pt] {$C$} -- (-1, -1) node[morphism, label=0:{$f$}, label=180:{$\trace_R(f)\ \eqdiag \ $}] {} -- (-1, -2) .. controls (-1, -2.5) and (1, -2.5) .. (1, -2) -- (1, 0) .. controls (1, 0.5) and (-1, 0.5) .. (-1, 0);
	\end{tikzpicture}
	$\ =\ $
	\begin{tikzpicture}
		\draw[very thick, ->] (4, 0) node[right=1pt] {$C$} -- 	(4, -1) node[morphism, label=180:{$f$}, label=0:{$\ \eqdiag\ \trace_L(f).$}] {} -- (4, -2) .. controls (4, -2.5) and (2, -2.5) .. (2, -2) -- (2, 0) .. controls (2, 0.5) and (4, 0.5) .. (4, 0);
	\end{tikzpicture}
\end{center}
\clearpage
\section{The tensor category $\vectgo$} \label{sec:vecgo}
We now introduce the main example of a spherical fusion category used in this thesis, the category $\vectgo$ of $G$-graded vector spaces.
First we define $\vectgo$ and prove that it is right rigid, following \cite[Section 2.3]{egno}. We then classify the pivotal structures on $\vectgo$ and determine which of these turn $\vectgo$ into a spherical category.

\begin{definition}
	Let $X$ be a set. The \emph{category of $X$-graded vector spaces} $\mathrm{Vec}_X$ has as objects finite dimensional $\mathbb F$-vector spaces $V$ with a grading $V = \bigoplus_{x\in X} V_x$ and as morphisms $\mathbb F$-linear maps $f:V\to W$ that respect the grading: $f(V_x) \subset W_x$ for all $x\in X$.
\end{definition}

It is straightforward to check that $\mathrm{Vec}_X$ is a locally finite category over $\mathbb F$ and that the simple objects of $\mathrm{Vec}_X$ are of the form $\delta^x$, $x\in X$, with grading
\[ \delta^x = \bigoplus_{y\in X} \delta_y^x, \qquad \delta_y^x = \left\{ \begin{array}{ll} \mathbb F, & x=y \\ 0, & x\neq y. \end{array} \right. \]

Furthermore, $\mathrm{Vec}_X$ is semisimple. Lemma \ref{lem:ssimp_epi_split} and Lemma \ref{lem:proj_split} imply that all objects of $\mathrm{Vec}_X$ are projective. Thus $\mathrm{Vec}_X$ is a finite category over $\mathbb F$ if and only if the set $X$ is finite.

If an object $\delta^x$ appears in a formula as an index, e.g. for a component morphism of a natural transformation, we usually replace it by $x$ to improve readability.

\subsection{Monoidal structures and rigidity}
We now replace $X$ with a group $G$. Define the functor $\otimes: \vectg \times \vectg \to \vectg$ as the tensor product of vector spaces with gradings
\[ (V \otimes W)_g = \bigoplus_{\substack{h, k\in G \\ hk=g}} V_h \otimes W_k. \]
In particular, $\delta^g \otimes \delta^h = \delta^{gh}$. Now assume $a: \otimes (\otimes \times \Ident) \Rightarrow \otimes (\Ident \times \otimes)$ is a natural isomorphism that makes $\vectg$ a monoidal category. Then, by Corollary \ref{cor:nat_simp2}, $a$ is determined by the collection of linear maps 
\[ a_{g, h, k}: (\delta^g \otimes \delta^h) \otimes \delta^k \to \delta^g \otimes (\delta^h \otimes \delta^k) \]
indexed by $g, h, k\in G$, which are non-zero multiples of the identity on $\delta^{ghk}$. Write $\omega(g, h, k)\in\units, g, h, k\in G$ for the scalars with $a_{g, h, k} = \omega(g, h, k)\, \ident_{ghk}$. The pentagon axiom for monoidal categories translates to the condition
\begin{equation} \label{eq:3cocycle}
	\omega(h, k, l)\, \omega(g, hk, l)\, \omega(g, h, k) = \omega(gh, k, l)\, \omega(g, h, kl), 
\end{equation}
for any quadruple $(g, h, k, l) \in G^{\times 4}$, making $\omega$ a 3-cocycle of $G$ with coefficients in $\units$.
We will refer to (\ref{eq:3cocycle}) as the \enquote{cocycle condition for $(g, h, k, l)$} subsequently.
We write $\vectgo$ for the monoidal category $(\mathrm{Vec}_G, \otimes, a)$ with $a_{g, h, k} = \omega(g, h, k)\, \ident_{ghk}$.

The tensor unit of $\vectgo$ is $\delta^1$, where $1\in G$ is the neutral element of $G$. The associativity constraint determines the left and right unit constraints, which are of the form
\[ l_g = \rho_l (g)\, \ident_g: \delta^1 \otimes \delta^g \to \delta^g, \quad r_g = \rho_r (g)\, \ident_g: \delta^g \otimes \delta^1 \to \delta^g \]
for certain maps
\[ \rho_l: G\to \units, \quad \rho_r: G\to \units. \]
The cocycle condition (\ref{eq:3cocycle}) for $(g, 1, 1, h)$ implies $\omega(g, 1, h) = \omega(g, 1, 1)\, \omega(1, 1, h)$. Together with the triangle axiom from the definition of a monoidal category, this implies
\[ \rho_l(g) = \omega^{-1}(1, 1, g), \quad \rho_r(g) = \omega(g, 1, 1). \]
It can be shown that the monoidal categories $\mathrm{Vec}_G^{\omega_G}$ and $\mathrm{Vec}_H^{\omega_H}$ are related by an $\mathbb F$-linear monoidal equivalence iff there exists an isomorphism $f: G \to H$ and $\omega_G$ is cohomologous to $\omega_H \circ (f\times f \times f)$ \cite[Section 2.6]{egno}.
Since every cocycle is cohomologous to a normalised one by Lemma \ref{lem:norm_cochain}, it may be assumed without loss of generality that $\omega$ is normalised. We will make this assumption for all $3$-cocycles $\omega$ and categories $\vectgo$ in the following. Also note that this condition makes the left and right constraints trivial. We therefore omit them subsequently.

Define the dual of a simple object $\delta^g$ as $(\delta^g)^* = \delta^{g^{-1}}$. The right evaluation and coevaluation
\[ \eval_g^R =: \zeta_g\, \ident_1: \delta^{g^{-1}}\otimes \delta^g \to \delta^1, \quad \coev_g^R =: \zeta'_g\, \ident_1: \delta^1 \to \delta^g \otimes \delta^{g^{-1}} \quad (\zeta_g, \zeta'_g \in\mathbb F)\]
need to satisfy the snake identities
\begin{align*}
	(\ident_g \otimes \eval_g^R) \circ a_{g, g^{-1}, g} \circ (\coev_g^R \otimes \ident_g) &= \ident_{g} \\
	(\eval_g^R \otimes \ident_{g^{-1}}) \circ a^{-1}_{g^{-1}, g, g^{-1}} \circ (\ident_{g^{-1}} \otimes \coev_g^R) &= \ident_{g^{-1}}, \\
\end{align*}
which translate to the conditions
\begin{equation} \label{eq:rig_temp} \zeta_g\, \omega(g, g^{-1}, g)\, \zeta'_g = 1 = \zeta_g\, \omega^{-1}(g^{-1}, g, g^{-1})\, \zeta'_g. \end{equation}
The cocycle condition (\ref{eq:3cocycle}) for $(g^{-1}, g, g^{-1}, g)$ implies
\begin{equation} \label{eq:3co_inv} \omega(g, g^{-1}, g)\, \omega(g^{-1}, g, g^{-1}) = 1 \end{equation}
for a normalized cocycle $\omega$ and hence (\ref{eq:rig_temp}) is equivalent to $\zeta_g\, \zeta'_g = \omega(g^{-1}, g, g^{-1})$. Thus any choice of $(\zeta'_g)_{g\in G}$ determines right evaluation and coevaluation uniquely. We choose $\zeta'_g = 1$ for all $g\in G$ to simplify subsequent calculations. This yields
\begin{align*}
	\eval_g^R &= \omega(g^{-1}, g, g^{-1}) \, \ident_1 \\
	\coev_g^R &=
	\ident_1.
\end{align*}

It follows that all objects in $\vectgo$ have right duals: If $X$ and $Y$ have right duals, $X\oplus Y$ has the right dual $X^*\oplus Y^*$, with evaluation and coevaluation given by

\[\begin{tikzcd}[row sep=tiny]
	\eval_{X\oplus Y}^R = && (X^* \oplus Y^*) \otimes (X\oplus Y) \ar[r, "\pi_{X^*}\otimes \pi_X"] & [1.5em] X^* \otimes X \ar[r, "\eval_X^R"] & e \\
	& \hskip-5em + \hskip-5em & (X^* \oplus Y^*) \otimes (X\oplus Y) \ar[r, "\pi_{Y^*}\otimes \pi_Y"] & Y^* \otimes Y \ar[r, "\eval_Y^R"] & e,
\end{tikzcd}\]
\[\begin{tikzcd}[row sep=tiny]
	\coev_{X\oplus Y}^R = && e \ar[r, "\coev_X^R"] & X\otimes X^* \ar[r, "\iota_X \otimes \iota_{X^*}"] &[1.5em] (X \oplus Y) \otimes (X^* \oplus Y^*) \\
	& \hskip-5em + \hskip-5em & e \ar[r, "\coev_Y^R"] & Y\otimes Y^* \ar[r, "\iota_Y \otimes \iota_{Y^*}"] &[1.5em] (X \oplus Y) \otimes (X^* \oplus Y^*),
\end{tikzcd}\]
where $+$ denotes the sum of morphisms in an additive category and $\iota_X, \iota_Y, \pi_X, \pi_Y$ are inclusions and projections for $X\oplus Y$.
A direct calculation confirms that this choice of evaluation and coevaluation satisfies the snake identities.

\subsection{Pivotal structures on $\vectgo$} \label{ssec:vecgo_piv}
We now determine all pivotal structures on $\vectgo$. 
It has already been noted in \cite[Example 3.13]{schaumann} that the pivotal structures on $\vectgo$ are in bijection with characters $\kappa: G\to \units$. As this bijection is rather subtle, we give the required calculation in detail.

By Corollary \ref{cor:nat_simp2}, every natural isomorphism $b: \Ident \Rightarrow **$ is determined by a map $\beta: G \to \units$ in the sense that
\[ b_g = \beta(g)\, \ident_g:\ \delta^g \to \delta^g = (\delta^g)^{**} \qquad (g\in G). \]

The natural transformation $b$ is monoidal if and only if the diagram
\begin{equation} \label{diag:piv_**} \begin{tikzcd}[column sep=3.5cm]
	\delta^g \otimes \delta^h \ar[r, "b_g \otimes b_h = \beta(g)\, \beta(h)\, \ident_{gh}"] \ar[d, "\ident_{gh}"] &[2em] \delta^g \otimes \delta^h = (\delta^g)^{**} \otimes (\delta^h)^{**} \ar[d, "\phi^{\otimes}_{g, h}"] \\
	\delta^g \otimes \delta^h \ar[r, "b_{gh} = \beta(gh)\, \ident_{gh}"] & \delta^g \otimes \delta^h = (\delta^g \otimes \delta^h)^{**}
\end{tikzcd} \end{equation}
commutes for all $g, h\in G$, where $\phi^{\otimes}: \otimes\, (** \times **) \Rightarrow ** \otimes$ denotes the coherence datum of the monoidal functor $**$.

We now calculate the coherence datum of $**$ in general, that is, for an arbitrary right rigid monoidal category $\mathcal C$. For simplicity, we assume that $\mathcal C$ has trivial left and right unit constraints.
To determine the coherence datum of $**$, one needs the coherence datum of the monoidal functor
\[ *: \mathcal C \to \mathcal C^{op, rev}. \]

This coherence datum $\psi^{\otimes}_{X, Y}: (X\otimes Y)^* \to Y^*\otimes X^*$ is defined graphically in \cite[Proposition 2.1.5]{tenscat}
(see also \cite[Proposition XIV.2.2]{kassel}) as

\begin{center}
\begin{tikzpicture}
	
	\node[morphism, label=180:{$\ident_{X\otimes Y}$}, minimum size=1cm] (I) at (0, 0) {};
	
	\draw[thick] (I.south) .. controls +(0, -1) and +(0, -1) .. node[right=2pt, near start] {$X\otimes Y$} ++(-2, 0) -- node[left=1pt, near end] {$(X\otimes Y)^*$} ++(0, 3);
	
	\draw[thick] (I.north east) .. controls +(0, 1) and +(0, 1) .. node[right=1pt, very near start] {$Y$} ++(1.5, 0) -- node[right=1pt, very near end] {$Y^*$} ++(0, -2);
	\draw[thick] (I.north west) .. controls +(0, 2) and +(0, 2) .. node[left=1pt, very near start] {$X$} ++(3.5, 0) -- node[right=1pt, very near end] {$X^*$} ++(0, -2);
\end{tikzpicture}
\end{center}
which translates to the algebraic expression
\begin{multline} \label{eq:coh_*}
	\psi^{\otimes}_{X, Y} := 
	\left( \eval_{X\otimes Y}^R \otimes \ident_{Y^* \otimes X^*} \right) \circ a^{-1}_{(X\otimes Y)^*, X\otimes Y, (Y^* \otimes X^*)} \circ \left( \ident_{(X\otimes Y)^*} \otimes a_{X\otimes Y, Y^*, X^*} \right) \\
	\circ \left(\ident_{(X\otimes Y)^*} \otimes a_{X, Y, Y^*}^{-1} \otimes \ident_{X^*} \right)
	\circ \left( \ident_{(X\otimes Y)^*\otimes X} \otimes \coev_Y^R \otimes \ident_{X^*} \right) \circ \left(\ident_{(X\otimes Y)^*} \otimes \coev_X^R\right).
\end{multline}
Using the graphical calculus, it is easy to check that $(*, \psi^{\otimes})$ is a monoidal functor.

\begin{lemma} \label{lem:coh_**}
	Let $\psi^{\otimes}$ denote the coherence datum of $*:\mathcal C \to \mathcal C^{op, rev}$. Then $**: \mathcal C \to \mathcal C$ is a monoidal functor with coherence datum
	\[ \begin{tikzcd}[column sep=large]
		\phi^{\otimes}_{X, Y} \text{ : } X^{**} \otimes Y^{**} \ar[r, "\left( \psi_{Y^*, X^*}^{\otimes} \right)^{-1}"] & (Y^* \otimes X^*)^* \ar[r, "\left( \psi_{X, Y}^{\otimes} \right)^*"] & (X\otimes Y)^{**}.
	\end{tikzcd}\]
\end{lemma}
\begin{proof}
	Let $\mathcal C, \mathcal D, \mathcal E$ be arbitrary monoidal categories.
	\begin{enumerate}
		\item \label{it:op_mon} Given any monoidal functor $(F, \phi^{\otimes}): \mathcal C \to \mathcal D^{op}$, the functor $F^{op}: \mathcal C^{op} \to \mathcal D$ is monoidal with coherence datum
		\[ \tilde{\phi}^{\otimes}_{X, Y} = \left(\phi^{\otimes}_{X, Y}\right)^{-1}:\ FX \otimes FY \to F(X\otimes Y). \]
		\item \label{it:rev_mon} Given any monoidal functor $(F, \phi^{\otimes}): \mathcal C\to \mathcal D^{rev}$, the functor $F^{rev}: \mathcal C^{rev} \to \mathcal D$ (identical to $F$ as functor) is monoidal with coherence datum
		\[ \tilde{\phi}^{\otimes}_{X, Y} = \phi^{\otimes}_{Y, X}:\ FX \otimes FY\to F(Y \otimes X). \]
		\item \label{it:comp_mon}
		Given monoidal functors $(F, \phi^{\otimes}): \mathcal C\to \mathcal D, (G, \rho^{\otimes}): \mathcal D\to \mathcal E$, their composite $GF$ is a monoidal functor with coherence datum
		\[ \begin{tikzcd}
			\tilde\phi^{\otimes}_{X, Y} =\
			GFX \otimes GFY \ar[r, "\rho^{\otimes}_{FX, FY}"] & G(FX \otimes FY) \ar[r, "G\phi^{\otimes}_{X, Y}"] & GF(X\otimes Y).
		\end{tikzcd} \]
	\end{enumerate}
	It is straightforward to check that these constructions yield monoidal functors. \ref{it:op_mon}. and \ref{it:rev_mon}. imply that $(*)^{op, rev}: \mathcal C^{op, rev}\to \mathcal C$ is a monoidal functor with coherence datum $\tilde\psi_{X, Y}^{\otimes} = (\psi_{Y, X}^{\otimes})^{-1}$.
	Now apply \ref{it:comp_mon}. to $F=*$ and $G=(*)^{op, rev}$.
\end{proof}

We now return to the case $\mathcal C = \vectgo$.
\begin{lemma} \label{lem:vecgo_piv}
	Pivotal structures on $\vectgo$ are in bijection with characters $\kappa: G\to \units$ of $G$. If $\kappa: G \to \units$ is a character,
	\[ \beta(g) := \kappa(g)\, \omega^{-1}(g, g^{-1}, g) \]
	defines a pivotal structure $b$ on $\vectgo$ with $b_g = \beta(g)\, \ident_g$.
\end{lemma}
\begin{proof}
The coherence datum of $*$ can be determined by setting $X=\delta^g, Y=\delta^h$ in (\ref{eq:coh_*}):
\[
	\psi^{\otimes}_{g, h} = \omega(gh, h^{-1}, g^{-1})\, \omega^{-1}(g, h, h^{-1})\ \ident_{h^{-1}g^{-1}}.
\]
Hence, by Lemma \ref{lem:coh_**}, the coherence datum of $**$ is
\begin{multline*} 
	\phi^{\otimes}_{g, h} = \left( \psi_{g, h}^{\otimes} \right)^* \circ \left( \psi_{h^{-1}, g^{-1}}^{\otimes}\right)^{-1} = \\ 
	\omega(gh, h^{-1}, g^{-1}) \, \omega^{-1}(g, h, h^{-1}) \, \omega^{-1}(h^{-1}g^{-1}, g, h) \, \omega(h^{-1}, g^{-1}, g)\ \ident_{gh}.
\end{multline*}
We apply the cocycle condition (\ref{eq:3cocycle}) for $(g, h, h^{-1}, g^{-1}), (h^{-1}g^{-1}, g, h, h^{-1}g^{-1}), (h^{-1}, h, h^{-1}, g^{-1})$ and $(h^{-1}, g^{-1}, g, g^{-1})$ and get
\[ \phi^{\otimes}_{g, h} = \omega^{-1}(h^{-1}, h, h^{-1}) \, \omega^{-1}(g^{-1}, g, g^{-1}) \, \omega(h^{-1}g^{-1}, gh, h^{-1}g^{-1})\ \ident_{gh}. \]
By (\ref{eq:3co_inv}), this is equivalent to
\[ \phi^{\otimes}_{g, h} = \omega(h, h^{-1}, h) \, \omega(g, g^{-1}, g) \, \omega^{-1}(gh, h^{-1}g^{-1}, gh)\ \ident_{gh}. \]
We insert this formula for $\phi^{\otimes}_{g, h}$ into (\ref{diag:piv_**}) and conclude that $b_g = \beta(g)\, \ident_g$ for $g\in G$ defines a pivotal structure on $\vectgo$ if and only if
\[  \beta(g)\, \beta(h)\, \omega(g, g^{-1}, g)\, \omega(h, h^{-1}, h) = \beta(gh)\, \omega(gh, h^{-1}g^{-1}, gh). \]
This holds if and only if the map $\kappa: G\to \units$ defined by
\[ \kappa(g) := \beta(g)\, \omega(g, g^{-1}, g) \]
is a character of $G$. This concludes the proof.
\end{proof}

We write $\vectgok:= (\vectgo, b)$ with $b_g=\kappa(g)\, \omega^{-1}(g, g^{-1}, g)\, \ident_g$ for the pivotal category determined by a character $\kappa:G\to \units$.
We can now calculate the left evaluation and coevaluation of $\mathrm{Vec}_G^{\omega, \kappa}$ as defined by (\ref{eq:left_eval}) and (\ref{eq:left_coev}).
\begin{eqformarray}{lcrlcr} \label{eq:vecgo_ev}
	\eval_g^R &=& \omega(g^{-1}, g, g^{-1}) \, \ident_1,
	\hspace*{2em} & \eval_g^L &=& \kappa(g)\, \ident_1 \\
	\coev_g^R &=&
	\ident_1,
	\hspace*{2em} &\coev_g^L &=& \kappa^{-1}(g)\, \omega(g, g^{-1}, g) \, \ident_1
\end{eqformarray}

In order to determine whether $\vectgok$ is spherical, we calculate the dimensions of simple objects in $\vectgo$:
\[ \dim_R (\delta^g) = \kappa(g), \qquad \dim_L (\delta^g) = \kappa^{-1}(g) \qquad (g\in G), \]
where (\ref{eq:3co_inv}) was used to calculate $\dim_L(\delta^g)$. The fact that $\kappa$ is a character translates to the compatibility of dimensions and tensor products from \cite[Proposition 4.7.3]{egno}:
\[ \dim_{L/R} (X \otimes Y) = \dim_{L/R} (X) \cdot \dim_{L/R} (Y). \]

\begin{lemma} \label{lem:vecgok_sph}
	$\vectgok$ is spherical if and only if $\kappa(g) \in \{ 1, -1\}$ for all $g\in G$.
\end{lemma}
\begin{proof}
	If $\vectgok$ is spherical, then
	\[ \kappa(g) = \dim_R(\delta^g) = \dim_L(\delta^g) = \kappa^{-1}(g) \]
	so $\kappa(g)\in \{ 1, -1\}$ for all $g\in G$.
	
	Now suppose $\kappa(g) \in \{ 1, -1\}$ for all $g\in G$. Then $\dim_L(\delta^g) = \dim_R(\delta^g)$ for all $g\in G$. Let $\varphi: V \to V$ be a $\vectgo$-morphism, and let $V\cong \bigoplus_{i\in I} \delta^{g_i}$ be the composition of $V$ into simple objects. Let $(\iota_i, \pi_i)_{i\in I}$ be inclusions and projections for $V$. Then
	\begin{align*} 
		\trace_{L/R} (\varphi) &= \sum_{i\in I} \trace_{L/R} (\iota_i \circ \pi_i \circ \varphi) = \sum_{i\in I} \trace_{L/R} (\pi_i \circ \varphi \circ \iota_i) \overset{\ref{cor:trace_mult}} = \sum_{i\in I} \dim_{L/R} (\delta^{g_i})\ \pi_i \circ \varphi \circ \iota_i,
	\end{align*}

	where we used Corollary \ref{cor:trace_mult} as well as cyclicity and linearity of the traces. In the last step, we identified $\Hom(\delta^{g_i}, \delta^{g_i})$ and $\mathbb F$. Since the left and right dimensions of $\delta^{g_i}$ agree for all $i$, it follows that $\trace_L (\varphi) = \trace_R (\varphi)$.
\end{proof}

Characters $\kappa: G\to \{1, -1\}$ are uniquely determined by their kernel. Therefore, the spherical structures on $\vectgo$ are classified as follows: The category $\vectgo$ always admits the trivial spherical structure given by $\kappa= 1$. The non-trivial spherical structures on $\vectgo$ are in bijection with normal subgroups $N\trianglelefteq G$ of index $2$.

We summarize the results of this section: $\vectgo$ is a semisimple tensor category over $\mathbb F$ and spherical structures on $\vectgo$ correspond to characters $\kappa: G \to \{\pm 1\}$ which encode the dimensions of simple objects. If $G$ is finite, $\vectgo$ is a fusion category.
\clearpage
\section{Module categories}\label{sec:modc}
So far, we have been working exclusively with $\mathbb F$-linear monoidal categories, which are a categorification of $\mathbb F$-algebras. In a similar fashion, one can categorify the notion of modules and bimodules over such algebras. This leads to the concept of module categories and bimodule categories over multitensor categories.

Throughout this section let $(\mathcal C, \otimes, a)$ and $(\mathcal D, \otimes', a')$ be multitensor categories over $\mathbb F$ as defined in Section \ref{ssec:tens_cat}.

\subsection{Definition and properties}\label{ssec:modc}
This section recalls basic definitions and results from \cite[chapter 7]{egno}.

\begin{definition}
	A \emph{$\mathcal C$-(left) module category} or \emph{(left) module category over $\mathcal C$} is a triple $(\mathcal M, \opl, m)$ consisting of
	\begin{itemize}
		\item a category $\mathcal M$ which is locally finite over $\mathbb F$,
		\item a functor $\opl: \mathcal C \times \mathcal M \to \mathcal M$ that is $\mathbb F$-bilinear and exact in the first argument, sometimes referred to as \emph{action functor}, and
		\item a natural isomorphism $m: \opl (\otimes \times \Ident) \Rightarrow \opl (\Ident \times \opl)$ called \emph{(left) module constraint}
	\end{itemize}
	such that the functor $e\opl -: \mathcal M \to \mathcal M$ is an equivalence and the \emph{module pentagon axiom} holds, i.e., the diagram
	\begin{equation} \label{diag:modc_pent} \begin{tikzcd}[column sep=large]
		((X \otimes Y) \otimes Z) \opl M \ar[r, "m_{X\otimes Y, Z, M}"] \ar[d, "a_{X, Y, Z}\opl \ident_M", swap] & (X\otimes Y) \opl (Z\opl M) \ar[r, "m_{X, Y, Z\opl M}"] & X \opl (Y \opl (Z\opl M)) \\
		(X\otimes (Y\otimes Z)) \opl M \ar[r, "m_{X, Y\otimes Z, M}", swap] & X\opl ((Y\otimes Z) \opl M) \ar[ru, "\ident_X \opl m_{Y, Z, M}", swap]
	\end{tikzcd} \end{equation}
	commutes for all $X, Y, Z\in \mathcal C, M\in \mathcal M$.
\end{definition}

Likewise, one defines a $\mathcal C$-right module category with action functor $\opr: \mathcal M \times \mathcal C \to \mathcal M$ and right module constraint $n: \opr (\Ident \times \otimes) \Rightarrow \opr (\opr \times \Ident)$ that satisfies an analogous pentagon axiom. Alternatively, one may define a $\mathcal C$-right module category as a $\mathcal C^{rev}$-left module category.

Every $\mathcal C$-module category admits a unique natural isomorphism
\[ l_M: e \opl M \to M \qquad (M\in \mathcal M), \]
called the \emph{unit constraint}, such that the diagram
\begin{equation} \label{diag:mod_triangle} \begin{tikzcd}
		(X\otimes e) \opl M \ar[rr, "m_{X, e, M}"] \ar[rd, "r_X \opl \ident_M", swap] && X \opl (e \opl M) \ar[ld, "\ident_X \opl l_M"] \\ & X\opl M
\end{tikzcd} \end{equation}
commutes for all $X\in \mathcal C, M\in\mathcal M$. In fact, $l_M$ is uniquely determined by
\[ \begin{tikzcd}
	\ident_e \opl l_M = \ e\opl (e\opl M) \ar[r, "m^{-1}_{e, e, M}"] & (e\otimes e)\opl M \ar[r, "r_e\opl \ident_M"] & e\opl M.
\end{tikzcd} \]
By diagram chasing one shows that $l_M$ is natural in $M$ and that (\ref{diag:mod_triangle}) commutes.

\begin{remark}
	Often, we will only consider semisimple module categories. In \cite{egno}, much of the theory is instead developed for so-called \emph{exact} module categories, see \cite[Definition 7.5.1]{egno}. Since all semisimple module categories are exact \cite[Section 7.5]{egno}, these results apply to semisimple module categories.
\end{remark}

The following lemma allows us to ignore the condition that $\opl$ needs to be exact in the first argument in the case that $\mathcal C$ is a semisimple multitensor category (for instance, $\mathcal C=\vectgo$). The lemma is a special case of \cite[Proposition 7.6.9]{egno}.

\begin{lemma} \label{lem:exactness}
	Let $\mathcal A$ and $\mathcal B$ be abelian categories and suppose that $\mathcal A$ is semisimple and locally finite over $\mathbb F$. Then every additive functor $F: \mathcal A \to \mathcal B$ is exact.
\end{lemma}
\begin{proof}
	By Lemma \ref{lem:ssimp_epi_split}, every epimorphism in $\mathcal A$ splits. Hence every exact sequence in $\mathcal A$ is of the form
	\[ \begin{tikzcd}
		0 \ar[r] & X \ar[r] & X\oplus Y \ar[r] & Y \ar[r] & 0.
	\end{tikzcd} \]
	The functor $F$ is additive, so $F(X\oplus Y) \cong FX \oplus FY$ and the sequence
	\[ \begin{tikzcd}
		0 \ar[r] & FX \ar[r] & F(X\oplus Y) \ar[r] & FY \ar[r] & 0
	\end{tikzcd} \]
	is exact.
\end{proof}

For $\mathbb F$-algebras $A$ and $B$, an $(A, B)$-bimodule is an abelian group $M$ with the structure of both an $A$-left module and a $B$-right module subject to the condition
\begin{equation} \label{eq:bimod_mot} (a\opl m) \opr b = a \opl (m \opr b) \qquad (a\in A, b\in B, m\in M). \end{equation}
If we replace algebras by multitensor categories, modules by module categories and the condition (\ref{eq:bimod_mot}) by a natural isomorphism with suitable coherence properties, we obtain the definition of a bimodule category:

\begin{definition} \label{def:bimodc}
	A \emph{$(\mathcal C, \mathcal D)$-bimodule category} is a tuple $(\mathcal M, \opl, \opr, m, n, b)$, where
	\begin{itemize}
		\item $(\mathcal M, \opl, m)$ is a $\mathcal C$-left module category,
		\item $(\mathcal M, \opr, n)$ is a $\mathcal D$-right module category and
		\item $b: \opr (\opl \times \Ident) \Rightarrow \opl (\Ident \times \opr)$ is a natural isomorphism, called the \emph{middle module constraint},
	\end{itemize}
	such that the diagrams
	\[ \begin{tikzcd}[column sep=large]
		((C \otimes C') \opl M) \opr D \ar[r, "b_{C\otimes C', M, D}"] \ar[d, "m_{C, C', M} \opr \ident_D", swap] & (C\otimes C') \opl (M \opr D) \ar[r, "m_{C, C', M\opr D}"] & C \opl (C' \opl (M \opr D)) \\
		(C \opl (C' \opl M)) \opr D \ar[r, "b_{C, C'\opl M, D}"] & C \opl ((C' \opl M) \opr D) \ar[ru, "\ident_C \opl b_{C', M, D}", swap]
	\end{tikzcd} \]
	and
	\[ \begin{tikzcd}[column sep=large]
		(C \opl M) \opr (D\otimes' D') \ar[r, "b_{C, M, D\otimes D'}"] \ar[d, "n_{C\opl M, D, D'}", swap] & C \opl (M \opr (D\otimes' D')) \ar[r, "\ident_C \opl n_{M, D, D'}"] & C \opl ((M \opr D) \opr D') \\
		((C\opl M) \opr D) \opr D' \ar[r, "b_{C, M, D}\opr \ident_{D'}"] & (C \opl (M\opr D)) \opr D' \ar[ru, "b_{C, M\opr D, D'}", swap]
	\end{tikzcd} \]
	commute for all objects $C, C' \in \mathcal C, D, D'\in \mathcal D, M\in\mathcal M$.
\end{definition}

\begin{remark} \label{rem:bimod_deligne}
	$(\mathcal C, \mathcal D)$-bimodule categories can be identified with left module categories over the Deligne product $\mathcal C \boxtimes \mathcal D^{rev}$ \cite[Exercise 7.4.3]{egno}. We will therefore focus on the theory of left module categories in the following.
\end{remark}

\begin{example} \label{ex:c_bimod_cat}
	Every multitensor category $\mathcal C$ is a $(\mathcal C, \mathcal C)$-bimodule category with $\opl = \opr = \otimes$ and 
	\[ m_{X, Y, Z} = b_{X, Y, Z} = a_{X, Y, Z},\qquad n_{X, Y, Z}=a_{X, Y, Z}^{-1}. \]
	Note that $\otimes$ is biexact by \cite[Proposition 4.2.1]{egno}.
	All four pentagon axioms for a bimodule category reduce to the pentagon axiom for a monoidal category in this case.
\end{example}

Morphisms in a bimodule category can be represented by diagrams similar to those introduced in Section \ref{ssec:graph_fusion}. We briefly summarize this graphical calculus which is described in \cite{bm_tv+}, see also \cite[Section 3.1]{schaumann}.

Let $\mathcal M$ be a $(\mathcal C, \mathcal D)$-bimodule category. The identity morphisms $\ident_M$ in $\mathcal M$ are represented by coloured vertical lines labelled by $M$. Morphisms $f:M\to M'$ can be represented by a coloured circle labelled by $f$. If $g: C\to C'$ is a morphism in $\mathcal C$, it is represented by a (black) diagram as described in Section \ref{ssec:graph_fusion} and the morphism $g \opl f$ is represented by the horizontal composite of the diagrams for $g$ and $f$. Similarly, a morphism $h: D \to D'$ in $\mathcal D$ is represented by a grey diagram and the morphism $f \opr h$ is represented by horizontal composition of the diagrams for $f$ and $h$. Brackets are suppressed and the natural transformations $m, n$ and $b$ are omitted from the diagrams. This is justified by an analogue of the coherence theorem for bimodule categories \cite[Remark 7.2.4]{egno}. For instance, the diagram

\begin{center}
\begin{tikzpicture}
	\draw[very thick] (0, 0) node[left=1pt] {$C$} -- (0, -1) node[morphism, label=180:{$g$}] {} -- (0, -2) node[left=1pt] {$C'$};
	\draw[very thick, color=blue] (1.5, 0) node[left=1pt] {$M$} -- (1.5, -1) node[morphism, label=180:$f$] {} -- (1.5, -2) node[left=1pt] {$M'$};
	\draw[very thick, color=black!60] (2.5, 0) node[right=1pt] {$D$} -- (2.5, -1) node[morphism, label=0:{$h$}] {} -- (2.5, -2) node[right=1pt] {$D'$};
\end{tikzpicture}
\end{center}

may represent any of the four morphisms
\[ \psi \circ ((g\opl f)\opr h) \circ \phi \qquad \text{ for } \psi\in \left\{ \ident_{(C\opl M)\opr D}, b_{C', M', D'}\right\}, \phi\in\left\{\ident_{(C\opl M)\opr D}, b_{C, M, D}^{-1}\right\}. \]
As before, composition of morphisms is represented by vertical composition of diagrams from top to bottom.

\begin{example}
	(\cite[Remark 7.1.5]{egno}, see also \cite{bm_tv+})
	Let $\mathcal M$ be a $(\mathcal C, \mathcal D)$-bimodule category. Let $\mathcal M^{\#}$ be the dual category of $\mathcal M$ and equip $\mathcal M^{\#}$ with the structure of a $(\mathcal D, \mathcal C)$-bimodule category as follows:
	\[ D \opl^{\#} M := M \opr D^*, \qquad M \opr^{\#} C := C^* \opl M \]
	for $M\in \mathcal M, C\in\mathcal C, D\in\mathcal D$. The module associativity constraints are
	\begin{formarray}{lll}
		m^{\#}_{D, D', M} &:= (\ident_M \opr \psi^{\otimes}_{D', D}) \circ n^{-1}_{M, (D')^*, D^*} &: (M \opr (D')^*) \opr D^*
		\to M \opr (D \otimes D')^* \\
		n^{\#}_{M, C, C'} &:= (\psi^{\otimes}_{C', C} \opl \ident_M) \circ m^{-1}_{(C')^*, C^*, M} &: (C')^* \opl (C^* \opl M)
		\to (C \otimes C')^* \opl M \\
		b^{\#}_{D, M, C} &:= b_{C^*, M, D^*} &: (C^* \opl M) \opr D^* \to C^* \opl (M \opr D^*),
	\end{formarray}
	\hspace{-0.5em}
	where $\psi^{\otimes}$ denotes the coherence datum of the functor $*$ as in Section \ref{ssec:vecgo_piv}. $\mathcal M^{\#}$ is called the \emph{opposite bimodule category}.
\end{example}

The following lemma will occasionally be helpful for constructing morphisms in module categories. It is a generalisation of the well-known adjunction
\[ \Hom_{\mathcal C}(X^* \otimes Y, Z) \cong \Hom_{\mathcal C}(Y, X\otimes Z) \]
which holds for all rigid categories $\mathcal C$ and all objects $X, Y, Z\in \mathcal C$ \cite[Proposition 2.10.8]{egno}.

\begin{lemma} \label{lem:*opl_adj}
	\cite[Proposition 7.1.6]{egno}
	Let $\mathcal M$ be a $\mathcal C$-module category. There exists a natural isomorphism
	\[ \Hom_{\mathcal M}(C^* \opl M, N) \cong \Hom_{\mathcal M}(M, C\opl N) \]
	for $M, N\in\mathcal M, C\in\mathcal C$, so the functor $C^* \opl -: \mathcal M \to \mathcal M$ is left adjoint to $C \opl -$.
\end{lemma}
\begin{proof}
	Let $f: C^* \opl M\to N$ be a $\mathcal M$-morphism. Define $f^{\#}: M \to C\opl N$ as
	\[ \begin{tikzcd}[column sep=large]
		M \ar[r, "\coev_C^R\opl \ident_M"] & (C\otimes C^*) \opl M \ar[r, "m_{C, C^*, M}"] & C \opl (C^* \opl M) \ar[r, "\ident_C \opl f"] & C \opl N.
	\end{tikzcd} \]
	Using the graphical calculus, it is straightforward to check that mapping $f: M \to C\opl N$ to
	\[ \begin{tikzcd}[column sep=large]
		C^* \opl M \ar[r, "\ident_{C^*} \opl f"] & C^* \opl (C \opl N) \ar[r, "m_{C^*, C, N}^{-1}"] & (C^* \otimes C) \opl N \ar[r, "\eval_C^R \opl \ident_N"] & N
	\end{tikzcd} \]
	defines an inverse to $f \mapsto f^{\#}$ and that $f \mapsto f^{\#}$ is a natural transformation.
\end{proof}

In the theory of monoidal categories, monoidal functors and monoidal natural transformations take a place similar to that of functors and natural transformation in general category theory. In particular, one is only interested in classifying monoidal categories up to monoidal equivalences. 
We now introduce \emph{module functors}, which are the analogue of monoidal functors in the theory of module categories.

\begin{definition}
	\begin{enumerate}
	\item Let $(\mathcal M, \opl, m)$ and $(\mathcal N, \opl', m')$ be $\mathcal C$-module categories. A $\mathcal C$-module functor $(F, s): \mathcal M \to \mathcal N$ consists of a left exact functor $F: \mathcal M \to \mathcal N$ and a natural isomorphism $s: F\opl \Rightarrow \opl' F$ such that the \emph{module functor pentagon axiom} is satisfied, i.e., the diagram
	\begin{equation} \label{diag:modf} \begin{tikzcd}
		F((X \otimes Y) \opl M) \ar[r, "s_{X\otimes Y, M}"] \ar[d, "Fm_{X, Y, M}", swap] & (X\otimes Y) \opl' F(M) \ar[r, "m'_{X, Y, F(M)}"] & X \opl' (Y \opl' F(M)) \\
		F(X\opl (Y\opl M)) \ar[r, "s_{X, Y\opl M}"] & X\opl' F(Y\opl M) \ar[ru, "\ident_X\opl' s_{Y, M}", swap]
	\end{tikzcd} \end{equation}
	commutes for all objects $X, Y\in \mathcal C, M\in \mathcal M$.
	\item Let $(\mathcal M, \opr, n)$ and $(\mathcal N, \opr', n')$ be $\mathcal C$-right module categories. A $\mathcal C$-right module functor $(F, t): \mathcal M\to \mathcal N$ consists of a left exact functor $F: \mathcal M\to \mathcal N$ and a natural isomorphism $t:F\opr \Rightarrow \opr' F$ such that the \emph{right module functor pentagon axiom} holds, i.e., the diagram
	\[ \begin{tikzcd}
	F(M \opr (X\otimes Y)) \ar[r, "t_{M, X\otimes Y}"] \ar[d, "Fn_{M, X, Y}", swap] & F(M) \opr' (X\otimes Y) \ar[r, "n'_{F(M), X, Y}"] & (F(M) \opr' X) \opr' Y \\
	F((M \opr X) \opr Y) \ar[r, "t_{M\opr X, Y}"] & F(M\opr X) \opr' Y \ar[ru, "t_{M, X} \opr' \ident_Y", swap]
	\end{tikzcd} \]
	commutes for all objects $X, Y\in \mathcal C, M\in \mathcal M$.
	\item Let $(\mathcal M, \opl, \opr, m, n, b)$ and $(\mathcal N, \opl', \opr', m', n', b')$ be $(\mathcal C, \mathcal D)$-bimodule categories. A $(\mathcal C, \mathcal D)$-bimodule functor $(F, s, t): \mathcal M \to \mathcal N$ is a triple such that $(F, s)$ is a $\mathcal C$-left module functor, $(F, t)$ is a $\mathcal D$-right module functor and the \emph{bimodule functor hexagon axiom} is satisfied, i.e., the diagram
	\[ \begin{tikzcd}
		F((C\opl M)\opr D) \ar[r, "t_{C\opl M, D}"] \ar[d, "Fb_{C, M, D}", swap] & F(C \opl M) \opr' D \ar[r, "s_{C, M} \opr' \ident_D"] & (C\opl' F(M)) \opr' D \ar[d, "b'_{C, F(M), D}"] \\
		F(C\opl (M\opr D)) \ar[r, "s_{C, M\opr D}"] & C\opl' F(M\opr D) \ar[r, "\ident_C\opl' t_{M, D}"] & C \opl' (F(M) \opr' D)
	\end{tikzcd} \]
	commutes for all $C\in \mathcal C, D\in\mathcal D, M\in \mathcal M$.
	\end{enumerate}

	A (bi-) module functor is called an \emph{equivalence} of $\mathcal C$-module categories (or $\mathcal C$-right module categories or $(\mathcal C, \mathcal D)$-bimodule categories, respectively) if its underlying functor is an equivalence of categories.
\end{definition}

\begin{remark}\label{rem:modf_triangle}
	In \cite[Definition 7.2.1]{egno}, left module functors are required to satisfy the additional condition that
	\[ \begin{tikzcd}
		F(e \opl M) \ar[rr, "s_{e, M}"] \ar[rd, "F(l_M)", swap] && e\opl' F(M) \ar[ld, "l'_M"] \\ & F(M)
	\end{tikzcd} \]
	commutes for all $M\in\mathcal M$. But this already follows from (\ref{diag:modf}) by a diagram chase, where one uses that $\ident_e \opl l_M = (r_e \opl \ident_M) \circ m_{e, e, M}^{-1}$ by construction of $l$.
\end{remark}

\begin{remark}
	All module functors are $\mathbb F$-linear: Suppose $(F, s)$ is a $\mathcal C$-module functor, then for every $\lambda\in\mathbb F$ and every $\mathcal M$-morphism $f:X \to Y$, the diagram
	\[ \begin{tikzcd}[column sep=2.5cm]
		F(X) \ar[r, "F(\lambda f)"] \ar[d, "F(l_X)^{-1}"] \ar[ddd, "\ident_{F(X)}", swap, bend right=70] & F(Y) \ar[d, "F(l_Y)^{-1}"] \ar[ddd, "\ident_{F(Y)}", bend left=70] \\
		F(e\opl X) \ar[r, "F(\ident_e \opl (\lambda f))"] \ar[r, "=F((\lambda\, \ident_e)\opl f)", swap] \ar[d, "s_{e, X}"] & F(e\opl Y) \ar[d, "s_{e, Y}"] \\
		e\opl' F(X) \ar[r, "(\lambda\, \ident_e)\opl' F(f)"] \ar[r, "=\ident_e \opl' (\lambda\, F(f))", swap] \ar[d, "l'_X"] & e\opl' F(Y) \ar[d, "l'_Y"] \\
		F(X) \ar[r, "\lambda\, F(f)"] & F(Y)
	\end{tikzcd} \]
	commutes due to the commutative triangle from Remark \ref{rem:modf_triangle}.
\end{remark}

Following \cite{bm_tv+}, we expand the graphical calculus for module categories to include module functors.

Let $\mathcal M$ and $\mathcal N$ be $(\mathcal C, \mathcal D)$-bimodule categories and let $(F, s): \mathcal M \to \mathcal N$ be a $\mathcal C$-left module functor. For a morphism $f:M\to M'$ in $\mathcal M$, the morphism $Ff: FM \to FM'$ is represented by the diagram for $f$ together with a coloured dashed line to the right labelled by $F$:
\begin{center}
	$Ff \ \eqdiag \ $ 
	\begin{tikzpicture}
		\draw[very thick, dashed, color=red] (2, 0) -- node[near start, right=1pt] {$F$} (2, -2);
		\draw[very thick, color=blue] (1, 0) -- node[near start, left=1pt] {$M$} (1, -1) node[morphism, label=180:{$f$}] {} -- node[near end, left=1pt] {$M'$} (1, -2);
	\end{tikzpicture}
\end{center}
The coherence datum of $F$ is suppressed in diagrams.
If multiple module categories appear in the same diagram, we we use a different colour for each of them.
The action of a $\mathcal D$-right module functor $F: \mathcal M\to \mathcal N$ is represented by a coloured dashed line, labelled by $F$, to the left of the diagram for an $\mathcal M$-morphism.

\begin{example} \label{ex:modf}
	\cite{bm_tv+}
	Let $\mathcal M$ be a $(\mathcal C, \mathcal D)$-bimodule category. Recall that $\mathcal C$ is a $(\mathcal C, \mathcal C)$-bimodule category and $\mathcal D$ is a $(\mathcal D, \mathcal D)$-bimodule category by Example \ref{ex:c_bimod_cat}.
	\begin{enumerate}
		\item For $D\in\mathcal D$ the functor $- \opr D: \mathcal M \to \mathcal M$ is a $\mathcal C$-module functor with $b_{-, -, D}$ as coherence datum. The pentagon axiom for module functors reduces to the first pentagon axiom satisfied by $b$.
		
		\item Analogously, for $C\in\mathcal C$ the functor $C \opl -: \mathcal M \to \mathcal M$ is a $\mathcal D$-right module functor.
		\item For $M \in \mathcal M$ the functor $- \opl M: \mathcal C\to \mathcal M$ is a $\mathcal C$-module functor with coherence datum $m_{-, -, M}$. Likewise, $M \opr -: \mathcal D\to \mathcal M$ is a $\mathcal D$-right module functor.
	\end{enumerate}
\end{example}

Monoidal natural transformations also have an analogue in the theory of module categories:

\begin{definition}
	Let $(F, s^F), (G, s^G): \mathcal M \to \mathcal N$ be $\mathcal C$-module functors.
	A \emph{morphism of module functors} or \emph{$\mathcal C$-(left) module natural transformation} is a natural transformation $\eta: F \Rightarrow G$ such that the diagram
	\begin{equation} \label{diag:mor_modf} \begin{tikzcd}[column sep=large]
		F(C \opl M) \ar[r, "\eta_{C\opl M}"] \ar[d, "s^F_{C, M}", swap] & G(C \opl M) \ar[d, "s^G_{C, M}"] \\
		C\opl F(M) \ar[r, "\ident_C \opl \eta_M"] & C\opl G(M)
	\end{tikzcd} \end{equation}
	commutes for all $C\in \mathcal C, M\in \mathcal M$.
	
	Morphisms of $\mathcal C$-right module functors are defined analogously. A morphism of $(\mathcal C, \mathcal D)$-bimodule functors is a natural transformation that is both a morphism of $\mathcal C$-left module functors and a morphism of $\mathcal D$-right module functors.
\end{definition}

Like additive categories, module categories may be decomposed into direct sums.

\begin{proposition}
	Let $(\mathcal M_1, \opl_1, m^{(1)}), (\mathcal M_2, \opl_2, m^{(2)})$ be $\mathcal C$-module categories. Then their direct sum $\mathcal M_1 \oplus \mathcal M_2$ is a $\mathcal C$-module category with the following structures:
	\begin{formarray}{rl}
	X\opl (M_1 \oplus M_2) &:= (X\opl_1 M_1) \oplus (X\opl_2 M_2) \\
	m_{X, Y, M_1\oplus M_2} &:= m^{(1)}_{X, Y, M_1} \oplus m^{(2)}_{X, Y, M_2}
	\end{formarray}
\end{proposition}

\begin{definition}
	A module category $\mathcal M$ is called \emph{indecomposable} if it is not equivalent to the direct sum of two non-zero module categories $\mathcal M_1$ and $\mathcal M_2$.
\end{definition}

Analogous definitions can be given for right module categories and bimodule categories. It suffices to study indecomposable bimodule categories, but it is important to note that an indecomposable bimodule category may not be indecomposable as a left (or right) module category.

To define generalised 6j symbols for bimodule categories, we need a trace on module categories which generalises the trace on spherical categories. Such module traces where first introduced in \cite{schaumann}.

\begin{definition}
	\cite[Definition 3.7]{schaumann}
	Let $\mathcal C$ be a pivotal multifusion category over $\mathbb F$ and let $\mathcal M$ be a $\mathcal C$-module category. Then a \emph{$\mathcal C$-(left) module trace} $\theta$ on $\mathcal M$ is a collection of $\mathbb F$-linear maps
	\[ \theta_M: \End(M) \to \mathbb F \]
	with the following properties:
	\begin{enumerate}
		\item Cyclicity: Let $f: M\to M', g:M'\to M$ be $\mathcal M$-morphisms, then
		\[ \theta_M (g\circ f) = \theta_{M'} (f\circ g). \]
		\item Non-degeneracy: The pairing
		\[ \Hom_{\mathcal M}(M, M') \times \Hom_{\mathcal M}(M', M) \to \mathbb F: (f, g) \mapsto \theta_M(g\circ f) \]
		is non-degenerate for all $M, M'\in\mathcal M$.
		\item $\theta$ is $\mathcal C$-compatible: For all $X\in\mathcal C, M\in \mathcal M$ and $f\in\End(X\opl M)$ we have 
		\[ \theta_{X\opl M} (f) = \theta_M (\trace^{\mathcal C}_{X, M}(f)), \]
		where $\trace_{X, M}^{\mathcal C}$ denotes the \emph{partial trace}
		\[ \trace^{\mathcal C}_{X, M} (f) := (\eval_X^R\opl \ident_M)\circ m_{X^*, X, M}^{-1} \circ (\ident_{X^*}\opl f)\circ m_{X^*, X, M}\circ (\coev_X^L \opl \ident_M). \]
	\end{enumerate}
\end{definition}

A \emph{$\mathcal D$-right module trace} on a $\mathcal D$-right module category $\mathcal M$ is a $\mathcal D^{rev}$-left module trace on $\mathcal M$ as a $\mathcal D^{rev}$-module category.
A \emph{$(\mathcal C, \mathcal D$)-bimodule trace} on a $(\mathcal C, \mathcal D)$-bimodule category $\mathcal M$ is a $\mathcal C \boxtimes \mathcal D^{rev}$-module trace on $\mathcal M$ considered as a $\mathcal C \boxtimes \mathcal D^{rev}$-left module category.
If $\theta$ is a left module trace, right module trace or bimodule trace on $\mathcal M$, define the \emph{dimension} of an object $M\in \mathcal M$ as
$\dim^{\theta}(M) := \theta_M(\ident_M)$.

We represent the (bi-)module trace graphically as in \cite{bm_tv+}, see also \cite[Section 3.2]{schaumann}:
\begin{center}
\begin{tikzpicture}
	\draw[very thick, color=blue] (-0.2, 0) -- (0.2, 0);
	\draw[very thick, color=blue] (0, 0) -- node[near start, left=1pt] {$M$} (0, -1) node[morphism, label=180:{\textcolor{black}{$\theta_M(f)\ \eqdiag\ $} $f$}] {} -- node[near end, left=1pt] {$M$} (0, -2);
	\draw[very thick, color=blue] (-0.2, -2) -- (0.2, -2);
\end{tikzpicture}
\end{center}

The condition that a $\mathcal C$-module trace $\theta$ is $\mathcal C$-compatible is written diagrammatically as

\begin{center}
\begin{tikzpicture}
	\draw[very thick, color=blue] (-1.7, 0) -- (0.2, 0);
	\draw[very thick, color=blue] (0, 0) -- node[near start, right=1pt] {$M$} (0, -1.5) node[morphism, label=180:{$f$}] (f) {} -- node[near end, right=1pt] {$M$} (0, -3);
	\draw[very thick, color=blue] (-1.7, -3) -- (0.2, -3);
	\draw[very thick] (f.north west) -- node[near end, left=1pt] {$X$} (-1.5, 0);
	\draw[very thick] (f.south west) -- node[near end, left=1pt] {$X$} (-1.5, -3);
\end{tikzpicture}
$\ =\ $
\begin{tikzpicture}
	\draw[very thick, color=blue] (-0.2, 0) -- (0.2, 0);
	\draw[very thick, color=blue] (0, 0) -- node[near start, right=1pt] {$M$} (0, -1.5) node[morphism, label=0:{$f$}] (f) {} -- node[near end, right=1pt] {$M$} (0, -3);
	\draw[very thick, color=blue] (-0.2, -3) -- (0.2, -3);
	\draw[very thick, ->] (f.south west) .. controls +(0, -1) and +(0, -1) .. +(-1, 0) coordinate (c1) {};
	\draw[very thick] (f.north west) .. controls +(0, 1) and +(0, 1) .. +(-1, 0) -- node[left=1pt] {$X$} (c1);
\end{tikzpicture}
\end{center}
and an analogous diagram describes the $\mathcal D$-compatibility of a $\mathcal D$-right module trace.

Corollary \ref{cor:trace_mult} also holds for module traces since the proof only uses the linearity of the trace:

\begin{corollary} \label{cor:mod_tr_mult}
	Let $\mathcal M$ be a $\mathcal C$-module category and let $x\in\mathcal M$ be simple. Then, for every morphism $h: x\to x$, we have
	\[ \theta(h)\, \ident_x = \dim^{\theta}(x)\, h. \]
\end{corollary}

\subsection{Module categories over $\vectgo$} \label{ssec:vecgo_modc}
We now classify the semisimple module categories over $\vectgo$. This is a slight generalisation of \cite[Example 7.4.10]{egno}, where only the case $\omega=1$ is treated.

We proceed in three steps:
\begin{enumerate}
	\item \label{it:modxp} Construct a family of semisimple module categories over $\vectgo$.
	\item Prove that every semisimple module category over $\vectgo$ is equivalent to at least one of the module categories constructed.
	\item Determine which of the module categories from step \ref{it:modxp} are equivalent.
\end{enumerate}

\begin{definition}
Let $X$ be a $G$-set with action map $\opl: G\,\times\, X\to X$ and let $n\geq 0$. There is a canonical inclusion
\begin{equation} \label{eq:cochain_incl} \,\widetilde{}\,: C^n(G, \units) \to C^n(G, \Map(X, \units)): \eta \mapsto \widetilde{\eta},\quad \widetilde{\eta}(g_1, \dots, g_n, x) := \eta(g_1, \dots, g_n).
\end{equation}
Let $\Psi\in C^2(G, \Map(X, \Psi))$ be a 2-cochain satisfying $\dif \Psi = \widetilde{\omega}^{-1}$ or, equivalently,
\begin{equation} \label{eq:2cocycle} \Psi(h, k, g^{-1} \opl x)\, \Psi^{-1} (gh, k, x)\, \Psi(g, hk, x)\, \Psi^{-1}(g, h, x) = \omega^{-1}(g, h, k) 
\end{equation}
for all $g, h, k\in G, x\in X$. We will refer to (\ref{eq:2cocycle}) as the \emph{coboundary condition for $(g, h, k, x)$} in subsequent calculations. We will also suppress the inclusion homomorphism $\,\widetilde{}\,$ subsequently, writing $\dif \Psi = \omega^{-1}$ instead of $\dif \Psi = \widetilde{\omega}^{-1}$.

Define the $\vectgo$-module category $\mathcal M(X, \Psi)$ as the category of $X$-graded finite dimensional vector spaces over $\mathbb F$ with the $\vectgo$-module category structure described below. To distinguish objects in $\modc$ from objects in $\vectgo$, we identify simple objects in $\modc$ with the respective elements of $X$, i.e., we write $x$ instead of $\delta^x$ for $x\in X$.

The functor $\opl: \vectgo \times \modc \to \modc$ is defined by $\delta^g \opl x := g \opl x$ and the module constraint by
\[ m_{g, h, x} := \Psi(g, h, (gh)\opl x)\, \ident_{(gh) \opl x}:\ (\delta^g \otimes \delta^h) \opl x \to \delta^g \opl (\delta^h \opl x). \]
\end{definition}

The module pentagon axiom translates to (\ref{eq:2cocycle}), hence $\modc$ is indeed a $\vectgo$-module category. $\modc$ is a finite category over $\mathbb F$ iff the set $X$ is finite, as was shown in Section \ref{sec:vecgo}.

\begin{proposition}
	Every semisimple module category over $\vectgo$ is equivalent to $\modc$ for some $G$-set $X$ and some 2-cochain $\Psi$ with $\dif \Psi = \omega^{-1}$.
\end{proposition}
\begin{proof}
	Let $\mathcal M$ be an arbitrary semisimple module category over $\vectgo$ and denote by $X$ a set of representatives for the simple objects in $\mathcal M$. Let $\mathcal M'$ be an additive, full subcategory of $\mathcal M$ whose simple objects are the ones in $X$. Note that $\mathcal M'$ is not necessarily a module subcategory of $\mathcal M$ because it may not be closed under the action of $\vectgo$ on $\mathcal M$.
	
	1. We show that $\delta^g \opl x$ is simple for all $g\in G, x\in X$.
	
	Assume $\delta^g \opl x$ is a non-trivial direct sum: $\delta^g \opl x \cong M_1 \oplus M_2$ with $0\neq M_1, M_2\in\mathcal M$. Then
	\[ (\delta^{g^{-1}} \opl M_1) \oplus (\delta^{g^{-1}} \opl M_2) \cong \delta^{g^{-1}} \opl (\delta^g \opl x) \cong (\delta^{g^{-1}} \otimes \delta^g) \opl x = \delta^1 \opl x \cong x, \]
	and as $x$ is simple, either $\delta^{g^{-1}} \opl M_1 = 0$ or $\delta^{g^{-1}} \opl M_2=0$. But if $\delta^{g^{-1}} \opl M_1 = 0$,
	\[ M_1 \cong \delta^g \opl (\delta^{g^{-1}} \opl M_1) = 0, \]
	and likewise $\delta^{g^{-1}} \opl M_2=0$ implies $M_2=0$. Hence $\delta^g \opl x$ is simple.
	
	2. We define a $G$-action $\opl$ on $X$ by taking as $g \opl x\in X$ the unique object from $X$ that is isomorphic to $\delta^g \opl x$. It is easy to check that this turns $X$ into a $G$-set.

	3. We now equip $\mathcal M'$ with the structure of a $\vectgo$-module category.
	Define the action functor $\opl'$ on $\mathcal M'$ by $\delta^g \opl' x := g \opl x$ for $g\in G, x\in X$, where $\opl$ denotes the $G$-action on $X$. Fix a family of isomorphisms
	\[ \alpha_{g, x}: \delta^g \opl' x \to \delta^g \opl x \qquad (g\in G, x\in X). \]
	By Corollary \ref{cor:nat_simp2}, these define a natural isomorphism $\alpha: \opl' \to \opl$.
	The module constraint of $\mathcal M'$ is then given by
	\[ \begin{tikzcd}
		m'_{X, Y, M} \esccol (X\otimes Y) \opl' M \ar[r, "\alpha_{X\otimes Y, M}"] & (X\otimes Y)\opl M \ar[r, "m_{X, Y, M}"] & X\opl (Y\opl M) \ar[r, "\ident_X \opl \alpha_{Y, M}^{-1}"] & X\opl (Y\opl' M) \ar[d, "\alpha_{X, Y\opl' M}^{-1}"] \\ &&& X\opl' (Y\opl' M).
	\end{tikzcd} \]
	A diagram chase using the naturality of $\alpha$ and $m$ as well as the module pentagon axiom for $\mathcal M$ shows that $m'$ satisfies the module pentagon axiom. 
	
	4. Define $\Psi\in C^2(G, \Map(X, \units))$ by
	\[ m'_{g, h, x} = \Psi(g, h, (gh)\opl x)\, \ident_{(gh)\opl x}. \]
	Obviously, $\mathcal M'$ is equivalent to $\mathcal M(X, \Psi)$ as a module category. The inclusion functor $\mathcal M' \to \mathcal M$ is an equivalence and it becomes a module functor when equipped with coherence datum $\alpha$. Hence $\mathcal M$ is equivalent to $\modc$ as a module category.
\end{proof}

In order to decide which module categories $\mathcal M(X, \Psi)$ are equivalent, it is necessary to describe all module functors $F:\mathcal M(X, \Psi_X) \to \mathcal M(Y, \Psi_Y)$ that \emph{preserve simple objects}. A functor $F$ is said to preserve simple objects if $Fx$ is simple for all simple objects $x$. This class of functors includes all equivalences. General module functors $\mathcal M(X, \Psi_X) \to \mathcal M(Y, \Psi_Y)$ are analysed in Section \ref{sec:modf}.

We now construct a family of module functors $\mathcal M(X, \Psi_X) \to \mathcal M(Y, \Psi_Y)$ and then show that every module functor that preserves simple objects is isomorphic to one of these module functors.

Let $X$ and $Y$ be $G$-sets and let $\Psi_X\in C^2(G, \Map(X, \units)), \Psi_Y\in C^2(G, \Map(Y, \units))$ such that $\dif \Psi_X = \omega^{-1}$ and $\dif \Psi_Y = \omega^{-1}$. 

We construct a module functor $(F_{f, \Lambda}, s): \mathcal M(X, \Psi_X) \to \mathcal M(Y, \Psi_Y)$ for each $G$-equivariant map $f: X \to Y$ and each 1-cochain $\Lambda \in C^1(G, \Map(X, \units))$ with
\begin{equation} \label{eq:Lambda} \dif \Lambda(g, h, x) = \Lambda(h, g^{-1}\opl x)\, \Lambda^{-1}(gh, x) \, \Lambda(g, x) = \Psi_X^{-1}(g, h, x)\, \Psi_Y(g, h, f(x)) \end{equation}
for all $g, h\in G, x\in X$. For this, we define $F_{f, \Lambda}$ on simple objects $x\in X$ by
$F_{f, \Lambda}(x) := f(x)$
and the coherence datum as
\[ s_{g, x} := \Lambda(g, g \opl x)\, \ident_{g\opl f(x)}:\ F_{f, \Lambda}(\delta^g \opl x)\ \to\ \delta^g \opl F_{f, \Lambda}(x). \]
The module pentagon axiom translates to (\ref{eq:Lambda}), hence $F_{f, \Lambda}$ is a module functor. Clearly, $F_{f, \Lambda}$ preserves simple objects.

\begin{lemma} \label{lem:vecgo_coch_modf}
	Let $(F, s): \mathcal M(X, \Psi_X) \to \mathcal M(Y, \Psi_Y)$ be a module functor that preserves simple objects.
	Then there exists a pair $(f, \Lambda)$ as above such that $F$ and $F_{f, \Lambda}$ are isomorphic as module functors.
\end{lemma}
\begin{proof}
	Assume w.l.o.g. that $Fx \in Y$ for all $x\in X$. Define $f$ by $f(x) := Fx \in Y$, which is $G$-equivariant because
	\[ f(g \opl x) = F(\delta^g \opl x) \cong \delta^g \opl Fx = g \opl f(x) \]
	for all $g\in G, x\in X$. Then
	\[ s_{g, x}: \ g\opl f(x) = f(g\opl x) = F(\delta^g \opl x)\ \overset{\sim}{\longrightarrow}\ \delta^g \opl Fx = g \opl f(x) \]
	is a non-zero multiple of the identity, so we can define $\Lambda\in C^1(G, \Map(X, \units))$ such that $s_{g, x} = \Lambda(g, g\opl x) \, \ident_{g\opl f(x)}$. 
	The module pentagon axiom for $(F, s)$ implies that $\Lambda$ satisfies (\ref{eq:Lambda}) and, obviously, we have $F_{f, \Lambda} = F$.
\end{proof}

\begin{remark} \label{rem:vecgo_coch_modf}
	It is straightforward to check that $\vectgo$-module functors $F_{f_1, \Lambda_1}, F_{f_2, \Lambda_2}: \mathcal M(X, \Psi_X) \to \mathcal M(Y, \Psi_Y)$ are isomorphic as module functors iff $f_1 = f_2$ and the 1-cochain $\rho:= \Lambda_1\, \Lambda_2^{-1}\in C^1(G, \Map(X, \units))$
	is a coboundary.
\end{remark}

$F_{f, \Lambda}$ is an equivalence iff $f$ is an isomorphism.
So $\mathcal M(X, \Psi_X)$ is equivalent to $\mathcal M(Y, \Psi_Y)$ if and only if there exists an isomorphism of $G$-sets $f:X \to Y$ and $\Psi_X$ is cohomologous to $\Psi_Y \circ (\ident_G \times \ident_G \times f)$. Here we treated $\Psi_Y$ as a map $G\times G\times Y\to \units$.

In order to make this classification more concrete, first note that 2-cochains $\Psi\in C^2(G, \Map(X, \units))$ with $\dif \Psi= \tilde\omega^{-1}$ exist if and only if $\tilde\omega\in C^3(G, \Map(X, \units))$ defined by (\ref{eq:cochain_incl}) is a coboundary.

Let $X$ be a $G$-set such that $\tilde\omega\in C^3(G, \Map(X, \units))$ is a coboundary. Set $\eta:=\tilde\omega^{-1}$ and $M:=\Map(X, \units)$. We need to classify the 2-cochains $\Psi\in C^2(G, M)$ with $\dif \Psi = \eta$ up to cohomology. Multiplying with an arbitrary $\Psi\in C^2(G, M)$ that satisfies $\dif \Psi=\eta^{-1}$ is a bijection
\[ \left\{ \Psi\in C^2(G, M) \mid \dif \Psi = \eta \right\}\ \overset{\sim}{\longrightarrow}\ \left\{ \overline\Psi\in C^2(G, M) \
\mid \dif \overline\Psi = 1 \right\} = Z^2(G, M). \]
Cohomologous pairs of cochains are mapped to cohomologous cochains by this bijection and by its inverse. Hence it induces a bijection
\[ \left\{ \Psi\in C^2(G, M) \mid \dif \Psi = \eta \right\} / \mathbin{\sim} \ \overset{\sim}{\longrightarrow} \ H^2(G, M), \]
where $\Psi \sim \Psi'$ iff $\Psi$ and $\Psi'$ are cohomologous. Hence the equivalence classes of semisimple $\vectgo$-module categories are in
bijection with pairs $([X], \overline\Psi)$, where
\begin{itemize}
	\item $[X]$ is an isomorphism class of a $G$-set $X$ with $\tilde\omega \in B^3(G, \Map(X, \units))$, where
	\begin{equation} \label{eq:omega_tilde} \tilde\omega (g, h, k, x) := \omega(g, h, k), \end{equation}
	\item $\overline\Psi \in H^2(G, \Map(X, \units))$.
\end{itemize}

As every 2-cochain $\Psi$ with $\dif\Psi = \tilde\omega^{-1}$ is cohomologous to a normalised one by Lemma \ref{lem:norm_cochain}, we will assume that $\Psi$ is normalised whenever we consider a $\vectgo$-module category $\mathcal M(X, \Psi)$ from here on, i.e.,
\[ \Psi(g, 1, x) = \Psi(1, g, x) = 1 \qquad (g\in G, x\in X). \]
It is possible to give a more concrete classification of \emph{indecomposable} semisimple $\vectgo$-module categories. This uses the well-known classification of transitive $G$-sets:

\begin{proposition}\label{prop:trans_Gset}
	\cite[Proposition 6.8.4]{artin_algebra}
	\vspace*{-1\baselineskip}
	\begin{enumerate}
		\item Every transitive $G$-set is isomorphic to $G/H$ with $G$ acting by
		\[ g \opl (g'H) := gg'H \qquad (g, g'\in G) \]
		for some subgroup $H\subset G$.
		\item $G/H_1 \cong G/H_2$ as $G$-sets iff $H_1$ and $H_2$ are conjugate.
	\end{enumerate}
\end{proposition}

We now apply Proposition \ref{prop:trans_Gset} to the classification of semisimple $\vectgo$-module categories. Let $X$ be a $G$-set and let $\Psi\in C^2(G, \Map(X, \units))$ with $\dif \Psi= \omega^{-1}$.

\begin{proposition} \label{prop:indecomp}
	\begin{enumerate}
	\item \label{it:indecomp} $\mathcal M(X, \Psi)$ is an indecomposable $\vectgo$-module category iff $X$ is a transitive $G$-set.
	\item The indecomposable semisimple $\vectgo$-module categories are classified by pairs $([H], \psi)$, where
	\begin{itemize}
		\item $[H]$ is a conjugacy equivalence class of a subgroup $H\subset G$ such that
		\[ \omega|_{H\times H \times H}\in B^3(H, \units), \]
		\item $\psi\in H^2(H, \units)$.
	\end{itemize}
	\end{enumerate}
\end{proposition}
\begin{proof}
	1. If $X$ is a non-transitive $G$-set, there exist non-empty subsets $X_1, X_2 \subset X$ that are closed under the action of $G$ and satisfy $X = X_1 \mathbin{\dot\cup} X_2$.
	For $i\in\{1, 2\}$ let $\mathcal M_i:=\mathcal M(X_i, \Psi_i)$, where $\Psi_i(g, h, x):=\Psi(g, h, x)$ for $x\in X_i$. Then the simple objects of $\mathcal M_1 \oplus \mathcal M_2$ correspond to the elements of $X_1 \cup X_2 = X$ and it is easy to check that $\mathcal M_1 \oplus \mathcal M_2 \cong \mathcal M(X, \Psi)$ as module categories.
	
	Now assume $\mathcal M(X, \Psi) \cong \mathcal M_1 \oplus \mathcal M_2$ for non-trivial module categories $\mathcal M_1, \mathcal M_2$. Then every simple object from $\modc$ is either isomorphic to a simple object in $\mathcal M_1$ or it is isomorphic to a simple object in $\mathcal M_2$. This gives a bijection $X \cong X_1 \dcup X_2$, where $X_1$ and $X_2$ are sets of representatives for the isomorphism classes of simple objects in $\mathcal M_1$ and $\mathcal M_2$, respectively. Since $\mathcal M_1$ and $\mathcal M_2$ both need to be semisimple and non-trivial, we have $\emptyset \neq X_1, X_2$. Both $X_1$ and $X_2$, considered as subsets of $X$, must be closed under the action of $G$, so $X$ is not transitive.
	
	2. By \ref{it:indecomp}. and Proposition \ref{prop:trans_Gset} every indecomposable semisimple $\vectgo$-module category is equivalent to $\mathcal M(G/H, \Psi)$
	for some subgroup $H\subset G$ and some 2-cochain $\Psi\in C^2(G, \Map(G/H, \units))$ with $\dif\Psi = \omega^{-1}$. Proposition \ref{prop:trans_Gset} also implies that $H$ is only determined up to conjugacy. By Shapiro's Lemma \ref{lem:shapiro}, 
	\[ H^n(G, \Map(G/H, \units)) \cong H^n(H, \units), \]
	so $\Psi$ can be replaced by its image under this isomorphism, which is
	\[ \psi(h_1, h_2) := \Psi(h_1, h_2, H) \qquad (h_1, h_2 \in H). \]
	The isomorphism $H^3(G, \Map(G/H, \units)) \cong H^3(H, \units)$ maps $\tilde\omega$ to $\omega|_{H\times H\times H}$, so $\omega|_{H\times H \times H}$ is a coboundary iff $\tilde\omega$ is one. This concludes the proof.
\end{proof}

We denote the $\vectgo$-module category determined by a conjugacy equivalence class $[H]$ and $\psi\in H^2(H, \units)$ by $\mathcal M'(H, \psi)$.

It remains to investigate whether a given $\vectgok$-module category admits a module trace. This depends on the choice of $\kappa$.

\begin{lemma} \label{lem:vecgo_mod_tr}
	\begin{enumerate}
		\item The $\vectgok$-module category $\mathcal M'(H, \psi)$ admits a module trace (which is unique up to rescaling) if and only if $\kappa|_H = 1$ \cite[Example 3.13]{schaumann}.

		\item Let $X$ be a $G$-set with finitely many orbits. The $\vectgok$-module category $\mathcal M(X, \Psi)$ admits a module trace iff every orbit $\Gamma$ of $X$ is isomorphic as a $G$-set to $G/H_{\Gamma}$ for some subgroup $H_{\Gamma}\subset G$ with $\kappa|_{H_{\Gamma}} = 1$. In particular, $\modc$ admits a module trace if $\kappa = 1$.
	\end{enumerate}
\end{lemma}
\begin{proof}
	1. Assume $\theta$ is a module trace on $\mathcal M'(H, \psi)$. Since $\theta$ is $\vectgok$-compatible,
	\[ \dim^{\theta}(gH) = \theta_{gH} (\ident_{gH}) = \theta_H \left(\trace^{\vectgok}_{g, H}(\ident_{gH})\right) = \dim(\delta^g) \dim^{\theta}(H) = \kappa(g) \dim^{\theta}(H). \]
	For $g\in H$, this implies $\kappa(g) = 1$. The above equation defines $\theta$ uniquely up to a factor since $\dim^{\theta}(H) \in\units$ can be chosen arbitrarily. It is straightforward to check that the above formula for the dimension of simple objects extends uniquely to a module trace on $\mathcal M'(H, \psi)$.
	
	2. From the proof of Proposition \ref{prop:indecomp} it follows that 
	\[ \modc \cong \bigoplus_{\Gamma\in X/G} \, \mathcal M\left(\Gamma, \Psi|_{G\times G\times \Gamma}\right), \]
	where $X/G$ denotes the set of orbits of $X$. Since each orbit is transitive as a $G$-set, $\mathcal M(\Gamma, \Psi|_{G\times G\times \Gamma}) \cong \mathcal M'(H_{\Gamma}, \psi_{\Gamma})$ for some $H_{\Gamma}$ and $\psi_{\Gamma}$. If $\kappa|_{H_{\Gamma}} = 1$ for all $\Gamma\in X/G$, then $\mathcal M(X, \Psi)$ admits a module trace because it is the direct sum of module categories with module traces \cite[Proposition 3.11 (ii)]{schaumann}. If, on the other hand, $\modc$ admits a module trace $\theta$, the restrictions of $\theta$ to the direct summands are module traces on the direct summands.
\end{proof}

If $H\subset G$ is a subgroup with $\kappa|_H = 1$, set $\tilde\kappa (gH) := \kappa(g)$ for $g\in G$. There exists a module trace $\theta$ on $\mathcal M'(H, \psi)$ which satisfies
\[ \dim^{\theta}(x) = \tilde\kappa(x) \qquad (x\in G/H). \]
For a given $\vectgok$-module category $\mathcal M(X, \Psi)$ with module trace $\theta$, set
\[ \tilde\kappa(x) := \dim^{\theta}(x)\qquad (x\in X). \]
If $\vectgok$ is spherical and thus $\kappa(g)\in \{-1, 1\}$ for all $g\in G$, it follows from the proof of Lemma \ref{lem:vecgo_mod_tr} that we may assume $\tilde\kappa(x)\in \{-1, 1\}$ for all $x\in X$. As the module trace is $\vectgok$-compatible, $\tilde \kappa$ satisfies
\[ \tilde\kappa(g \opl x) = \kappa(g)\, \tilde\kappa(x) \qquad (g\in G, x\in X). \]

\newcommand{\deliProd}{\mathrm{Vec}_G^{\omega_G} \boxtimes (\mathrm{Vec}_H^{\omega_H})^{rev}}

\subsection [\texorpdfstring{$\monos$-bimodule categories}{(VecG, VecH)-bimodule categories}]
{$\monos$-bimodule categories} \label{ssec:vecgo_bimodc}
By Remark \ref{rem:bimod_deligne}, $(\mathcal C, \mathcal D)$-bimodule categories can be identified with module categories over the Deligne product $\mathcal C\boxtimes \mathcal D^{rev}$.
In this section we describe this identification explicitly in the case $\mathcal C=\vectgog{G}, \mathcal D=\vectgog{H}$.

First observe that the monoidal category $\left(\vectgo\right)^{rev}$ is equivalent to $\mathrm{Vec}_{G^{op}}^{\omega'}$, where $G^{op}=(G, \cdot_{op})$ denotes the opposite group of $G$ with multiplication $g\cdot_{op} h := hg$ and 
\[ \omega'(g, h, k) := \omega^{-1}(k, h, g). \]
The monoidal category $\mathrm{Vec}_{G^{op}}^{\omega'}$ is in turn equivalent to $\mathrm{Vec}_G^{\bar \omega}$ with
\begin{equation}\label{eq:omega_bar} 
	\bar\omega(g, h, k) := \omega^{-1}(k^{-1}, h^{-1}, g^{-1}).
\end{equation}
Hence $\left(\vectgo\right)^{rev}=\mathrm{Vec}_G^{\bar \omega}$. It is straightforward to check that
\begin{eqformarray}{lrcl} \label{eq:vecgo_rev_equiv}
	F: &\left(\vectgo\right)^{rev} \hspace*{-0.5em} &\to& \mathrm{Vec}_G^{\bar \omega} \\
	& \delta^g &\mapsto& \delta^{g^{-1}}
\end{eqformarray}
\[ \phi^{\otimes}_{g, h} := \ident_{g^{-1}h^{-1}}:\ F(\delta^g) \otimes F(\delta^h) = \delta^{g^{-1}h^{-1}} \longrightarrow \delta^{g^{-1}h^{-1}} = F(\delta^h \otimes \delta^g) \]
defines an $\mathbb F$-linear monoidal equivalence $(F, \phi^{\otimes}): \left(\vectgo\right)^{rev} \to \mathrm{Vec}_G^{\bar \omega}$.

\begin{lemma}
Let $(\mathcal C, b)$ be a pivotal category. Then $\mathcal C^{rev}$ admits a canonical pivotal structure. 
\end{lemma}
\begin{proof}
	Recall from Section \ref{sec:prelim} that every pivotal category is rigid. Left duals in $\mathcal C$ are right duals in $\mathcal C^{rev}$ and
	\[ \begin{tikzcd}[column sep=large]
		\rho_C \escdef\ C \ar[r, "\ident_C \otimes \coev_{\ldual C}^R"] & C \otimes (\ldual C \otimes (\ldual C)^*) \ar[r, "a^{-1}_{C, \ldual C, (\ldual C)^*}"] & (C\otimes \ldual C) \otimes (\ldual C)^* \ar[r, "\eval_C^L"] & (\ldual C)^*
	\end{tikzcd} \]
	is a natural isomorphism, as can be checked using the graphical calculus from Section~\ref{ssec:graph_fusion} \cite[Remark 2.10.3]{egno}. Now define the pivotal structure $b'$ on $\mathcal C^{rev}$ by
	\[ \begin{tikzcd}
		b'_C \escdef\ C \ar[r, "\rho_C"] & (\ldual C)^* \ar[r, "\left(\left(\rho_{\ldual C}\right)^{-1}\right)^*"] &[2em] (\dldual C)^{**} \ar[r, "b_{\dldual C}^{-1}"] & \dldual C.
	\end{tikzcd} \]
	 A long, but straightforward calculation using the graphical calculus shows that this is indeed a monoidal natural transformation.
\end{proof}

We now specialise to $(\mathcal C, b)=\vectgok$. In this case, $\eval_g^L = \kappa(g)\, \ident_1$ by (\ref{eq:vecgo_ev}) and thus
\[ \rho_g = \omega^{-1}(g, g^{-1}, g) \, \kappa(g) \, \ident_g. \]
This implies
\[
	b'_g =
	\kappa(g) \, \omega(g^{-1}, g, g^{-1}) \, \ident_g = \beta(g) \, \ident_g = b_g,
\]
where $\beta(g) := \kappa(g) \, \omega(g^{-1}, g, g^{-1})$ is defined as in Lemma \ref{lem:vecgo_piv}.
\vspace{2pt} In (\ref{eq:vecgo_rev_equiv}), we identified $\delta^g \in \mathrm{Vec}_G^{\bar\omega}$ with $\delta^{g^{-1}}\in \left(\vectgo\right)^{rev}$. As a result,
\begin{equation} \label{eq:vecgo_rev_piv} \left(\vectgok\right)^{rev} = \mathrm{Vec}_G^{\bar \omega, \kappa^{-1}} \end{equation}
as a pivotal category.

\begin{lemma} \label{lem:vecgo_deli_prod}
	Let $G$ and $H$ be groups and let $\omega_G\in Z^3(G, \units), \omega_H\in Z^3(H, \units)$. Then
	\[ \vectgog G \boxtimes \vectgog H = \mathrm{Vec}_{G\times H}^{\omega}, \]
	where
	\[ \omega((g_1, h_1), (g_2, h_2), (g_3, h_3)) := \omega_G(g_1, g_2, g_3)\, \omega_H(h_1, h_2, h_3). \]
\end{lemma}
\begin{proof}
	It is straightforward to check that $\mathrm{Vec}_{G\times H}$ satisfies the universal property of the Deligne product from Definition \ref{def:deli_prod}. One easily verifies that the associativity constraint defined by $\omega$ is the associativity constraint of the Deligne product defined in \cite[Proposition 4.6.1]{egno}.
\end{proof}

If $\vectgokg G$ and $\vectgokg H$ are pivotal categories, then
\[ \kappa((g, h)) := \kappa_G(g) \, \kappa_H(h) \qquad (g\in G, h\in H) \]
defines a character of $G \times H$ and we have
\[ \vectgokg G \boxtimes \vectgokg H = \mathrm{Vec}_{G\times H}^{\omega, \kappa} \]
with $\omega$ defined as in Lemma \ref{lem:vecgo_deli_prod}. By combining this result with (\ref{eq:vecgo_rev_piv}), we get
\[ \vectgokg G \boxtimes \left(\vectgokg H \right)^{rev} = \mathrm{Vec}_{G\times H}^{\omega, \kappa}, \]
where
\begin{align}
	\label{eq:omega_deli_prod}
	\omega((g_1, h_1), (g_2, h_2), (g_3, h_3)) &:= \omega_G(g_1, g_2, g_3) \, \omega_H^{-1}(h_3^{-1}, h_2^{-1}, h_1^{-1}) \\
	\label{eq:kappa_deli_prod}
	\kappa((g, h)) &:= \kappa_G(g) \, \kappa_H^{-1}(h).
\end{align}

We now give an explicit bijection between semisimple $\monos$-bimodule categories and semisimple $\mathrm{Vec}_{G\times H}^{\omega}$-module categories, where $\omega$ is defined as in (\ref{eq:omega_deli_prod}).
 
Let $\mathcal M$ be a $\monos$-bimodule category. As a $\mathrm{Vec}_G^{\omega_G}$-left module category, $\mathcal M$ is equivalent to $\mathcal M(X, \Psi)$ for some $G$-set $X$ and some $\Psi\in C^2(G, \Map(X, \units))$ with $\dif \Psi=\omega_G^{-1}$.
Assume w.l.o.g. that $x\opr \delta^h \in X$ for all $x\in X, h\in H$. Then $h\opl x := x \opr \delta^{h^{-1}}$ defines an $H$-set structure on $X$, which satisfies
\[ g \opl (h\opl x) = \delta^g \opl (x\opr \delta^{h^{-1}}) = (\delta^g \opl x) \opr \delta^{h^{-1}} = h \opl (g \opl x) \]
for $g\in G, h\in H, x\in X$, so
\[ (g, h) \opl x := g \opl (h \opl x) \qquad (g\in G, h\in H, x\in X) \]
turns $X$ into a $G\times H$-set. Notationally, we do not distinguish between the $G$-action, $H$-action and $G\times H$-action on $X$.
Now define maps $\Phi\in C^2(H, \Map(X, \units)), \Omega\in \Map(G\times H\times X, \units)$ by
\begin{eqformarray}{rrlcl} \label{eq:bimodc_constr}
	n_{x, h_1, h_2} =& \hspace*{-0.5em} \Phi(h_2^{-1}, h_1^{-1}, (h_1 h_2)^{-1} \opl x) \hspace*{-0.5em} &\ident_{(h_1 h_2)^{-1}\opl x} \hspace*{-0.5em}&:& x \opr (\delta^{h_1} \otimes \delta^{h_2}) \to (x\opr \delta^{h_1}) \opr \delta^{h_2} \\
	b_{g, x, h} =& \Omega(g, h^{-1}, (g, h^{-1}) \opl x) \hspace*{-0.5em} &\ident_{(g, h^{-1})\opl x} \hspace*{-0.5em}&:& (\delta^g \opl x) \opr \delta^h \to \delta^g \opl (x\opr \delta^h)
\end{eqformarray}
for $g\in G, h, h_1, h_2\in H, x\in X$.
\newcommand{\baromega}{\bar\omega_H}
The axioms for a bimodule category imply $\dif \Phi = \baromega^{-1}$:
\begin{multline*} 
	\Phi(h_2, h_3, h_1^{-1}\opl x) \, \Phi^{-1}(h_1 h_2, h_3, x) \, \Phi(h_1, h_2 h_3, x) \, \Phi^{-1}(h_1, h_2, x) = \\ =\baromega^{-1} (h_1, h_2, h_3) = \omega_H(h_3^{-1}, h_2^{-1}, h_1^{-1}) 
\end{multline*}
for $h_1, h_2, h_3\in H, x\in X$, and
\begin{align} \label{eq:omega_cond_1}
	\Omega(g_2, h, g_1^{-1} \opl x) \, \Omega^{-1}(g_1 g_2, h, x) \, \Omega(g_1, h, x) &= \Psi(g_1, g_2, x) \, \Psi^{-1}(g_1, g_2, h^{-1} \opl x) \\
	\label{eq:omega_cond_2}
	\Omega(g, h_2, h_1^{-1} \opl x) \, \Omega^{-1}(g, h_1 h_2, x) \, \Omega(g, h_1, x) &= \Phi^{-1}(h_1, h_2, x)\, \Phi(h_1, h_2, g^{-1} \opl x)
\end{align}
for $g, g_1, g_2 \in G, h, h_1, h_2\in H, x\in X$. Conversely, let $X$ be a $G\times H$-set and let $\Psi, \Phi, \Omega$ be maps that satisfy $\dif\Psi = \omega_G^{-1}$ and $\dif\Phi = \baromega^{-1}$ as well as (\ref{eq:omega_cond_1}) and (\ref{eq:omega_cond_2}).
Then these data define a $\monos$-bimodule category via (\ref{eq:bimodc_constr}), which we denote $\mathcal B(X, \Psi, \Phi, \Omega)$.

In the following, we denote the neutral elements of both $G$ and $H$ by $1$. Let $X$ be a $G\times H$-set and let $\Gamma'\in C^2(G\times H, \Map(X, \units))$ be normalised, then $\Gamma'$ is cohomologous to $\Gamma:=\Gamma' \cdot \dif \mu$, where
\[ \mu((g, h), x) := \Gamma'((1, h), (g, 1), x). \]
$\Gamma$ is still normalised and satisfies
\begin{equation}\label{eq:gamma_norm} 
	\Gamma((1, h), (g, 1), x) = 1
\end{equation}
for all $g\in G, h\in H, x\in X$, as a simple calculation shows. Hence, we may assume that (\ref{eq:gamma_norm}) is satisfied for every $\mathrm{Vec}_{G\times H}^{\omega}$-module category $\mathcal M(X, \Gamma)$.

\begin{proposition} \label{pro:vecgo_bimodc}
	A bijection between semisimple $\monos$-bimodule categories and semisimple $ \mathrm{Vec}_{G\times H}^{\omega}$-module categories is given by the two mutually inverse assignments
	\begin{formarray}{rcl}
		\mathcal B(X, \Psi, \Phi, \Omega) &\mapsto& \mathcal M(X, \Gamma') \\
		\mathcal B(X, \tilde\Psi, \tilde\Phi, \tilde\Omega) &\text{\reflectbox{$\mapsto$}}& \mathcal M(X, \Gamma),
	\end{formarray}
	where
	\begin{align*}
		\omega((g_1, h_1), (g_2, h_2), (g_3, h_3)) &:= \omega_G(g_1, g_2, g_3) \, \omega_H^{-1}(h_3^{-1}, h_2^{-1}, h_1^{-1}) 
		\\
		\Gamma'((g_1, h_1), (g_2, h_2), x) &:= \Psi(g_1, g_2, (h_2^{-1}h_1^{-1}) \opl x) \, \Phi(h_1, h_2, x) \, \Omega(g_1, h_2, h_1^{-1} \opl x) \\
		\tilde\Psi(g_1, g_2, x) &:= \Gamma((g_1, 1), (g_2, 1), x) \\
		\tilde\Phi(h_1, h_2, x) &:= \Gamma((1, h_1), (1, h_2), x) \\
		\tilde\Omega(g, h, x) &:= \Gamma((g, 1), (1, h), x)
	\end{align*}
	for $g, g_1, g_2, g_3\in G, h, h_1, h_2, h_3\in H, x\in X$.
\end{proposition}
\begin{proof}
	Using (\ref{eq:omega_cond_1}), (\ref{eq:omega_cond_2}), $\dif \Psi=\omega_G^{-1}$ and $\dif \Phi = \baromega^{-1}$ one checks that $\dif \Gamma' = \omega^{-1}$, and it is easy to verify $\dif \tilde\Psi = \omega_G^{-1}$ and $\dif \tilde\Phi = \baromega^{-1}$. We now check that $\tilde\Omega$ satisfies (\ref{eq:omega_cond_1}).
	\begin{multline*} 
		\tilde\Omega(g_2, h, g_1^{-1} \opl x) \, \tilde\Omega^{-1}(g_1 g_2, h, x) \, \tilde\Omega(g_1, h, x) = \\ \Gamma((g_2, 1), (1, h), g_1^{-1} \opl x)\, \Gamma^{-1}((g_1 g_2, 1), (1, h), x) \, \Gamma((g_1, 1), (1, h), x), 
	\end{multline*}
	to which we apply the coboundary condition (\ref{eq:2cocycle}) for the tuples $((g_1, 1), (g_2, 1), (1, h), x)$, $((g_1, 1), (1, h), (g_2, 1), x)$ and $((h, 1), (g_1, 1), (g_2, 1), x)$, as well as (\ref{eq:gamma_norm}). This yields
	\[ \tilde\Omega(g_2, h, g_1^{-1} \opl x) \, \tilde\Omega^{-1}(g_1 g_2, h, x) \, \tilde\Omega(g_1, h, x) = \tilde\Psi(g_1, g_2, x) \, \tilde\Psi^{-1}(g_1, g_2, h^{-1}\opl x), \]
	so $\tilde\Omega$ and $\tilde\Psi$ satisfy (\ref{eq:omega_cond_1}). Analogously, one checks (\ref{eq:omega_cond_2}) for $\tilde\Omega$ and $\tilde\Phi$.
	Hence $\mathcal M(X, \Gamma')$ and $\mathcal B(X, \tilde\Psi, \tilde\Phi, \tilde\Omega)$ are well-defined.
	
	Since $\Psi$ and $\Phi$ are normalised, (\ref{eq:omega_cond_1}) and (\ref{eq:omega_cond_2}) imply that
	\[ \Omega(g, 1, x) = 1 = \Omega(1, h, x) \]
	for $g\in G, h\in H, x\in X$. This makes it trivial to check that the composite of maps
	\[ \mathcal B(X, \Psi, \Phi, \Omega) \mapsto \mathcal M(X, \Gamma') \mapsto \mathcal B(X, \tilde\Psi', \tilde\Phi', \tilde\Omega') \]
	is the identity. Finally, let $\mathcal M(X, \tilde\Gamma')$ be the image of $\mathcal M(X, \Gamma)$ under the composition
	\begin{equation} \label{eq:second_comp} \mathcal M(X, \Gamma) \mapsto \mathcal B(X, \tilde\Psi, \tilde\Phi, \tilde\Omega) \mapsto \mathcal M(X, \tilde\Gamma'). \end{equation}
	We apply the coboundary condition (\ref{eq:2cocycle}) for $\Gamma$ and the tuples $((g_1, h_1), (1, h_2), (g_2, 1), x)$, $((1, h_1 h_2), (g_1, 1), (g_2, 1), x)$, $((1, h_1), (g_1, 1), (1, h_2), x)$ and $((1, h_1), (1, h_2), (g_1, 1), x)$ and get
	\begin{align*}
		\Gamma((g_1, h_1), (g_2, h_2), x) &= \Gamma((g_1, 1), (g_2, 1), (h_1 h_2)^{-1}\opl x) \, \Gamma((1, h_1), (1, h_2), x) \\ & \hphantom{{}\mathbin{=}{}} \Gamma((g_1, 1), (1, h_2), h_1^{-1}\opl x) \\
		&= \tilde\Psi(g_1, g_2, (h_1 h_2)^{-1}\opl x) \, \tilde\Phi(h_1, h_2, x) \, \tilde\Omega(g_1, h_2, h_1^{-1}\opl x) \\
		&= \tilde\Gamma'((g_1, h_1), (g_2, h_2), x).
	\end{align*}
	Hence the composite of maps in (\ref{eq:second_comp}) is also the identity map.
\end{proof}

In particular, it is now possible to decide whether a given $\monosk$-bimodule category admits a bimodule trace by using Lemma \ref{lem:vecgo_mod_tr} with $\kappa$ defined by (\ref{eq:kappa_deli_prod}).
\clearpage
\section{Module functors}\label{sec:modf}
In the previous section, we defined (bi-)module functors between (bi-)module categories. The aim of this section is to describe the module functors between finite semisimple $\vectgo$-module categories in more detail.

\subsection{Properties of module functors} \label{ssec:modf_prop}
Let $\mathcal C$ be a multitensor category over $\mathbb F$.
It can be shown that semisimple $\mathcal C$-module categories together with module functors and module natural transformations form a 2-category $\mathrm{Mod}(\mathcal C)$ \cite[Remark 7.12.15]{egno}. We will only use a part of this result, which is proved in the following.

Let $\mathcal M$ and $\mathcal N$ be semisimple $\mathcal C$-module categories.

\begin{proposition} \label{pro:modf_cat}
	\cite[Proposition 7.11.1]{egno}
	The $\mathcal C$-module functors $\mathcal M \to \mathcal N$ form an abelian $\mathbb F$-linear category $\Fun_{\mathcal C}(\mathcal M, \mathcal N)$.
\end{proposition}
\begin{proof}
	It is shown in \cite[Proposition 7.11.1]{egno} that $\Fun_{\mathcal C}(\mathcal M, \mathcal N)$ is abelian and it is obvious that $\Fun_{\mathcal C}(\mathcal M, \mathcal N)$ is $\mathbb F$-linear. 
	
	We now give an explicit construction of the direct sum of module functors which will be used later.
	For module functors $(F^{(1)}, s^{(1)}), \dots, (F^{(n)}, s^{(n)}) \in \Fun_{\mathcal C}(\mathcal M, \mathcal N)$, define a $\mathcal C$-module functor $(F, s): \mathcal M \to \mathcal N$ by
	\[ F(M) := \bigoplus_{i=1}^n F^{(i)}(M), \]
	\[ (\ident_C \opl \pi^{(i)}_M) \circ s_{C, M} \circ \iota^{(k)}_{C\opl M} := \delta_{i,k} \, s^{(i)}_{C, M}. \]
	Here $\iota^{(i)}_M: F^{(i)}(M) \to F(M), \pi^{(i)}_M: F(M)\to F^{(i)}(M)$ are inclusions and projections for $F(M)$ which are natural in $M$.
	It is straightforward to check that $(F, s)$ is a $\mathcal C$-module functor and that
	\[ (\iota^{(i)}_M)_{M\in\mathcal M}: (F^{(i)}, s^{(i)}) \Rightarrow (F, s), \quad (\pi^{(i)}_M)_{M\in \mathcal M}: (F, s) \Rightarrow (F^{(i)}, s^{(i)}) \qquad (i=1, ..., n) \]
	are inclusions and projections for $(F, s)$. Hence $(F, s)=\bigoplus_{i=1}^n (F^{(i)}, s^{(i)})$.
\end{proof}

\begin{proposition} \label{pro:modfc_semi}
	\cite[Theorem 2.16]{eno05}
	If $\mathcal C$ is a fusion category and $\mathcal M, \mathcal N$ are module categories over $\mathcal C$, the category $\Fun_{\mathcal C}(\mathcal M, \mathcal N)$ is semisimple.
\end{proposition}

We will call a $\mathcal C$-module functor $F: \mathcal M\to \mathcal N$ \emph{simple} if it is simple as an object of $\Fun_{\mathcal C}(\mathcal M, \mathcal N)$. By Proposition \ref{pro:modfc_semi}, it is sufficient to classify simple module functors in order to fully describe the category $\Fun_{\mathcal C}(\mathcal M, \mathcal N)$.

Let $\mathcal M$ and $\mathcal N$ be finite semisimple $\mathcal C$-module categories. Then every additive functor $F: \mathcal M\to \mathcal N$ is exact by Lemma \ref{lem:exactness}. The following facts imply that $F$ has a right adjoint:

\begin{proposition} 
	\cite[Section 1.8]{egno}
	Let $\mathcal A$ be a finite category over $\mathbb F$. Then there exists a finite dimensional $\mathbb F$-algebra $A$ such that $\mathcal A$ is equivalent to the category $A-\Mod$ of finite dimensional $A$-modules. 
\end{proposition}

\begin{proposition}
	\cite[Proposition 1.8.10]{egno}
	Let $A$ and $B$ be finite-dimensional $\mathbb F$-algebras. Then for every right exact functor $F: A-\Mod \to B-\Mod$ there exists a $(B, A)$-bimodule $V$ such that $F$ is naturally isomorphic to the functor $V \otimes_A -: A-\Mod \to B-\Mod$.
\end{proposition}

\begin{proposition}
	\cite[Proof of Corollary 4.5.8]{riehl17}
	Let $A, B$ be $\mathbb F$-algebras and let $V$ be a $(B, A)$-bimodule. Then the functor $V \otimes_A -: A-\Mod \to B-\Mod$ is left adjoint to 
	\[ \Hom_{B-\Mod}(V, -): B-\Mod \to A-\Mod. \]
\end{proposition}

Let $(F, s):\mathcal M\to \mathcal N$ be a $\mathcal C$-module functor. Then, by \cite[Section 7.12]{egno}, its right adjoint $F^r: \mathcal N\to \mathcal M$ is also a $\mathcal C$-module functor whose coherence datum $s'_{C, N}: F^r(C\opl N) \to C\opl F^rN$ is defined by the isomorphism
\begin{multline}\label{eq:modf_adj}
	s'_{C, N} \circ - : \Hom_{\mathcal M}(M, F^r(C\opl N)) \cong \Hom_{\mathcal N}(FM, C\opl N) \overset{\ref{lem:*opl_adj}}{\cong} \Hom_{\mathcal N}(C^*\opl FM, N) \\ 
	\cong \Hom_{\mathcal N}(F(C^*\opl M), N) \cong \Hom_{\mathcal M}(C^* \opl M, F^rN) \overset{\ref{lem:*opl_adj}}{\cong} \Hom_{\mathcal M}(M, C\opl F^rN)
\end{multline}
natural in $M\in\mathcal M, C\in\mathcal C$ and $N\in\mathcal N$.

\begin{remark} \cite[Exercise 7.12.1]{egno}
	Suppose $\mathcal M = \mathcal N$. The category $\Fun_{\mathcal C}(\mathcal M, \mathcal M)$ is a monoidal category with composition as tensor product and the module functor $(F^r, s')$ is a left dual of $(F, s)$ in $\Fun_{\mathcal C}(\mathcal M, \mathcal M)$. The evaluation for $F$ is the counit $\epsilon: FF^r \Rightarrow \Ident_{\mathcal M}$ and the coevaluation is the unit $\eta: \Ident_{\mathcal M} \Rightarrow FF^r$. Likewise, the left adjoint of $F$ is a right dual of $(F, s)$ when equipped with a certain coherence datum. Thus, $\Fun_{\mathcal C}(\mathcal M, \mathcal M)$ is a tensor category.
\end{remark}

\subsection{Description of $\vectgo$-module functors} \label{ssec:vecgo_modf}
In the following, we restrict attention to the fusion category $\mathcal C=\vectgo$. As before, we only consider semisimple $\vectgo$-module categories. Additionally, all $\vectgo$-module categories are considered to be finite.

\newcommand{\modfxy}{(F, s): \mathcal M(X, \Psi_X) \to \mathcal M(Y, \Psi_Y)}
Let $F: \mathcal A \to \mathcal B$ be an additive functor between finite semisimple categories $\mathcal A$ and $\mathcal B$ and let $I_{\mathcal A}$ and $I_{\mathcal B}$ be sets of representatives for the simple objects of $\mathcal A$ and $\mathcal B$, respectively. For every simple object $a\in I_{\mathcal A}$, there is a unique decomposition of $F(a)$ into simple objects from $I_{\mathcal B}$:
\[ F(a) \cong \bigoplus_{b\in I_{\mathcal B}} b^{\oplus m_{ab}^F}, \]
where $m_{ab}^F = \dim_{\mathbb F} \Hom(F(a), b) \in\mathbb N$. Choose inclusions and projections for $F(a)$ and denote them
\[ j_{Fab}^{\alpha}: b \to F(a), \quad p_{Fab}^{\alpha}: F(a) \to b \qquad (b\in I_{\mathcal B}, \alpha=1, \dots, m_{ab}^F). \]
By definition, they satisfy
\[ p_{Fab}^{\alpha} \circ j_{Fab'}^{\beta} = \delta_{b, b'}\, \delta_{\alpha \beta}\, \ident_b\, , \quad \sum_{b\in I_{\mathcal B}} \sum_{\alpha=1}^{m_{ab}^F} j_{Fab}^{\alpha} \circ p_{Fab}^{\alpha} = \ident_{F(a)}. \]
It will be convenient to write
\[ j_{Fab}:= \left(j_{Fab}^{1}, \dots, j_{Fab}^{m_{ab}^F}\right): b^{\oplus m_{ab}^F} \to F(a), \quad p_{Fab}:= \left(p_{Fab}^1, \dots, p_{Fab}^{m_{ab}^F}\right)^T: F(a) \to b^{\oplus m_{ab}^F}. \]
Clearly, the morphisms $(j_{Fab}, p_{Fab})_{b\in I_{\mathcal B}}$ are also inclusions and projections for $F(a)$.

Now let $\mathcal M:=\mathcal M(X, \Psi_X)$ and $\mathcal N:=\mathcal M(Y, \Psi_Y)$ be finite semisimple $\vectgo$-module categories as described in Section \ref{ssec:vecgo_modc} and let $(F, s): \mathcal M\to \mathcal N$ be a $\vectgo$-module functor. Clearly, we can choose $I_{\mathcal M}=X, I_{\mathcal N}=Y$ and the functor $F$ is described up to natural isomorphism by the numbers $(m_{xy}^F)_{x\in X, y\in Y}$. Due to
\[ \bigoplus_{y\in Y} y^{\oplus m_{g\opl x, y}^F} = F(\delta^g \opl x) \cong \delta^g \opl F(x) = \bigoplus_{y\in Y} (g\opl y)^{\oplus m_{xy}^F} = \bigoplus_{y\in Y} y^{\oplus m_{x, g^{-1}\opl y}^F}, \]
these numbers satisfy $m_{g\opl x, y}^F = m_{x, g^{-1}\opl y}^F$ for all $g\in G, x\in X, y\in Y$ or, equivalently, $m_{g\opl x, g\opl y}^F = m_{xy}^F$.
It remains to find a suitable description of the coherence datum $s$.

The coherence datum $s_{g, x}: F(\delta^g \opl x) \to \delta^g \opl F(x)$ is uniquely determined by the morphisms
\[ (\ident_g \opl p_{Fxy}) \circ s_{g, x} \circ j_{F, g\opl x, z}: z^{\oplus m_{g\opl x, z}^F} \to (g\opl y)^{\oplus m_{xy}^F} \]
for $g\in G, x\in X, y, z\in Y$. But if $z \neq g\opl y$, $\Hom_{\vectfin[Y]}(z, g\opl y) = 0$, so the above morphism is automatically $0$. Hence $s$ is already determined by the smaller family of morphisms
\begin{equation} \label{eq:def_A} A_{g, x, y}^F := (\ident_g \opl p_{Fxy}) \circ s_{g, x} \circ j_{F,g\opl x, g\opl y} \qquad (g\in G, x\in X, y\in Y). \end{equation}
Furthermore, it is easy to check that $A_{g, x, y}^F$ has the inverse
\[ \left(A_{g, x, y}^F\right)^{-1} = p_{F, g\opl x, g\opl y} \circ s_{g, x}^{-1} \circ (\ident_g \opl j_{Fxy}). \]
From here on, we will identify $\Hom_{\mathrm{Vec}_X}(x^n, x^m) \cong \Hom_{\mathrm{Vec}}(\mathbb F^n, \mathbb F^m)$ with $\Mat(m\times n, \mathbb F)$ for $x\in X$, thus viewing $A_{g, x, y}^F$ as an element of $\glmat(m_{xy}^F, \mathbb F)$. The elements of this matrix are given by the morphisms
\[ \left( A_{g, x, y}^F \right)_\beta^\alpha = (\ident_g \opl p_{Fxy}^{\alpha}) \circ s_{g, x} \circ j_{F,g\opl x, g\opl y}^{\beta} \in \Hom(g\opl y, g\opl y)\cong \mathbb F \]
for $\alpha, \beta = 1, \dots, m_{xy}^F$.
Here $A_j^i$ denotes the entry in the $i$-th row and $j$-th column of $A$.
The pentagon axiom for module functors (\ref{diag:modf}) implies
\begin{align*} \label{eq:calc_A}
	A_{gh, x, y}^F &= (\ident_{gh} \opl p_{Fxy}) \circ s_{gh, x} \circ j_{F, (gh)\opl x, (gh)\opl y} \\ &\overset{(\ref{diag:modf})} = (\ident_{gh} \opl p_{Fxy}) \circ \left(m_{g, h, F(x)}^{-1} \circ (\ident_g \opl s_{h, x}) \circ s_{g, h\opl x} \circ Fm_{g, h, x}\right) \circ j_{F, (gh)\opl x, (gh)\opl y} \\
	&\overset{(*)} = \Psi_Y^{-1}(g, h, (gh)\opl y)\, \Psi_X(g, h, (gh)\opl x)\\ 
	\refstepcounter{equation}\tag{\theequation}
	& \hphantom{=}(\ident_g \opl (\ident_h \opl p_{Fxy})) \circ (\ident_g \opl s_{h, x}) \circ s_{g, h\opl x} \circ j_{F, (gh)\opl x, (gh)\opl y} \\
	&= \Psi_X(g, h, (gh)\opl x)\, \Psi_Y^{-1}(g, h, (gh)\opl y) \\ & \hphantom{=} (\ident_g \opl A_{h, x, y}^F) \circ (\ident_g \opl p_{F, h\opl x, h\opl y}) \circ s_{g, h\opl x} \circ j_{F, (gh)\opl x, (gh)\opl y} \\
	&= \Psi_X(g, h, (gh)\opl x)\, \Psi_Y^{-1}(g, h, (gh)\opl y) \ A_{h, x, y}^F \, A_{g, h\opl x, h\opl y}^F
\end{align*}
Here we abused notation and denoted by $m$ the module constraint of $\mathcal M(X, \Psi_X)$ and of $\mathcal M(Y, \Psi_Y)$.
Furthermore, we identified $A_{h, x, y}^F$ and $\ident_g \opl A_{h, x, y}^F$ as matrices.
One can check that this is consistent by calculating the entries of $A_{gh, x, y}^F$.
In $(*)$, we used the naturality of $m$ and the definition of $m$ on simple objects. The remaining equalities follow from (\ref{eq:def_A}).

\begin{lemma} \label{lem:modf_vectgo}
	Every $\vectgo$-module functor $\modfxy$ is determined up to isomorphism by
	\begin{itemize}
		\item natural numbers $m_{xy}^F \in \mathbb N$ for $x\in X, y\in Y$ with
		\[ m_{x, y}^F = m_{g\opl x, g\opl y}^F \]
		for all $g\in G, x\in X, y\in Y$,
		\item matrices $A_{g, x, y}^F \in \glmat(m_{xy}^F, \mathbb F)$ for $g\in G, x\in X, y\in Y$ which satisfy
		\begin{equation} \label{eq:cond_A} A_{gh, x, y}^F = \Psi_X(g, h, (gh)\opl x)\, \Psi_Y^{-1}(g, h, (gh)\opl y) \ A_{h, x, y}^F \, A_{g, h\opl x, h\opl y}^F \end{equation}
		for all $g, h\in G, x\in X, y\in Y$.
	\end{itemize}
	Conversely, each choice of such data defines a $\vectgo$-module functor $\modfxy$.
\end{lemma}
\begin{proof}
	It has already been shown that every $\vectgo$-module functor $(F, s)$ defines such families via
	\[ m_{xy}^F := \dim_{\mathbb F} \Hom(F(x), y), \]
	\[ A_{g, x, y}^F := (\ident_g \opl p_{Fxy}) \circ s_{g, x} \circ j_{F,g\opl x, g\opl y} \quad (g\in G, x\in X, y\in Y). \]
	Let $(m_{xy})_{x, y}, (A_{g, x, y})_{g, x, y}$ be families that satisfy the conditions in this lemma. Then
	\[ F(x) := \bigoplus_{y\in Y} y^{\oplus m_{xy}} \qquad (x\in X) \]
	determines a functor $F: \mathcal M(X, \Psi_X) \to \mathcal M(Y, \Psi_Y)$ up to natural isomorphism by Lemma \ref{lem:fun_simp}. Choose inclusions and projections 
	\[ j_{Fxy}: y^{\oplus m_{xy}} \to F(x), \quad p_{Fxy}: F(x) \to y^{\oplus m_{xy}} \qquad (x\in X, y\in Y) \]
	for $F(x)$ and define a natural isomorphism $s$ by
	\[ (\ident_g \opl p_{Fxy}) \circ s_{g, x} \circ j_{F,g\opl x, z} := \delta_{z, g\opl y} \, A_{g, x, y} \quad (g\in G, x\in X, y\in Y). \]
	A calculation similar to (\ref{eq:calc_A}) shows that $(F, s)$ satisfies the module functor pentagon axiom.
	We now show that two $\vectgo$-module functors $(F, s^F), (H, s^H): \mathcal M(X, \Psi_X)\to \mathcal M(Y, \Psi_Y)$ are isomorphic if
	\[ m_{xy}^F = m_{xy}^H, \quad A_{g, x, y}^F = A_{g, x, y}^H \]
	for all $x\in X, y\in Y, g\in G$.
	Define a natural transformation $\nu: F\Rightarrow H$ by
	\[ \nu_x := \sum_{y\in Y} j_{Hxy} \circ p_{Fxy}:\ F(x) \to H(x) \qquad (x\in X). \]
	It is straightforward to check that $\nu$ is a natural isomorphism. In addition,
	\begin{align*}
		(\ident_g \opl \nu_x) \circ s^F_{g, x} &= \sum_{y\in Y} (\ident_g \opl j_{Hxy}) \circ (\ident_g \opl p_{Fxy}) \circ s^F_{g, x} \\ 
		&= \sum_{y\in Y} (\ident_g \opl j_{Hxy}) \circ A_{g, x, y}^F \circ p_{F, g\opl x, g\opl y} \\
		&= \sum_{y\in Y} (\ident_g \opl j_{Hxy}) \circ A_{g, x, y}^H \circ p_{F, g\opl x, g\opl y} \\
		&= \sum_{y\in Y} s^H_{g, x} \circ j_{H, g\opl x, g\opl y} \circ p_{F, g\opl x, g\opl y} = s^H_{g, x} \circ \nu_{g\opl x},
	\end{align*}
	so $\nu$ is an isomorphism of module functors. This concludes the proof.
\end{proof}

In particular, setting $g=h=1$ in (\ref{eq:cond_A}) shows that $A_{1, x, y}^F$ is the identity matrix for all $x\in X, y\in Y$.

Just like $\vectgo$-module functors, a $\vectgo$-module natural transformation can be described by a family of matrices.
Let $(F, s^F), (H, s^H): \mathcal M(X, \Psi_X) \to \mathcal M(Y, \Psi_Y)$ be $\vectgo$-module functors and let $\eta: (F, s^F) \Rightarrow (H, s^H)$ be a morphism of module functors. Then $\eta$ is uniquely determined by the morphisms
\[ M_{x, y}^{\eta} := p_{Hxy} \circ \eta_x \circ j_{Fxy} \qquad (x\in X, y\in Y), \]
which can again be regarded as matrices $M_{xy}^{\eta} \in \Mat(m_{xy}^H \times m_{xy}^F, \mathbb F)$. Since $\eta$ is a morphism of module functors, we have
\begin{align*} 
	M_{x, y}^{\eta}\, A^F_{g, x, y} &= (\ident_g \opl p_{Hxy}) \circ (\ident_g \opl \eta_x) \circ (\ident_g \opl j_{Fxy}) \circ (\ident_g \opl p_{Fxy}) \circ s^F_{g, x} \circ j_{F, g\opl x, g\opl y} \\
	&\overset{(*)} = \sum_{z\in Y} (\ident_g \opl p_{Hxy}) \circ (\ident_g \opl \eta_x) \circ (\ident_g \opl j_{Fxz}) \circ (\ident_g \opl p_{Fxz}) \circ s^F_{g, x} \circ j_{F, g\opl x, g\opl y}\\
	&= (\ident_g \opl p_{Hxy}) \circ (\ident_g \opl \eta_x) \circ s^F_{g, x} \circ j_{F, g\opl x, g\opl y} \\
	&\overset{(\ref{diag:mor_modf})} = (\ident_g \opl p_{Hxy}) \circ s^H_{g, x} \circ \eta_{g\opl x} \circ j_{F, g\opl x, g\opl y} \\
	&= A^H_{g, x, y} \circ p_{H, g\opl x, g\opl y} \circ \eta_{g\opl x} \circ j_{F, g\opl x, g\opl y}
	= A^H_{g, x, y} \, M^{\eta}_{g\opl x, g\opl y}.
\end{align*}
Here, $(*)$ holds because $(\ident_g \opl p_{Fxz}) \circ s^F_{g, x} \circ j_{F, g\opl x, g\opl y}: (g\opl y)^{\oplus m_{xy}^F} \to (g\opl z)^{\oplus m_{xz}^F}$ vanishes by Schur's Lemma \ref{lem:schur} if $z\neq y$.

\begin{lemma} \label{lem:vectgo_mor_modf}
	Let $(F, s^F), (H, s^H): \mathcal M(X, \Psi_X) \to \mathcal M(Y, \Psi_Y)$ be $\vectgo$-module functors. Then morphisms $\eta: (F, s^F)\Rightarrow (H, s^H)$ are in bijection with families of matrices
	\[ M_{x, y}^{\eta} \in \Mat(m_{xy}^H\times m_{xy}^F, \mathbb F) \quad (x\in X, y\in Y) \]
	which satisfy
	\begin{equation} \label{eq:cond_M} M_{x, y}^{\eta}\, A^F_{g, x, y} = A^H_{g, x, y} \, M^{\eta}_{g\opl x, g\opl y} \end{equation}
	for all $g\in G, x\in X, y\in Y$.
	Moreover, $\eta$ is an isomorphism iff $M_{x, y}^{\eta}$ is quadratic and invertible for all $x\in X, y\in Y$.
\end{lemma}
\begin{proof}
	Analogous to Lemma \ref{lem:modf_vectgo}.
\end{proof}

\newcommand{\bimodcx}{\mathcal B(X, \Psi_X, \Phi_X, \Omega_X)}
\newcommand{\bimodcy}{\mathcal B(Y, \Psi_Y, \Phi_Y, \Omega_Y)}
We now turn to $\monos$-bimodule functors. Let $\mathcal M:=\bimodcx$ and $\mathcal N:=\bimodcy$ be finite semisimple $\monos$-bimodule categories as defined in Section \ref{ssec:vecgo_bimodc}. Then a $\monos$-bimodule functor $(F, s, t): \mathcal M \to \mathcal N$ is determined by the families $(m_{xy}^F)_{x\in X, y\in Y}$ and $(A_{g, x, y}^F)_{g, x, y}$ defined above, together with the family of matrices
\[ B_{h, x, y}^F := (p_{Fxy} \opr \ident_h) \circ t_{x, h} \circ j_{F, h^{-1}\opl x, h^{-1}\opl y} \in \glmat(m_{xy}^F, \mathbb F) \qquad (h\in H, x\in X, y\in Y), \]
where $H$ operates on both $X$ and $Y$ by $h \opl x := x \opr \delta^{h^{-1}}$ as described in Section \ref{ssec:vecgo_bimodc}. These data satisfy the following conditions, in addition to those from Lemma \ref{lem:modf_vectgo}:
\begin{itemize}
	\item The isomorphism $F(x \opr \delta^h) \cong F(x) \opr \delta^h$ implies $m_{xy}^F = m_{h\opl x, h\opl y}^F$ for $h\in H, x\in X, y\in Y$.
	\item The pentagon axiom for $\mathcal D$-right module functors implies
	\[ B_{gh, x, y}^F = \Phi_X(h^{-1}, g^{-1}, (gh)^{-1}\opl x) \, \Phi_Y^{-1}(h^{-1}, g^{-1}, (gh)^{-1}\opl y) \ B_{g, x, y}^F \, B_{h, g^{-1}\opl x, g^{-1}\opl y}^F \]
	for $g, h\in H, x\in X, y\in Y$.
	\item The hexagon axiom for bimodule functors translates to
	\begin{multline*} \Omega_X (g, h^{-1}, (g, h^{-1}) \opl x) \ B_{h, x, y}^F \, A_{g, h^{-1}\opl x, h^{-1}\opl y}^F \\ = \Omega_Y(g, h^{-1}, (g, h^{-1})\opl y) \ A_{g, x, y}^F \, B_{h, g\opl x, g\opl y}^F \end{multline*}
	for all $g\in G, h\in H, x\in X, y\in Y$.
\end{itemize}
The conditions on the family $(m_{xy}^F)_{x, y}$ can also be stated as follows: $m_{xy}^F = m_{x', y'}^F$ if $(x, y)$ and $(x', y')$ are in the same orbit of the $G\times H$-set $X\times Y$ with action
\[ (g, h) \opl (x, y) := ((g, h)\opl x, \ (g, h)\opl y) \qquad (g\in G, h\in H, x\in X, y\in Y). \]

As in Lemma \ref{lem:modf_vectgo}, each choice of such data defines a $\monos$-bimodule functor $F:\mathcal M\to \mathcal N$.

In Proposition \ref{pro:vecgo_bimodc} we proved that $\monos$-bimodule categories are in bijection with $\mathrm{Vec}_{G\times H}^{\omega}$-module categories, where
\[ \omega((g_1, h_1), (g_2, h_2), (g_3, h_3)) = \omega_G(g_1, g_2, g_3) \, \omega_H^{-1}(h_3^{-1}, h_2^{-1}, h_1^{-1}). \]
We now show that this bijection is compatible with (bi-)module functors. Write $\Fun_{(\mathcal C, \mathcal D)}(\mathcal M, \mathcal N)$ for the category of $(\mathcal C, \mathcal D)$-bimodule functors $\mathcal M\to\mathcal N$, which is abelian and $\mathbb F$-linear by Proposition \ref{pro:modf_cat}.

\newcommand{\deliProdv}{\mathrm{Vec}_{G\times H}^{\omega}}

\begin{proposition}
	Let $\mathcal M:=\bimodcx$ and $\mathcal N:=\bimodcy$ be finite semisimple $\monos$-bimodule categories and let $\mathcal M(X, \Gamma_X)$ and $\mathcal M(Y, \Gamma_Y)$ be the associated $\deliProdv$-module categories from Proposition \ref{pro:vecgo_bimodc}.
	Then there is an $\mathbb F$-linear isomorphism of categories
	\begin{formarray}{rcl}
		\Fun_{\monos}(\mathcal M, \mathcal N) &\cong &\Fun_{\deliProdv}(\mathcal M(X, \Gamma_X), \mathcal M(Y, \Gamma_Y)) \\
		F &\mapsto &\tilde F \\
		K' &\text{\reflectbox{$\mapsto$}} &K,
	\end{formarray}
	where $\tilde F$ and $K'$ are defined by
	\begin{formarray}{ll}[1.5]
		m_{xy}^{\tilde F} := m_{xy}^F, &A_{(g, h), x, y}^{\tilde F} := A_{g, x, y}^F \, B_{h^{-1}, g\opl x, g\opl y}^F \\
		m_{xy}^{K'} := m_{xy}^K, &A_{g, x, y}^{K'} := A_{(g, 1), x, y}^K \, , 
		\quad B_{h, x, y}^{K'} := A_{(1, h^{-1}), x, y}^K
	\end{formarray}
	for $g\in G, h\in H, x\in X, y\in Y$.
\end{proposition}
\begin{proof}
	One checks that $\tilde F$ is well-defined by using the properties of the families $(m_{xy}^F)_{x, y}$, $(A_{g, x, y}^F)_{g, x, y}$, $(B_{h, x, y}^F)_{h, x, y}$ as well as the definition of $\Gamma_X$ and $\Gamma_Y$. It is straightforward to check that $K'$ is well-defined.
	
	We define the isomorphism of categories on morphisms of $\monos$-bimodule functors by
	\begin{equation} \label{eq:eta_tilde} \eta: F \Rightarrow K \quad \mapsto\quad \tilde\eta: \tilde F\Rightarrow \tilde K, \end{equation}
	\[ M^{\tilde\eta}_{x, y} := M_{x, y}^\eta \]
	and on morphisms of $\deliProdv$-module functors by
	\begin{equation} \label{eq:eta_prime} \eta: F \Rightarrow K \quad \mapsto\quad \eta': F' \Rightarrow K', \end{equation}
	\[ M_{x, y}^{\eta'} := M_{x, y}^\eta. \]
	These assignments are obviously $\mathbb F$-linear and it is easy to check that $\tilde \eta$ and $\eta'$ are well-defined.
	Note that a morphism of $\monos$-bimodule functors $\eta: F\Rightarrow K$ needs to satisfy the condition
	\[ M_{x, y}^{\eta} \, B_{h, x, y}^F = B_{h, x, y}^K \, M_{h^{-1}\opl x, h^{-1}\opl y}^{\eta} \qquad (h\in H, x\in X, y\in Y) \]
	in addition to (\ref{eq:cond_M}).
	
	We check that (\ref{eq:eta_tilde}) and (\ref{eq:eta_prime}) define functors between the categories $\Fun_{\monos}(\mathcal M, \mathcal N)$ and $\Fun_{\deliProdv}(\mathcal M(X, \Gamma_X), \mathcal M(Y, \Gamma_Y))$. The identity morphisms $\ident_F: F\Rightarrow F$ in both categories are represented by a family of identity matrices
	\[ M_{x, y}^{\ident_F} = \mathbb 1_{m_{xy}^F} \qquad (x\in X, y\in Y), \]
	so $\ident_F' = \ident_{F'}$ and $\tilde{\ident_F} = \ident_{\tilde F}$.
	Composition of (bi-)module natural transformations $\eta: F\Rightarrow K$ and $\mu: K \Rightarrow L$ is described by the formula
	\[ M_{x, y}^{\mu \circ \eta} = M_{x, y}^{\mu} \, M_{x, y}^{\eta} \qquad (x\in X, y\in Y), \]
	so the assignments (\ref{eq:eta_tilde}) and (\ref{eq:eta_prime}) define functors.
	It is easy to check that these functors are inverses of each other.
\end{proof}

\subsection{Examples of $\vectgo$-module functors}
In some cases, the matrices $(A_{g, x, y}^F)_{g, x, y}$ that determine a $\vectgo$-module functor can be represented elegantly by a certain functor. We now introduce the notation needed to state this correspondence.

\begin{definition}
	\cite[Example 1.5.18]{riehl17}
	Let $G$ be a group and let $X$ be a $G$-set. The \emph{action groupoid} $\mathcal C_X$ of $X$ is the category whose objects are elements $x$ of $X$ and whose morphisms are pairs
	\[ (g, x): x \to g\opl x \qquad (g\in G, x\in X). \]
	The composite of two morphisms is
	\[ (g, x) \circ (h, x') = (gh, x') \]
	and is defined if and only if $x = h \opl x'$.
\end{definition}

From here on, we will regard $X\times Y$ as a $G$-set with the diagonal action
\[ g \opl (x, y) := (g\opl x, g\opl y)\qquad (g\in G, x\in X, y\in Y). \]
Furthermore, we write $\vectfin$ for the category of finite-dimensional $\mathbb F$-vector spaces.

\begin{example}\label{ex:action_groupoid}
	Let $\mathcal M(X, \Psi_X)$ and $\mathcal M(Y, \Psi_Y)$ be finite semisimple $\vectgo$-module categories such that at least one of the following holds:
	\begin{itemize} 
		\item $\Psi_X$ and $\Psi_Y$ are trivial, i.e., $\Psi_X(g, h, x) = \Psi_Y(g, h, y) = 1$ for all $g\in G, h\in H, x\in X, y\in Y$ (this implies $\omega=1$), or
		\item $X=Y=\{\bullet\}$ is the trivial $G$-set with one element and $\Psi_X=\Psi_Y$. 
	\end{itemize}
	Then there is an $\mathbb F$-linear isomorphism of categories
	\[ \Fun_{\vectgo}(\mathcal M(X, \Psi_X), \mathcal M(Y, \Psi_Y)) \cong \Fun(\mathcal C_{X\times Y}^{op}, \vectfin). \]
\end{example}
\begin{proof}
	In both cases, condition (\ref{eq:cond_A}) simplifies to
	\begin{equation} \label{eq:cond_A_simp} A_{gh, x, y}^F = A_{h, x, y}^F \, A_{g, h\opl x, h\opl y}^F. \end{equation}
	For a functor $T: \mathcal C_{X\times Y}^{op} \to \vectfin$, define a $\vectgo$-module functor $\modfxy$ by
	\begin{formarray}{ll} &m_{xy}^F := \dim_{\mathbb F}(T((x, y))) \in \mathbb N, \\
	&A_{g, x, y}^F := T((g, (x, y))) \in \Mat(m_{g\opl x, g\opl y}^F\times m_{xy}^F, \mathbb F) \qquad (g\in G, x\in X, y\in Y), \end{formarray}
	\hspace*{-0.5em} where we again identified $\Hom_{\vectfin}(\mathbb F^n, \mathbb F^m)$ and $\Mat(m\times n, \mathbb F)$. Since every morphism in $\mathcal C_{X\times Y}$ is an isomorphism, $A_{g, x, y}^F$ is invertible and $m_{xy}^F = m_{g\opl x, g\opl y}^F$.
	Furthermore, compatibility of $T$ with composition implies that the family $(A_{g, x, y}^F)_{g, x, y}$ satisfies (\ref{eq:cond_A_simp}).
	Now let $\modfxy$ be a $\vectgo$-module functor. Then
	\begin{align*}
		T((x, y)) & := \mathbb F^{\oplus m^F_{xy}}, \\
		T((g, (x, y))) & := A^F_{g, x, y} \qquad (g\in G, x\in X, y\in Y)
	\end{align*}
	defines a functor $T:\mathcal C_{X\times Y}^{op} \to \vectfin$, as can be verified easily. It is straightforward to check that these constructions are mutually inverse and give rise to an $\mathbb F$-linear isomorphism of categories.
\end{proof}

We now reformulate the second case of Example \ref{ex:action_groupoid}, which generalises \cite[Example 7.12.19]{egno}.
Let $\mathrm BG$ be the category with a single object, elements of $G$ as morphisms and group multiplication as composition. Write $G^{op}$ for the opposite group of $G$ defined in Section \ref{ssec:vecgo_bimodc}. Then $(\mathrm BG)^{op}$ is isomorphic to $\mathrm B(G^{op})$.
It is well-known that $\Fun(\mathrm B(G^{op}), \vectfin)$ is isomorphic to the category $\mathrm{Rep}(G^{op})$ of representations of $G^{op}$ as an $\mathbb F$-linear category.
	
\begin{example} \label{ex:modf_repr}
	Suppose $\omega=1\in Z^3(G, \units)$, let $X=Y=\{\bullet\}$ be the trivial $G$-set with one element and let $\Psi_X = \Psi_Y = \Psi\in C^2(G, \Map(\{\bullet\}, \units))$ with $\dif \Psi = \omega^{-1}$. 
	
	Then $\mathcal C_{X\times Y} = \mathrm BG$, so we have an $\mathbb F$-linear isomorphism of categories
	\[ \Fun_{\mathcal C}(\mathcal M(\{\bullet\}, \Psi), \, \mathcal M(\{\bullet\}, \Psi)) \cong \mathrm{Rep}(G^{op}). \]
\end{example}

Example \ref{ex:modf_repr} shows that, in general, classifying all simple $\vectgo$-module functors between finite semisimple $\vectgo$-module categories is difficult since this classification would include a classification of the simple representations of $G$.

We now generalise Example \ref{ex:modf_repr} by dropping the condition $\Psi_X = \Psi_Y$. It will turn out that $\vectgo$-module functors $\mathcal M(\{\bullet\}, \Psi_X) \to \mathcal M(\{\bullet\}, \Psi_Y)$ are classified by certain projective representations. 

\begin{definition}
	\cite[Example 7.4.9]{egno}
	Let $G$ be a group and let $\psi\in Z^2(G, \units)$ be a 2-cocycle. A \emph{projective representation} of $G$ with \emph{Schur multiplier} $\psi$ consists of an $\mathbb F$-vector space $V$ and a map $\rho: G \to \glmat(V)$ that satisfies
	\[ \rho(g) \circ \rho(h) = \psi(g, h)\, \rho(gh) \]
	for all $g, h\in G$. Projective representations of $G$ with Schur multiplier $\psi$ form an $\mathbb F$-linear category, which will be denoted $\Rep_{\psi}(G)$.
\end{definition}

\begin{example}
	As before, let $\mathcal M(X, \Psi_X)$ and $\mathcal M(Y, \Psi_Y)$ be finite semisimple $\vectgo$-module categories. Assume $X=Y=\{\bullet\}$, but let $\Psi_X$ and $\Psi_Y$ be arbitrary. Then there is an $\mathbb F$-linear isomorphism of categories
	\[ \Fun_{\vectgo}(\mathcal M(X, \Psi_X), \mathcal M(Y, \Psi_Y)) \cong \mathrm{Rep}_{\psi}(G^{op}). \]
	for a certain $\psi\in Z^2(G^{op}, \units)$.
\end{example}
\begin{proof}
	By Proposition \ref{prop:indecomp}, $\mathcal M(X, \Psi_X) = \mathcal M'(G, \psi_X)$ and $\mathcal M(Y, \Psi_Y) = \mathcal M'(G, \psi_Y)$, where $\psi_X$ and $\psi_Y$ are the 2-cocycles defined by
	\[ \psi_U (g, h) := \Psi_U (g, h, \bullet) \qquad (g, h\in G, U\in \{X, Y\}). \]
	Hence (\ref{eq:cond_A}) is, in this case, equivalent to
	\[ A_{gh, \bullet, \bullet}^F = \psi_X(g, h)\, \psi_Y^{-1}(g, h) \ A_{h, \bullet, \bullet}^F \, A_{g, \bullet, \bullet}^F \qquad (g, h\in G). \]
	This can be rewritten as
	\begin{equation} \label{eq:proj_repr} A_{h, \bullet, \bullet}^F \, A_{g, \bullet, \bullet}^F = \psi(h, g) \  A_{gh, \bullet, \bullet}^F, \end{equation}
	where $\psi\in Z^2(G^{op}, \units)$ is defined by
	\[ \psi(g, h) := \psi_X^{-1}(h, g) \, \psi_Y(h, g) \quad (g, h\in G). \]
	It is easy to check that $\psi$ is indeed a cocycle.
	By (\ref{eq:proj_repr}), $V:= \mathbb F^{\oplus m_{\bullet, \bullet}^F}$ and $\rho(g) := A_{g, \bullet, \bullet}^F$ defines a projective representation $(V, \rho)$ of $G^{op}$ with Schur multiplier $\psi$. As in Example \ref{ex:action_groupoid}, this gives rise to an $\mathbb F$-linear isomorphism of categories.
\end{proof}

\subsection{Properties of $\vectgo$-module functors} \label{ssec:vecgo_modf_prop}
Let $\mathcal M(X, \Psi_X)$ and $\mathcal M(Y, \Psi_Y)$ be finite semisimple $\vectgo$-module categories. We have already seen in Example \ref{ex:action_groupoid} that the $G$-set $X\times Y$ plays an important role in describing the $\vectgo$-module functors $\modfxy$. It is also apparent that the numbers $m_{xy}^F$ only depend on the orbit of $(x, y)\in X\times Y$ and that condition (\ref{eq:cond_A})
\[ A_{gh, x, y}^F = \Psi_X(g, h, (gh)\opl x)\, \Psi_Y^{-1}(g, h, (gh)\opl y) \ A_{h, x, y}^F \, A_{g, h\opl x, h\opl y}^F \]
only relates matrices $A_{g, x, y}^F$ with $(x, y)$ from the same orbit. These facts give rise to a decomposition of $\vectgo$-module functors into direct sums, as will be made precise in the following.

\newcommand{\blockfun}[1][\Gamma]{\Fun_{#1}(\mathcal M(X, \Psi_X), \mathcal M(Y, \Psi_Y))}

First, we describe the direct sum of $\vectgo$-module functors by applying the construction from the proof of Proposition \ref{pro:modf_cat}.

\begin{lemma} \label{lem:sum_vecgo_modf}
	The direct sum of $(F_1, s_1), \dots, (F_n, s_n) \in \Fun_{\vectgo}(\mathcal M(X, \Psi_X), \mathcal M(Y, \Psi_Y))$ is the module functor $\modfxy$ defined by
	\[ m_{xy}^F := \sum_{i=1}^n m_{xy}^{F_i}, \]
	\[ A_{g, x, y}^F := 
	\begin{bmatrix}
		A_{g, x, y}^{F_1} & & 0 \\ & \ddots & \\ 0 & & A_{g, x, y}^{F_n}
	\end{bmatrix}. \]
\end{lemma}
\begin{proof}
	It is clear that the functor defined by $(m_{xy}^F)_{x, y}$ is the direct sum of the functors $F_i$ in the sense that
	\[ F(x) = \bigoplus_{i=1}^n F_i(x) \]
	for $x\in X$. As for the matrices $(A_{g, x, y}^F)_{g, x, y}$, note that we implicitly chose
	\[ j_{Fxy}^{\alpha} := \iota_x^{(i_{\alpha})} \circ j_{F_{i_{\alpha}}, x, y}^{\alpha'}: y \to F_{i_{\alpha}}(x) \to F(x), \quad 
	p_{Fxy}^{\alpha} := p_{F_{i_{\alpha}}, x, y}^{\alpha'} \circ \pi_x^{(i_{\alpha})}: F(x) \to F_{i_{\alpha}}(x) \to y \]
	as inclusions and projections for $F(x)$, where 
	\[ \iota_x^{(i)}: F_i(x) \to F(x), \quad \pi_x^{(i)}: F(x)\to F_i(x) \qquad (x\in X, i=1, ..., n) \] 
	are inclusions and projections for $F(x)$ and 
	\[ i_{\alpha} := \min\left\{ i\in \{1, \dots, n\} \ \middle| \ \sum_{k=1}^{i} m_{xy}^{F_k} \geq \alpha \right\}, \quad \alpha' := \alpha - \sum_{k=1}^{i_{\alpha}-1} m_{xy}^{F_k}. \]
	Now we can calculate
	\begin{align*} 
		\left( A_{g, x, y}^F \right)_{\beta}^{\alpha} &= (\ident_g \opl p_{Fxy}^{\alpha}) \circ s_{g, x} \circ j_{F, g\opl x, g\opl y}^{\beta} \\
		&= (\ident_g \opl p_{F_{i_{\alpha}}, x, y}^{\alpha'}) 	\circ (\ident_g \opl \pi_x^{(i_{\alpha})}) \circ s_{g, x} \circ \iota_{g\opl x}^{(i_{\beta})} \circ j_{F_{i_{\beta}}, g\opl x, g\opl y}^{\beta'} \\
		&\overset{\ref{pro:modf_cat}} = \delta_{i_{\alpha}, i_{\beta}} \ (\ident_g \opl p_{F_{i_{\alpha}}, x, y}^{\alpha'}) \circ \left(s_{i_{\alpha}}\right)_{g, x} \circ j_{F_{i_{\beta}}, g\opl x, g\opl y}^{\beta'} = \delta_{i_{\alpha}, i_{\beta}} \ \left( A_{g, x, y}^{F_{i_{\alpha}}} \right)^{\alpha'}_{\beta'}.
	\end{align*}
	This completes the proof.
\end{proof}

We now prove that the category of $\vectgo$-module functors $\mathcal M(X, \Psi_X)\to \mathcal M(Y, \Psi_Y)$ can be decomposed into subcategories which are parametrised by the orbits of $X\times Y$.

\begin{definition}
	For every orbit $\Gamma\in (X\times Y)/G$, define the $\mathbb F$-linear abelian category $\Fun_{\Gamma}(\mathcal M(X, \Psi_X), \mathcal M(Y, \Psi_Y))$ as the full subcategory of $\Fun_{\vectgo}(\mathcal M(X, \Psi_X), \mathcal M(Y, \Psi_Y))$ consisting of all $\vectgo$-module functors $\modfxy$ with $m_{xy}^F = 0$ whenever $(x, y)\notin \Gamma$.
\end{definition}

\begin{proposition} \label{prop:block_decomp}
	There exists an exact $\mathbb F$-linear equivalence
	\[ \Fun_{\vectgo}(\mathcal M(X, \Psi_X), \mathcal M(Y, \Psi_Y)) \cong \bigoplus_{\Gamma \in (X\times Y)/G} \Fun_{\Gamma}(\mathcal M(X, \Psi_X), \mathcal M(Y, \Psi_Y)). \]
\end{proposition}
\begin{proof}
	For a $\vectgo$-module functor $\modfxy$ define module functors 
	\( F_{\Gamma} \in \Fun_{\Gamma}(\mathcal M(X, \Psi_X), \mathcal M(Y, \Psi_Y)) \) by
	\[ m_{xy}^{F_{\Gamma}} := \left\{ \begin{array}{ll} m_{xy}^F, & (x, y)\in \Gamma \\ 0, & \text{ otherwise} \end{array} \right. , \qquad A_{g, x, y}^{F_{\Gamma}} := \left\{ \begin{array}{ll} A_{g, x, y}^F, & (x, y)\in \Gamma \\ \left[\right], & \text{ otherwise} \end{array} \right. \]
	for $\Gamma\in (X\times Y)/G, g\in G, x\in X$ and $y\in Y$, where $[ ]$ denotes the $0\times 0$-matrix. It is straightforward to check that $F_{\Gamma}$ is well-defined and Lemma \ref{lem:sum_vecgo_modf} implies that $F = \bigoplus_{\Gamma \in (X\times Y)/G}\, F_{\Gamma}$. Furthermore,
	\( \Hom(H_{\Gamma}, H_{\Gamma'}) = 0 \)
	whenever 
	\[ H_{\Gamma} \in \blockfun,\, H_{\Gamma'}\in\blockfun[\Gamma'] \text{ and } \Gamma \neq \Gamma'. \]
	This concludes the proof.
\end{proof}

\begin{corollary} \label{cor:simp_modf}
	Every simple $\vectgo$-module functor $\modfxy$ is contained in $\blockfun$ for some $\Gamma \in (X\times Y)/G$.
\end{corollary}

We now determine the right adjoint of a $\vectgo$-module functor $\modfxy$.
Recall from Section \ref{ssec:modf_prop} the definition (\ref{eq:modf_adj}) of the coherence datum for the right adjoint module functor $(F^r, s'):\mathcal M(Y, \Psi_Y)\to\mathcal M(X, \Psi_X)$.
To determine matrices $(A_{g, x, y}^{F^r})_{g, x, y}$ which describe this coherence datum, we need to choose inclusions and projections for $F^r(x)$. For $x\in X, y\in Y$, let
\[ \phi_{x,y}: \Hom(F(x), y) \overset{\sim}{\longrightarrow} \Hom(x, F^r(y)): f \mapsto F^r(f) \circ \eta_x, \]
where $\eta: \Ident_{\mathcal M(X, \Psi_X)} \Rightarrow F^r F$ is the unit of the adjunction $F \dashv F^r$. Since $F^r$ is $\mathbb F$-linear and composition is bilinear, $\phi_{x, y}$ is $\mathbb F$-linear.
We now define inclusions for $F^r(x)$ as
\[ j_{F^r, y, x}^{\alpha} := \phi_{x, y}(p_{F, x, y}^{\alpha}):\ x\to F^r(y) \qquad (x\in X, y\in Y, \alpha=1,\dots,m_{xy}^F). \]
Since they form a basis of $\Hom(y, F^r(x))$, there exist projections $(p_{F^r, y, x}^{\alpha})_{y, x, \alpha}$ with
\[ p_{F^r,y,x}^{\alpha} \circ j_{F^r,y,x'}^{\beta} = \delta_{x, x'}\, \delta_{\alpha \beta}\, \ident_x\, , \quad \sum_{x\in X} \sum_{\alpha=1}^{m_{xy}^F} j_{F^r,y,x}^{\alpha} \circ p_{F^r,y,x}^{\alpha} = \ident_{F^r(y)}. \]

\begin{proposition}
	The right and left adjoint of a $\vectgo$-module functor $\modfxy$ are both given by the $\vectgo$-module functor $(F^r, s'): \mathcal M(Y, \Psi_Y)\to \mathcal M(X, \Psi_X)$, which is defined by
	\[ m_{yx}^{F^r} = m_{xy}^F, \]
	\[ A_{g, y, x}^{F^r} = \Psi_X(g, g^{-1}, g\opl x) \, \Psi_Y^{-1}(g, g^{-1}, g\opl y) \, \left(A_{g^{-1}, g\opl x, g\opl y}^F\right)^T. \]
	Here, $A^T$ denotes the transpose of a matrix $A$.
\end{proposition}
\begin{proof}
We first establish that $(F^r, s')$ is right adjoint to $(F, s)$. Since $F^r$ is right adjoint to $F$ as a functor, we have
\[ m_{yx}^{F^r} = \dim_{\mathbb F}(\Hom(y, F^r(x))) = \dim_{\mathbb F}(\Hom(F(y), x)) = m_{xy}^F. \]
Substitute $M:=x, C:=\delta^g$ and $N:=y$ in (\ref{eq:modf_adj}). Then the chain of isomorphisms in (\ref{eq:modf_adj}) defines the map
\[ \varphi_{g, y}:= s'_{g, y} \circ -:\ \Hom(x, F^r(\delta^g \opl y)) \to \Hom(x, \delta^g \opl F^r(y)). \]
We now calculate the matrix entries of $A_{g, y, x}^{F^r}$. We have
\[ \left( A^{F^r}_{g, y, x} \right)^{\alpha}_{\beta} = (\ident_g \opl p_{F^r, y, x}^{\alpha}) \circ s'_{g, y} \circ j_{F^r, g\opl y, g\opl x}^{\beta} = (\ident_g \opl p_{F^r, y, x}^{\alpha}) \circ \varphi_{g, y} (j_{F^r, g\opl y, g\opl x}^{\beta}). \]
Set $q_{g, x, y}^{\alpha} := \varphi_{g, y}(j_{F^r, g\opl y, x}^{\alpha})$.
By writing out the isomorphisms in (\ref{eq:modf_adj}) explicitly, we get
\begin{align*}  q_{g, x, y}^{\alpha} &= \left( \ident_g \opl 
\phi_{g^{-1}\opl x, y}\left( 
	(\eval_g^R \opl \ident_y) \circ m_{g^{-1}, g, y}^{-1} \circ \left( \ident_{g^{-1}} \opl \phi_{x, g\opl y}^{-1}(j_{F^r, g\opl y, x}^{\alpha})\right) \circ s_{g^{-1}, x}
\right) \right) \\
&\hspace*{2em} \circ m_{g, g^{-1}, x}\circ (\coev_g^R \opl \ident_x)\\
&= \omega(g^{-1}, g, g^{-1})\, \Psi_X(g, g^{-1}, x) \, \Psi_Y^{-1}(g^{-1}, g, y) \, \sum_{\beta=1}^{m_{xy}^F} \left( A_{g^{-1}, x, g\opl y}^F \right)_{\alpha}^{\beta} \, (\ident_g \opl j_{F^r, y, g^{-1} \opl x}^{\beta}).
\end{align*}
Once more, we abused notation and wrote $m$ for the module constraint of $\mathcal M(X, \Psi_X)$ and of $\mathcal M(Y, \Psi_Y)$.
We now apply the coboundary condition (\ref{eq:2cocycle}) for $\Psi_Y$ and $(g^{-1}, g, g^{-1}, y)$ and obtain
\[ q_{g, x, y}^{\alpha} = \Psi_X(g, g^{-1}, x) \, \Psi_Y^{-1}(g, g^{-1}, g\opl y) \, \sum_{\beta=1}^{m_{xy}^F} \left( A_{g^{-1}, x, g\opl y}^F \right)_{\alpha}^{\beta} \, (\ident_g \opl j_{F^r, y, g^{-1} \opl x}^{\beta}). \]
As a result,
\[ \left(A_{g, y, x}^{F^r}\right)^{\alpha}_{\beta} = (\ident_g \opl p_{F^r, y, x}^{\alpha}) \circ q_{g, g\opl x, y}^{\beta} = \Psi_X(g, g^{-1}, g\opl x) \, \Psi_Y^{-1}(g, g^{-1}, g\opl y) \, \left( A_{g^{-1}, g\opl x, g\opl y}^F \right)^{\beta}_{\alpha} \]
\[ \Rightarrow \quad A_{g, y, x}^{F^r} = \Psi_X(g, g^{-1}, g\opl x) \, \Psi_Y^{-1}(g, g^{-1}, g\opl y) \, \left(A_{g^{-1}, g\opl x, g\opl y}^F\right)^T. \]

Replacing $F$ by $F^r$ in the formulae for $m_{yx}^{F^r}$ and $A_{g, y, x}^{F^r}$, one easily checks that $(F^r)^r \cong F$ as a module functor. Hence $F^r$ is left adjoint to $F$.
\end{proof}

\subsection
{Classification of simple module functors for $G=\mathbb Z/n \mathbb Z$}
We now use the explicit description of $\vectgo$-module functors from Section \ref{ssec:vecgo_modf} to classify the simple $\vectgo$-module functors between finite semisimple module categories in the case where $G$ a finite cyclic group.
Throughout this section assume that $G=\mathbb Z/n \mathbb Z$ is the cyclic group of order $n\in\mathbb N\backslash \{ 0 \}$.
We denote the elements of $\mathbb Z/n \mathbb Z$ by $\bar k$, where $k\in \mathbb Z$.

The case $G=\mathbb Z/n\mathbb Z$ is also of particular interest in the context of Dijkgraf-Witten theory because there exists an explicit formula for all 3-cocycles $\omega\in Z^3(G, \units)$ up to cohomology in this case. We have $H^3(G, \units) = \mathbb Z/n \mathbb Z$ and
\[ \omega_{\bar s}(\bar k, \bar l, \overline m) = \varepsilon^{\frac{sk(l+m-(l+m)')}{n}} \qquad \left( \bar k, \bar l, \overline m\in \mathbb Z/n \mathbb Z \right) \]
defines a set of representatives $\{ \omega_{\bar s} \mid \bar s\in \mathbb Z/n \mathbb Z \}$ for the cohomology classes, where $\varepsilon$ is a primitive $n$-th root of unity and $z'$ denotes the remainder of division by $n$ for all $z\in\mathbb Z$ \cite[Example 2.6.4]{egno}.

\begin{lemma} \label{lem:simp_modf}
	Let $\modfxy$ be a simple $\vectgo$-module functor. Then there exists an orbit $\Gamma \in (X\times Y)/G$ with
	\[ m_{xy}^F = \left\{ \begin{array}{ll} 1, & (x, y)\in \Gamma \\ 0, & \text{ otherwise}. \end{array} \right. \]
\end{lemma}
\begin{proof}
	By Corollary \ref{cor:simp_modf}, there exists an orbit $\Gamma \in (X\times Y)/G$ such that $m_{xy}^F = 0$ if $(x, y)\notin \Gamma$. The category $\blockfun[\vectgo]$ is semisimple by Proposition \ref{pro:modfc_semi}, so it follows from Schur's Lemma \ref{lem:schur} that $(F, s)$ is simple if and only if
	\[ \Hom((F, s), (F, s)) \cong \mathbb F. \]
	By Lemma \ref{lem:vectgo_mor_modf}, morphisms $\eta: (F, s) \Rightarrow (F, s)$ of $\vectgo$-module functors are in bijection with families of matrices $(M_{x, y}^{\eta})_{x, y}$ that satisfy
	\[ M_{x, y}^{\eta}\, A^F_{\bar k, x, y} = A^F_{\bar k, x, y} \, M^{\eta}_{\bar k\opl x, \bar k\opl y} \qquad (x\in X, y\in Y, \bar k \in G). \]
	Fix an element $(x, y)\in \Gamma$ and write $r:=|\Gamma|$ for the length of $\Gamma$. Note that $\bar r\in G$ generates the stabiliser subgroup of $(x, y)$. For every $(x', y')\in\Gamma$ there is some $\bar k\in G$ with $(x', y')=(\bar k\opl x, \bar k\opl y)$, so $\eta$ is uniquely determined by the single matrix $M_{x, y}^{\eta}$:
	\[ M_{\bar k \opl x, \bar k \opl y}^{\eta} = \left(A_{\bar k, x, y}^F\right)^{-1}\, M_{x, y}^{\eta}\, A_{\bar k, x, y}^F. \]
	A direct calculation shows that a given matrix $M_{x, y} \in \Mat(m_{xy}^F\times m_{xy}^F, \mathbb F)$ determines an endomorphism $\eta: (F, s) \Rightarrow (F, s)$ via $M_{\bar k \opl x, \bar k \opl y}^{\eta} := (A_{\bar k, x, y}^F)^{-1}\, M_{x, y}\, A_{\bar k, x, y}^F$ if and only if
	\[ M_{x, y} = \left(A_{\bar r, x, y}^F\right)^{-1}\, M_{x, y}\, A_{\bar r, x, y}^F, \]
	i.e. iff $M_{x, y}$ commutes with $A_{\bar r, x, y}^F$. Hence we have an $\mathbb F$-linear bijection
	\[ \Hom((F, s), (F, s)) \cong \left\{ M\in \Mat(m_{xy}^F\times m_{xy}^F, \mathbb F) \ \middle| \ MA_{\bar r, x, y}^F = A_{\bar r, x, y}^F\, M\right\}. \]
	Assume $m_{xy}^F \geq 2$. If $A_{\bar r, x, y}^F$ is a multiple of the identity matrix, every $2\times 2$-matrix commutes with $A_{\bar r, x, y}^F$ and therefore
	\[ 1 = \dim_{\mathbb F}(\Hom((F, s), (F, s))) = \dim_{\mathbb F}\left(\Mat(m_{xy}^F\times m_{xy}^F, \mathbb F)\right) \geq 4, \]
	a contradiction. If $A_{\bar r, x, y}^F$ is not a multiple of the identity matrix, then there are two linearly independent matrices that commute with $A_{\bar r, x, y}^F$: the identity matrix and $A_{\bar r, x, y}^F$ itself. Hence $1=\Hom((F, s), (F, s)) \geq 2$, a contradiction. It follows that $m_{xy}^F = 1$.
\end{proof}

We can now parametrise all simple $\vectgo$-module functors up to isomorphism:

\begin{definition} \label{def:vecgo_modf_class}
	For every $\Gamma \in (X\times Y)/G$, choose $(x, y)\in\Gamma$, write $r:=|\Gamma|$ and let $\Xi_{\Gamma} \subset \units$ be the set of $n$-th roots of
	\[ \gamma_{x, y} = \prod_{t=1}^{n-1} \Psi_X^{-1}(\bar 1, \bar t, \overline{t+1} \opl x)\, \Psi_Y(\bar 1, \bar t, \overline{t+1}\opl y) \in \units. \]
	Furthermore, we define an equivalence relation $\sim$ on $\Xi_{\Gamma}$ by
	\[ \xi_1 \sim \xi_2 \, :\Leftrightarrow\, \xi_1^r = \xi_2^r \]
	and write $[\xi]$ for the equivalence class of $\xi$ with respect to $\sim$.
\end{definition}

\begin{theorem}
	The isomorphism classes of simple $\vectgo$-module functors $\mathcal M(X, \Psi_X) \to \mathcal M(Y, \Psi_Y)$ are in bijection with pairs $(\Gamma, [\xi])$ of an orbit $\Gamma\in (X\times Y)/G$ and $[\xi]\in \Xi_{\Gamma}/\mathbin\sim$. In particular, if the characteristic of $\mathbb F$ is zero, the number of such isomorphism classes is
	\[ \sum_{\Gamma\in (X\times Y)/G} \frac n {|\Gamma|}. \]
\end{theorem}
\begin{proof}
	For every orbit $\Gamma \in (X\times Y)/G$ and $[\xi]\in \Xi_{\Gamma}/\mathbin\sim$, define a $\vectgo$-module functor
	$F_{\Gamma, [\xi]}: \mathcal M(X, \Psi_X) \to \mathcal M(Y, \Psi_Y)$
	by
	\[ m_{xy}^{F_{\Gamma, [\xi]}} = \left\{ \begin{array}{ll} 1, & (x, y)\in \Gamma \\ 0, & \text{ otherwise} \end{array} \right. , \]
	\[ A_{\bar k, x, y}^{F_{\Gamma, [\xi]}} = \xi^k \prod_{t=1}^{k-1} \Psi_X (\bar 1, \bar t, \overline{t+1}\opl x)\, \Psi_Y^{-1}(\bar 1, \bar t, \overline{t+1}\opl y) \in \units \cong \glmat(1, \mathbb F) \quad (k\geq 1, x, y\in \Gamma). \]
	
	1. We prove that $F_{\Gamma, [\xi]}$ is well-defined. To improve readability, write $A_{\bar k, x, y}$ instead of $A_{\bar k, x, y}^{F_{\Gamma, [\xi]}}$. It is straightforward to check that the definition of $A_{\bar k, x, y}$ does not depend on the choice of the representative for $\bar k$.
	It remains to show that (\ref{eq:cond_A}) is satisfied for $(A_{\bar k, x, y})_{\bar k, x, y}$.
	The coboundary condition (\ref{eq:2cocycle}) for $\Psi_X$ and $(\bar 1, \bar t, \bar l, \overline{t+l+1} \opl x)$ reads
	\begin{multline*} \Psi_X(\overline{t+1}, \bar l, \overline{t+l+1}\opl x) \, \Psi_X(\bar 1, \bar t, \overline{t+l+1}\opl x) = \\ \omega(\bar 1, \bar t, \bar l) \, \Psi_X(\bar t, \bar l, \overline{t+l} \opl x) \, \Psi_X(\bar 1, \overline{t+l}, \overline{t+l+1} \opl x). \end{multline*}
	Together with the analogous equation for $\Psi_Y$, this implies
	\begin{align*}
		&\hphantom{{}={}}\Psi_X(\overline{t+1}, \bar l, \overline{t+l+1}\opl x) \, \Psi_Y^{-1}(\overline{t+1}, \bar l, \overline{t+l+1} \opl y) \\
		&\hphantom{{}={}}\Psi_X(\bar 1, \bar t, \overline{t+l+1} \opl x) \, \Psi_Y^{-1}(\bar 1, \bar t, \overline{t+l+1} \opl y) \\
		&= \Psi_X(\bar t, \bar l, \overline{t+l} \opl x) \,
		\Psi_Y^{-1}(\bar t, \bar l, \overline{t+l} \opl x) \\
		&\hphantom{{}={}}\Psi_X(\bar 1, \overline{t+l}, \overline{t+l+1} \opl x) \,
		\Psi_Y^{-1}(\bar 1, \overline{t+l}, \overline{t+l+1} \opl x).
	\end{align*}
	Applying this for $t=1, \dots, k-1$, we get
	\begin{align*}
		&\hphantom{{}={}} \Psi_X(\bar k, \bar l, \overline{k+l}\opl x) \, \Psi_Y^{-1}(\bar k, \bar l, \overline{k+l}\opl y) \ A_{\bar l, x, y} \, A_{\bar k, \bar l\opl x, \bar l\opl y} \\
		&=\Psi_X(\bar k, \bar l, \overline{k+l}\opl x) \, \Psi_Y^{-1}(\bar k, \bar l, \overline{k+l}\opl y) \,
		\xi^{k+l} \left( \prod_{t=1}^{l-1} \Psi_X(\bar 1, \bar t, \overline{t+1} \opl x) \, \Psi_Y^{-1}(\bar 1, \bar t, \overline{t+1}\opl y) \right) \\
		&\hphantom{{}= \Psi_X(\bar k, \bar l, \overline{k+l}\opl x) \,} \prod_{t=1}^{k-1} \Psi_X(\bar 1, \bar t, \overline{t+l+1}\opl x) \, \Psi_Y^{-1}(\bar 1, \bar t, \overline{t+l+1}\opl y) \\
		&=\Psi_X(\bar 1, \bar l, \overline{l+1}\opl x) \, \Psi_Y^{-1}(\bar 1, \bar l, \overline{l+1}\opl y) \,
		\xi^{k+l} \left( \prod_{t=1}^{l-1} \Psi_X(\bar 1, \bar t, \overline{t+1} \opl x) \, \Psi_Y^{-1}(\bar 1, \bar t, \overline{t+1}\opl y) \right) \\
		&\hphantom{{}= \Psi_X(\bar k, \bar l, \overline{k+l}\opl x) \,} \prod_{t=1}^{k-1} \Psi_X(\bar 1, \overline{t+l}, \overline{t+l+1}\opl x) \, \Psi_Y^{-1}(\bar 1, \overline{t+l}, \overline{t+l+1}\opl y) \\
		&=\xi^{k+l} \prod_{t=1}^{k+l-1} \Psi_X(\bar 1, \bar t, \overline{t+1}\opl x) \, \Psi_Y^{-1}(\bar 1, \bar t, \overline{t+1}\opl y) = A_{\overline{k+l}, x, y}.
	\end{align*}
	
	2. Let $\modfxy$ be a simple $\vectgo$-module functor. We prove that $F\cong F_{\Gamma, [\xi]}$ for some $\Gamma\in (X\times Y)/G$ and $\xi\in \Xi_{\Gamma}$. By Lemma \ref{lem:simp_modf}, there exists an orbit $\Gamma\in (X\times Y)/G$ with
	\[ m_{xy}^F = \left\{ \begin{array}{ll} 1, & (x, y)\in \Gamma \\ 0, & \text{ otherwise}. \end{array} \right. \]
	This allows us to treat $A_{\bar k, x, y}^F$ as a scalar whenever $(x, y)\in \Gamma$.
	By applying (\ref{eq:cond_A}) inductively, we get
	\begin{align*} A_{\bar k, x, y}^F &= \Psi_X(\bar 1, \overline{k-1}, \bar k \opl x)\, \Psi_Y^{-1}(\bar 1, \overline{k-1}, \bar k \opl y) \ A_{\overline{k-1}, x, y}^F \, A_{\bar 1, \overline{k-1}\opl x, \overline{k-1}\opl y}^F = \dots \\
	&= \prod_{t=0}^{k-1} \Psi_X(\bar 1, \bar t, \overline{t+1} \opl x)\, \Psi_Y^{-1}(\bar 1, \bar t, \overline{t+1} \opl y) \ A_{\bar 1, \bar t \opl x, \bar t \opl y}^F. \end{align*}
	In particular,
	\newcommand{\orbitprod}{\prod_{t=0}^{r-1} A_{\bar 1, \bar t \opl x, \bar t \opl y}}
	\begin{align*} 
		1 &= A_{\bar 0, x, y}^F = A_{\bar n, x, y}^F = \prod_{t=0}^{n-1} \Psi_X(\bar 1, \bar t, \overline{t+1} \opl x)\, \Psi_Y^{-1}(\bar 1, \bar t, \overline{t+1} \opl y) \ A_{\bar 1, \bar t \opl x, \bar t \opl y}^F \\ 
		&= \gamma_{x, y}^{-1} \left( \orbitprod \right)^{n/r} .
	\end{align*}
	Hence $\orbitprod$ is an $n/r$-th root of $\gamma_{x, y}$ and we can choose $\xi\in \Xi_{\Gamma}$ such that $\xi^r = \orbitprod$.
	Fix some $(x, y)\in \Gamma$. Then
	\[ M_{\bar k \opl x, \bar k \opl y}^{\eta} := \xi^{-k} \prod_{t=0}^{k-1} A_{\bar 1, \bar t \opl x, \bar t \opl y}^F \qquad (k\geq 1) \]
	is well-defined and defines an isomorphism of module functors $\eta: F \overset{\sim}{\Rightarrow} F_{\Gamma, [\xi]}$, as is straightforward to check.
	
	3. Let $\eta: F_{\Gamma, [\xi]} \Rightarrow F_{\Gamma', [\xi']}$ be an isomorphism of $\vectgo$-module functors. Then clearly $\Gamma = \Gamma'$ and
	\[ \xi^r (\xi')^{-r} = A_{\bar r, x, y}^{F_{\Gamma, [\xi]}} \, \left( A_{\bar r, x, y}^{F_{\Gamma, [\xi']}} \right)^{-1} = A_{\bar r, x, y}^{F_{\Gamma, [\xi]}} \, M_{\bar r \opl x, \bar r\opl y}^{\eta} \left(A_{\bar r, x, y}^{F_{\Gamma, [\xi]}}\right)^{-1} \left( M_{x, y}^{\eta} \right)^{-1} = 1. \]
	Hence $\xi\sim \xi'$.
\end{proof}
\clearpage
\section{Generalised 6j symbols}\label{sec:6j}
Let $\mathcal C$ and $\mathcal D$ be spherical fusion categories. All bimodule categories in this section are assumed to be finite, semisimple and equipped with a bimodule trace $\theta$. For every finite semisimple category $\mathcal A$, let $I_{\mathcal A}$ be a set of representatives for the simple objects of $\mathcal A$. Note that $I_{\mathcal A}$ is a finite set. Throughout this section, we abuse notation and do not distinguish between the bimodule traces of different bimodule categories, writing $\theta$ for all of them. Similarly, $m$, $n$ and $b$ may denote the coherence data of any bimodule category.

Up to this point, we analysed spherical fusion categories, bimodule categories and module functors. We now define the generalised 6j symbols from \cite{bm_tv+} that are associated to these structures, cite some of their basic properties and then calculate the generalised 6j symbols in the case $\mathcal C=\mathrm{Vec}_G^{\omega_G}, \mathcal D=\mathrm{Vec}_H^{\omega_H}$.

\subsection{Definition of generalised 6j symbols} \label{ssec:6j_def}

This section closely follows \cite{bm_tv+}.

\subsubsection{6j symbols for spherical fusion categories} \label{sssec:6j_fusion}
Let $i, j\in I_{\mathcal C}$. Since $\mathcal C$ is semisimple, $i\otimes j$ has a decomposition into simple objects
\[ i \otimes j \cong \bigoplus_{k\in I_{\mathcal C}} k^{\otimes m_{ijk}} \]
with certain multiplicities $m_{ijk}\in\mathbb N$. This is similar to the decomposition considered in Section \ref{ssec:vecgo_modf}. Again, we choose inclusions and projections
\[ j_{ijk}^{\alpha}: k \to i\otimes j, \quad p_{ijk}^{\alpha}: i\otimes j \to k \qquad (k\in I_{\mathcal C}, \alpha = 1, \dots, m_{ijk}) \]
for $i\otimes j$, which satisfy
\[ p_{ijk}^{\alpha} \circ j_{ijl}^{\beta} = \delta_{k, l} \, \delta_{\alpha, \beta} \, \ident_k\ , \qquad \sum_{k\in I_{\mathcal C}} \sum_{\alpha=1}^{m_{ijk}} j_{ijk}^{\alpha} \circ p_{ijk}^{\alpha} = \ident_{i\otimes j}. \]
We would like to rescale these morphisms with the inverse of the dimension of $k$. This is possible due to the following lemma, which is an immediate consequence of \cite[Proposition 4.8.4]{egno}.

\begin{lemma}
	Let $\mathcal C$ be a pivotal fusion category. Then $\dim_L(x) \neq 0$ and $\dim_R(x)\neq 0$ for every simple object $x\in \mathcal C$.
\end{lemma}

By rescaling $j_{ijk}^{\alpha}$ and $p_{ijk}^{\alpha}$ with $\dim(k)$, we get morphisms
\[ j_{ijk}^{\alpha}: k \to i\otimes j, \quad p_{ijk}^{\alpha}: i\otimes j \to k \qquad (k\in I_{\mathcal C}, \alpha=1, \dots, m_{ijk}) \]
that satisfy
\[ p_{ijk}^{\alpha} \circ j_{ijl}^{\beta} = \delta_{k, l} \, \delta_{\alpha, \beta} \, \dim^{-1}(k) \, \ident_k\ , \qquad \sum_{k\in I_{\mathcal C}} \sum_{\alpha=1}^{m_{ijk}} \dim(k)\, j_{ijk}^{\alpha} \circ p_{ijk}^{\alpha} = \ident_{i\otimes j}. \]
We call such families of morphisms \emph{rescaled inclusions and projections} for $i\otimes j$. We represent them diagrammatically by
\begin{center}
$j_{ijk}^{\alpha} \ \eqdiag \ $
\begin{tikzpicture}
	\newcommand{\lowerCoor}{0.7}
	\node[morphism] (a) {$\alpha$};
	\draw[very thick] (a.north) -- node[near end, left=1pt] {$k$} +(0, 1);
	\draw[very thick] (a.south west) -- node[near end, left=5pt] {$i$} +(-\lowerCoor, -\lowerCoor);
	\draw[very thick] (a.south east) -- node[near end, right=5pt] {$j$} +(\lowerCoor, -\lowerCoor);
\end{tikzpicture}
$, \qquad p_{ijk}^{\alpha} \ \eqdiag \ $
\begin{tikzpicture}
	\newcommand{\upperCoor}{0.7}
	\node[morphism] (a) {$\alpha$};
	\draw[very thick] (a.south) -- node[near end, left=1pt] {$k$} +(0, -1);
	\draw[very thick] (a.north west) -- node[near end, left=5pt] {$i$} +(-\upperCoor, \upperCoor);
	\draw[very thick] (a.north east) -- node[near end, right=5pt] {$j$} +(\upperCoor, \upperCoor);
\end{tikzpicture}
.
\end{center}

The associativity constraint of $\mathcal C$ is uniquely determined by the morphisms
\begin{multline} \label{eq:6j_fusion} p_{iab}^{\beta} \circ (\ident_i \otimes p_{jka}^{\alpha}) \circ a_{i, j, k} \circ (j_{ijc}^{\gamma} \otimes \ident_k) \circ j_{ckb}^{\delta} \overset{\ref{cor:trace_mult}} = \\ 
\dim^{-1}(b) \, \trace\left( 
	p_{iab}^{\beta} \circ (\ident_i \otimes p_{jka}^{\alpha}) \circ a_{i, j, k} \circ (j_{ijc}^{\gamma} \otimes \ident_k) \circ j_{ckb}^{\delta}
\right) \, \ident_b
\end{multline}
for $a, b, c, i, j, k\in I_{\mathcal C}$ and all $\alpha, \beta, \gamma, \delta$.
The trace in (\ref{eq:6j_fusion}) is called a \emph{6j symbol} and represented by the diagram
\begin{center}
$\trace\left( 
p_{iab}^{\beta} \circ (\ident_i \otimes p_{jka}^{\alpha}) \circ a_{i, j, k} \circ (j_{ijc}^{\gamma} \otimes \ident_k) \circ j_{ckb}^{\delta}
\right) \ \eqdiag \ $
\begin{tikzpicture}[x=8mm, y=8mm, every node/.style={inner sep=2pt}]
	\newcommand{\sca}{1}
	\node[morphism] (d) {$\delta$};
	\node[morphism, on grid, below left=\sca and \sca of d] (g) {$\gamma$};
	\node[morphism, on grid, below right=\sca and 2*\sca of g] (a) {$\alpha$};
	\node[morphism, on grid, below left=\sca and \sca of a] (b) {$\beta$};
	\draw (d) -- node[left=5pt, pos=0] {$c$} (g);
	\draw (d) -- node[right=3pt] {$k$} (a);
	\draw (g) -- node[left=3pt] {$i$} (b);
	\draw (g) -- node[above=1pt] {$j$} (a);
	\draw (a) -- node[pos=1, right=5pt] {$a$} (b);
	\draw (b.south) .. node[very near start, left=5pt] {$b$} controls +(0, -\sca) and +(0, -\sca) .. +(2.5*\sca, 0) coordinate (P);
	\draw (d.north) .. node[very near start, left=5pt] {$b$} controls +(0, \sca) and +(0, \sca) .. +(2.5*\sca, 0) coordinate (Q);
	\draw[->] (P) -- (Q);
\end{tikzpicture}
.
\end{center}

Likewise, $a_{i, j, k}^{-1}$ is uniquely determined by the scalars

\begin{center}
$\trace\left( 
p_{ckb}^{\beta} \circ (p_{ijc}^{\alpha}\otimes \ident_k) \circ a_{i, j, k}^{-1} \circ (\ident_i \otimes j_{jka}^{\gamma}) \circ j_{iab}^{\delta}
\right) \ \eqdiag \ $
\begin{tikzpicture}[x=8mm, y=8mm, every node/.style={inner sep=2pt}]
	\newcommand{\sca}{1}
	\node[morphism] (d) {$\delta$};
	\node[morphism, on grid, below right=\sca and \sca of d] (g) {$\gamma$};
	\node[morphism, on grid, below left=\sca and 2*\sca of g] (a) {$\alpha$};
	\node[morphism, on grid, below right=\sca and \sca of a] (b) {$\beta$};
	\draw (d) -- node[right=5pt, pos=0] {$a$} (g);
	\draw (d) -- node[left=3pt] {$i$} (a);
	\draw (g) -- node[right=3pt] {$k$} (b);
	\draw (g) -- node[above=1pt] {$j$} (a);
	\draw (a) -- node[left=5pt, pos=1] {$c$} (b);
	\draw (b.south) .. node[very near start, left=5pt] {$b$} controls +(0, -\sca) and +(0, -\sca) .. +(2.5*\sca, 0) coordinate (P);
	\draw (d.north) .. node[very near start, left=5pt] {$b$} controls +(0, \sca) and +(0, \sca) .. +(2.5*\sca, 0)  coordinate (Q);
	\draw[->] (P) -- (Q);
\end{tikzpicture}
\end{center}

for $a, b, c, i, j, k\in I_{\mathcal C}$ and all $\alpha, \beta, \gamma, \delta$. These scalars are also called \emph{6j symbols} in the following.
The families $(j_{ijk}^{\alpha})_{\alpha}$ and $(p_{ijk}^{\alpha})_{\alpha}$ form bases of $\Hom(k, i\otimes j)$ and $\Hom(i\otimes j, k)$ respectively. Hence the 6j symbols defined above fully describe the linear maps
\begin{align*} \sixj i c j k a b + : \ &\Hom(b, c\otimes k) \otimes \Hom(c, i\otimes j) \otimes \Hom(j\otimes k, a) \otimes \Hom(i\otimes a, b) \to \mathbb F \\
&\delta \otimes \gamma \otimes \alpha \otimes \beta \mapsto \trace\left(\beta \circ (\ident_i \otimes \alpha) \circ a_{i, j, k} \circ (\gamma \otimes \ident_k) \circ \delta\right) \\
\sixj i c j k a b - : \ &\Hom(b, i\otimes a) \otimes \Hom(a, j\otimes k) \otimes \Hom(i\otimes j, c) \otimes \Hom(c\otimes k, b) \to \mathbb F \\
&\delta \otimes \gamma \otimes \alpha \otimes \beta \mapsto \trace\left( \beta \circ (\alpha\otimes\ident_k) \circ a_{i, j, k}^{-1} \circ (\ident_i \otimes \gamma) \circ \delta\right),
\end{align*}
which are also commonly called 6j symbols.

\subsubsection{6j symbols for bimodule categories}
Let $\mathcal M$ be a $(\mathcal C, \mathcal D)$-bimodule category. For any $i\in I_{\mathcal C}, j\in I_{\mathcal M}$,
\[ i\opl j \cong \bigoplus_{k\in I_{\mathcal M}} k^{\oplus m_{ijk}} \]
for certain multiplicities $m_{ijk}\in\mathbb N$. As in Section \ref{sssec:6j_fusion}, we wish to rescale inclusions and projections for $i\opl j$ with the inverse of the dimension of simple objects $k$.

\begin{lemma}
	\cite[Lemma II.4.2.4]{turaev10}
	Let $\mathcal M$ be a semisimple, locally finite $\mathcal C$-module category over $\mathbb F$ and let $\theta$ be a module trace on $\mathcal M$. Then $\dim^{\theta}(x) \neq 0$ for all simple objects $x\in \mathcal M$.
\end{lemma}
\begin{proof}
	This follows immediately from the fact that the pairing
	\[ \Hom_{\mathcal M}(x, x) \times \Hom_{\mathcal M}(x, x) \to \mathbb F: (f, g) \mapsto \theta_x(g\circ f) \]
	is non-degenerate.
\end{proof}

As before, this allows us to choose rescaled inclusions and projections
\begin{center}
	$j_{ijk}^{\alpha} \ \eqdiag \ $
	\begin{tikzpicture}
		\newcommand{\lowerCoor}{0.7}
		\node[morphism, draw=blue, text=blue] (a) {$\alpha$};
		\draw[very thick, color=blue] (a.north) -- node[near end, left=1pt] {$k$} +(0, 1);
		\draw[very thick] (a.south west) -- node[near end, left=5pt] {$i$} +(-\lowerCoor, -\lowerCoor);
		\draw[very thick, color=blue] (a.south east) -- node[near end, right=5pt] {$j$} +(\lowerCoor, -\lowerCoor);
	\end{tikzpicture}
	$, \qquad p_{ijk}^{\alpha} \ \eqdiag \ $
	\begin{tikzpicture}
		\newcommand{\upperCoor}{0.7}
		\node[morphism, draw=blue, text=blue] (a) {$\alpha$};
		\draw[very thick, color=blue] (a.south) -- node[near end, left=1pt] {$k$} +(0, -1);
		\draw[very thick] (a.north west) -- node[near end, left=5pt] {$i$} +(-\upperCoor, \upperCoor);
		\draw[very thick, color=blue] (a.north east) -- node[near end, right=5pt] {$j$} +(\upperCoor, \upperCoor);
	\end{tikzpicture}
\end{center}
for $k\in I_{\mathcal M}$ and $\alpha=1, \dots, m_{ijk}$ which satisfy
\[ p_{ijk}^{\alpha} \circ j_{ijl}^{\beta} = \delta_{k, l} \, \delta_{\alpha, \beta} \, \dim^{-1}(k) \, \ident_k\ , \qquad \sum_{k\in I_{\mathcal M}} \sum_{\alpha=1}^{m_{ijk}} \dim(k)\, j_{ijk}^{\alpha} \circ p_{ijk}^{\alpha} = \ident_{i\opl j}. \]
Similarly, for $i\in I_{\mathcal M}$ and $j\in I_{\mathcal D}$ we can decompose $i \opr j$ into simple objects. We also obtain rescaled inclusions and projections for $i\opr j$:
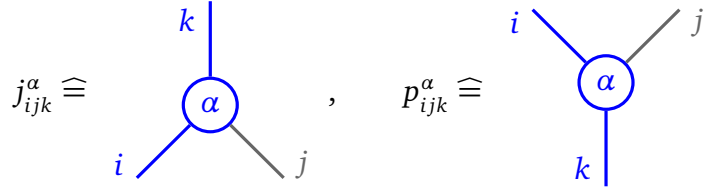
\begin{center}
	$j_{ijk}^{\alpha} \ \eqdiag \ $
	\begin{tikzpicture}
		\newcommand{\lowerCoor}{0.7}
		\node[morphism, draw=blue, text=blue] (a) {$\alpha$};
		\draw[very thick, color=blue] (a.north) -- node[near end, left=1pt] {$k$} +(0, 1);
		\draw[very thick, color=blue] (a.south west) -- node[near end, left=5pt] {$i$} +(-\lowerCoor, -\lowerCoor);
		\draw[very thick, color=grey] (a.south east) -- node[near end, right=5pt] {$j$} +(\lowerCoor, -\lowerCoor);
	\end{tikzpicture}
	$, \qquad p_{ijk}^{\alpha} \ \eqdiag \ $
	\begin{tikzpicture}
		\newcommand{\upperCoor}{0.7}
		\node[morphism, draw=blue, text=blue] (a) {$\alpha$};
		\draw[very thick, color=blue] (a.south) -- node[near end, left=1pt] {$k$} +(0, -1);
		\draw[very thick, color=blue] (a.north west) -- node[near end, left=5pt] {$i$} +(-\upperCoor, \upperCoor);
		\draw[very thick, color=grey] (a.north east) -- node[near end, right=5pt] {$j$} +(\upperCoor, \upperCoor);
	\end{tikzpicture}
\end{center}
for $k\in I_{\mathcal M}$ and $\alpha=1, \dots, m'_{ijk}$.
As in Section \ref{sssec:6j_fusion}, the coherence data of $\mathcal M$ are determined by certain 6j symbols:

1. The module constraint $m$ and its inverse $m^{-1}$ are uniquely determined by the 6j symbols
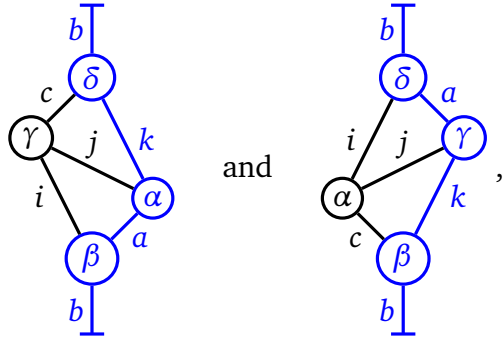
\begin{center}
\begin{tikzpicture}[x=8mm, y=8mm, every node/.style={inner sep=2pt}, module morphism/.style={morphism, draw=blue, text=blue}]
	\newcommand{\sca}{1}
	\node[module morphism] (d) {$\delta$};
	\node[morphism, on grid, below left=\sca and \sca of d] (g) {$\gamma$};
	\node[module morphism, on grid, below right=\sca and 2*\sca of g] (a) {$\alpha$};
	\node[module morphism, on grid, below left=\sca and \sca of a] (b) {$\beta$};
	
	\draw (d) -- node[left=5pt, pos=0] {$c$} (g);
	\draw[color=blue] (d) -- node[right=3pt] {$k$} (a);
	\draw (g) -- node[left=3pt] {$i$} (b);
	\draw (g) -- node[above=1pt] {$j$} (a);
	\draw[color=blue] (a) -- node[pos=1, right=5pt] {$a$} (b);
	
	\draw[color=blue] (d.north) -- node[left] {$b$} +(0, 0.8) coordinate (top);
	\draw[color=blue] ($(top) + (-0.2, 0)$) -- +(0.4, 0);
	\draw[color=blue] (b.south) -- node[left] {$b$} +(0, -0.8) coordinate (bottom);
	\draw[color=blue] ($(bottom) + (-0.2, 0)$) -- +(0.4, 0);
\end{tikzpicture}
\hspace{1em} and \hspace{1em}
\begin{tikzpicture}[x=8mm, y=8mm, every node/.style={inner sep=2pt}, module morphism/.style={morphism, draw=blue, text=blue}]
	\newcommand{\sca}{1}
	\node[module morphism] (d) {$\delta$};
	\node[module morphism, on grid, below right=\sca and \sca of d] (g) {$\gamma$};
	\node[morphism, on grid, below left=\sca and 2*\sca of g] (a) {$\alpha$};
	\node[module morphism, on grid, below right=\sca and \sca of a] (b) {$\beta$};
	\draw[color=blue] (d) -- node[right=5pt, pos=0] {$a$} (g);
	\draw (d) -- node[left=3pt] {$i$} (a);
	\draw[color=blue] (g) -- node[right=3pt] {$k$} (b);
	\draw (g) -- node[above=1pt] {$j$} (a);
	\draw (a) -- node[left=5pt, pos=1] {$c$} (b);
	
	\draw[color=blue] (d.north) -- node[left] {$b$} +(0, 0.8) coordinate (top);
	\draw[color=blue] ($(top) + (-0.2, 0)$) -- +(0.4, 0);
	\draw[color=blue] (b.south) -- node[left] {$b$} +(0, -0.8) coordinate (bottom);
	\draw[color=blue] ($(bottom) + (-0.2, 0)$) -- +(0.4, 0);
\end{tikzpicture}
,
\end{center}
respectively, where $c, i, j\in I_{\mathcal C}, a, b, k\in I_{\mathcal M}$ and $\alpha, \beta, \gamma, \delta$ are arbitrary.

2. The right module constraint $n$ and its inverse $n^{-1}$ are uniquely determined by the 6j symbols
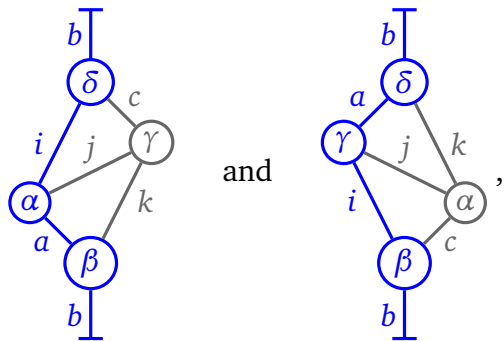
\begin{center}
\begin{tikzpicture}[x=8mm, y=8mm, every node/.style={inner sep=2pt}, module morphism/.style={morphism, draw=blue, text=blue}]
	\newcommand{\sca}{1}
	\node[module morphism] (d) {$\delta$};
	\node[right morphism, on grid, below right=\sca and \sca of d] (g) {$\gamma$};
	\node[module morphism, on grid, below left=\sca and 2*\sca of g] (a) {$\alpha$};
	\node[module morphism, on grid, below right=\sca and \sca of a] (b) {$\beta$};
	
	\draw[color=grey] (d) -- node[right=5pt, pos=0] {$c$} (g);
	\draw[color=blue] (d) -- node[left=3pt] {$i$} (a);
	\draw[color=grey] (g) -- node[right=3pt] {$k$} (b);
	\draw[color=grey] (g) -- node[above=1pt] {$j$} (a);
	\draw[color=blue] (a) -- node[pos=1, left=5pt] {$a$} (b);
	
	\draw[color=blue] (d.north) -- node[left] {$b$} +(0, 0.8) coordinate (top);
	\draw[color=blue] ($(top) + (-0.2, 0)$) -- +(0.4, 0);
	\draw[color=blue] (b.south) -- node[left] {$b$} +(0, -0.8) coordinate (bottom);
	\draw[color=blue] ($(bottom) + (-0.2, 0)$) -- +(0.4, 0);
\end{tikzpicture}
\hspace{1em} and \hspace{1em}
\begin{tikzpicture}[x=8mm, y=8mm, every node/.style={inner sep=2pt}, module morphism/.style={morphism, draw=blue, text=blue}]
	\newcommand{\sca}{1}
	\node[module morphism] (d) {$\delta$};
	\node[module morphism, on grid, below left=\sca and \sca of d] (g) {$\gamma$};
	\node[right morphism, on grid, below right=\sca and 2*\sca of g] (a) {$\alpha$};
	\node[module morphism, on grid, below left=\sca and \sca of a] (b) {$\beta$};
	\draw[color=blue] (d) -- node[left=5pt, pos=0] {$a$} (g);
	\draw[color=grey] (d) -- node[right=3pt] {$k$} (a);
	\draw[color=blue] (g) -- node[left=3pt] {$i$} (b);
	\draw[color=grey] (g) -- node[above=1pt] {$j$} (a);
	\draw[color=grey] (a) -- node[right=5pt, pos=1] {$c$} (b);
	
	\draw[color=blue] (d.north) -- node[left] {$b$} +(0, 0.8) coordinate (top);
	\draw[color=blue] ($(top) + (-0.2, 0)$) -- +(0.4, 0);
	\draw[color=blue] (b.south) -- node[left] {$b$} +(0, -0.8) coordinate (bottom);
	\draw[color=blue] ($(bottom) + (-0.2, 0)$) -- +(0.4, 0);
\end{tikzpicture}
,
\end{center}
respectively, where $c, j, k\in I_{\mathcal D}, a, b, i\in I_{\mathcal M}$ and $\alpha, \beta, \gamma, \delta$ are arbitrary.

3. The middle module constraint $b$ and its inverse $b^{-1}$ are uniquely determined by the 6j symbols
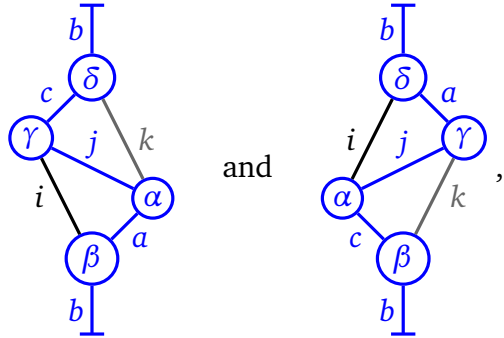
\begin{center}
	\begin{tikzpicture}[x=8mm, y=8mm, every node/.style={inner sep=2pt}, module morphism/.style={morphism, draw=blue, text=blue}]
		\newcommand{\sca}{1}
		\node[module morphism] (d) {$\delta$};
		\node[module morphism, on grid, below left=\sca and \sca of d] (g) {$\gamma$};
		\node[module morphism, on grid, below right=\sca and 2*\sca of g] (a) {$\alpha$};
		\node[module morphism, on grid, below left=\sca and \sca of a] (b) {$\beta$};
		
		\draw[color=blue] (d) -- node[left=5pt, pos=0] {$c$} (g);
		\draw[color=grey] (d) -- node[right=3pt] {$k$} (a);
		\draw (g) -- node[left=3pt] {$i$} (b);
		\draw[color=blue] (g) -- node[above=1pt] {$j$} (a);
		\draw[color=blue] (a) -- node[pos=1, right=5pt] {$a$} (b);
		
		\draw[color=blue] (d.north) -- node[left] {$b$} +(0, 0.8) coordinate (top);
		\draw[color=blue] ($(top) + (-0.2, 0)$) -- +(0.4, 0);
		\draw[color=blue] (b.south) -- node[left] {$b$} +(0, -0.8) coordinate (bottom);
		\draw[color=blue] ($(bottom) + (-0.2, 0)$) -- +(0.4, 0);
	\end{tikzpicture}
	\hspace{1em} and \hspace{1em}
	\begin{tikzpicture}[x=8mm, y=8mm, every node/.style={inner sep=2pt}, module morphism/.style={morphism, draw=blue, text=blue}]
		\newcommand{\sca}{1}
		\node[module morphism] (d) {$\delta$};
		\node[module morphism, on grid, below right=\sca and \sca of d] (g) {$\gamma$};
		\node[module morphism, on grid, below left=\sca and 2*\sca of g] (a) {$\alpha$};
		\node[module morphism, on grid, below right=\sca and \sca of a] (b) {$\beta$};
		
		\draw[color=blue] (d) -- node[right=5pt, pos=0] {$a$} (g);
		\draw (d) -- node[left=3pt] {$i$} (a);
		\draw[color=grey] (g) -- node[right=3pt] {$k$} (b);
		\draw[color=blue] (g) -- node[above=1pt] {$j$} (a);
		\draw[color=blue] (a) -- node[left=5pt, pos=1] {$c$} (b);
		
		\draw[color=blue] (d.north) -- node[left] {$b$} +(0, 0.8) coordinate (top);
		\draw[color=blue] ($(top) + (-0.2, 0)$) -- +(0.4, 0);
		\draw[color=blue] (b.south) -- node[left] {$b$} +(0, -0.8) coordinate (bottom);
		\draw[color=blue] ($(bottom) + (-0.2, 0)$) -- +(0.4, 0);
	\end{tikzpicture}
	,
\end{center}
respectively, where $i\in I_{\mathcal C}, a, b, c, j\in I_{\mathcal M}, k\in I_{\mathcal D}$ and $\alpha, \beta, \gamma, \delta$ are arbitrary.

\subsubsection{6j symbols for module functors}
Let $\mathcal M$ and $\mathcal N$ be $(\mathcal C, \mathcal D)$-bimodule categories and let $F: \mathcal M \to \mathcal N$ be an $\mathbb F$-linear functor. For every $i\in I_{\mathcal M}$, there exists a decomposition
\[ F(i) \cong \bigoplus_{j\in I_{\mathcal N}} j^{\oplus m_{ij}} \]
for certain multiplicities $m_{ij}\in \mathbb N$, as described in Section \ref{ssec:vecgo_modf}. As above, there exist rescaled inclusions and projections for $F(i)$, which are written as
\begin{center}
	\begingroup
	\tikzset{every picture/.append style={scale=0.6}}
	$j_{Fij}^{\alpha} \ \eqdiag \ $
	\begin{tikzpicture}
		\newcommand{\lowerCoor}{0.7}
		\node[morphism, draw=blue, text=blue] (a) {$\alpha$};
		\draw[very thick, color=blue] (a.north) -- node[near end, left=1pt] {$j$} +(0, 1);
		\draw[very thick, color=darkgreen] (a.south west) -- node[near end, left=5pt] {$i$} +(-\lowerCoor, -\lowerCoor);
		\draw[dashed, color=red] (a.south east) -- node[near end, right=5pt] {$F$} +(\lowerCoor, -\lowerCoor);
	\end{tikzpicture}
	$=$
	\begin{tikzpicture}
		\newcommand{\lowerCoor}{0.7}
		\node[morphism, draw=blue, text=blue] (a) {$\alpha$};
		\draw[very thick, color=blue] (a.north) -- node[near end, left=1pt] {$j$} +(0, 1);
		\draw[dashed, color=red] (a.south west) -- node[near end, left=5pt] {$F$} +(-\lowerCoor, -\lowerCoor);
		\draw[very thick, color=darkgreen] (a.south east) -- node[near end, right=5pt] {$i$} +(\lowerCoor, -\lowerCoor);
	\end{tikzpicture}
	$, \qquad p_{Fij}^{\alpha} \ \eqdiag \ $
	\begin{tikzpicture}
		\newcommand{\upperCoor}{0.7}
		\node[morphism, draw=blue, text=blue] (a) {$\alpha$};
		\draw[very thick, color=blue] (a.south) -- node[near end, left=1pt] {$j$} +(0, -1);
		\draw[very thick, color=darkgreen] (a.north west) -- node[near end, left=5pt] {$i$} +(-\upperCoor, \upperCoor);
		\draw[dashed, color=red] (a.north east) -- node[near end, right=5pt] {$F$} +(\upperCoor, \upperCoor);
	\end{tikzpicture}
	$=$
	\begin{tikzpicture}
		\newcommand{\upperCoor}{0.7}
		\node[morphism, draw=blue, text=blue] (a) {$\alpha$};
		\draw[very thick, color=blue] (a.south) -- node[near end, left=1pt] {$j$} +(0, -1);
		\draw[dashed, color=red] (a.north west) -- node[near end, left=5pt] {$F$} +(-\upperCoor, \upperCoor);
		\draw[very thick, color=darkgreen] (a.north east) -- node[near end, right=5pt] {$i$} +(\upperCoor, \upperCoor);
	\end{tikzpicture}
	\endgroup
\end{center}
for $j\in I_{\mathcal N}$ and $\alpha=1, \dots, m_{ij}^F$. They satisfy
\[ p_{Fij}^{\alpha} \circ j_{Fil}^{\beta} = \delta_{j, l} \, \delta_{\alpha, \beta} \, \dim^{-1}(j) \, \ident_j\ , \qquad \sum_{j\in I_{\mathcal N}} \sum_{\alpha=1}^{m_{ij}^F} \dim(j)\, j_{Fij}^{\alpha} \circ p_{Fij}^{\alpha} = \ident_{F(i)}. \]
From here on, $(j_{Fij}^{\alpha}, p_{Fij}^{\alpha})_{j, \alpha}$ will always denote \emph{rescaled} inclusions and projections, unlike in Section \ref{sec:modf}.

Let $(F, s): \mathcal M\to \mathcal N$ be a $\mathcal C$-left module functor and let $(G, t):\mathcal M\to \mathcal N$ be a $\mathcal D$-right module functor.
Then $s$ and its inverse $s^{-1}$ as well as $t$ and its inverse $t^{-1}$ are uniquely determined by the 6j symbols

\begin{center}
	\begin{tikzpicture}[x=8mm, y=8mm, every node/.style={inner sep=2pt}]
		\newcommand{\sca}{1}
		\node[m morphism] (d) {$\delta$};
		\node[n morphism, on grid, below left=\sca and \sca of d] (g) {$\gamma$};
		\node[m morphism, on grid, below right=\sca and 2*\sca of g] (a) {$\alpha$};
		\node[m morphism, on grid, below left=\sca and \sca of a] (b) {$\beta$};
		
		\draw[color=darkgreen] (d) -- node[left=5pt, pos=0] {$c$} (g);
		\draw[dashed, color=red] (d) -- node[right=3pt] {$F$} (a);
		\draw (g) -- node[left=3pt] {$i$} (b);
		\draw[color=darkgreen] (g) -- node[above=1pt] {$j$} (a);
		\draw[color=blue] (a) -- node[pos=1, right=5pt] {$a$} (b);
		
		\draw[color=blue] (d.north) -- node[left] {$b$} +(0, 0.8) coordinate (top);
		\draw[color=blue] ($(top) + (-0.2, 0)$) -- +(0.4, 0);
		\draw[color=blue] (b.south) -- node[left] {$b$} +(0, -0.8) coordinate (bottom);
		\draw[color=blue] ($(bottom) + (-0.2, 0)$) -- +(0.4, 0);
	\end{tikzpicture}
	\hspace{1em} and \hspace{1em}
	\begin{tikzpicture}[x=8mm, y=8mm, every node/.style={inner sep=2pt}]
		\newcommand{\sca}{1}
		\node[m morphism] (d) {$\delta$};
		\node[m morphism, on grid, below right=\sca and \sca of d] (g) {$\gamma$};
		\node[n morphism, on grid, below left=\sca and 2*\sca of g] (a) {$\alpha$};
		\node[m morphism, on grid, below right=\sca and \sca of a] (b) {$\beta$};
		
		\draw[color=blue] (d) -- node[right=5pt, pos=0] {$a$} (g);
		\draw (d) -- node[left=3pt] {$i$} (a);
		\draw[dashed, color=red] (g) -- node[right=3pt] {$F$} (b);
		\draw[color=darkgreen] (g) -- node[above=1pt] {$j$} (a);
		\draw[color=darkgreen] (a) -- node[left=5pt, pos=1] {$c$} (b);
		
		\draw[color=blue] (d.north) -- node[left] {$b$} +(0, 0.8) coordinate (top);
		\draw[color=blue] ($(top) + (-0.2, 0)$) -- +(0.4, 0);
		\draw[color=blue] (b.south) -- node[left] {$b$} +(0, -0.8) coordinate (bottom);
		\draw[color=blue] ($(bottom) + (-0.2, 0)$) -- +(0.4, 0);
	\end{tikzpicture}
, \hfil
	\begin{tikzpicture}[x=8mm, y=8mm, every node/.style={inner sep=2pt}]
		\newcommand{\sca}{1}
		\node[m morphism] (d) {$\delta$};
		\node[n morphism, on grid, below right=\sca and \sca of d] (g) {$\gamma$};
		\node[m morphism, on grid, below left=\sca and 2*\sca of g] (a) {$\alpha$};
		\node[m morphism, on grid, below right=\sca and \sca of a] (b) {$\beta$};
		
		\draw[color=darkgreen] (d) -- node[right=5pt, pos=0] {$c$} (g);
		\draw[dashed, color=red] (d) -- node[left=3pt] {$G$} (a);
		\draw[color=grey] (g) -- node[right=3pt] {$l$} (b);
		\draw[color=darkgreen] (g) -- node[above=1pt] {$j$} (a);
		\draw[color=blue] (a) -- node[left=5pt, pos=1] {$a$} (b);
		
		\draw[color=blue] (d.north) -- node[left] {$b$} +(0, 0.8) coordinate (top);
		\draw[color=blue] ($(top) + (-0.2, 0)$) -- +(0.4, 0);
		\draw[color=blue] (b.south) -- node[left] {$b$} +(0, -0.8) coordinate (bottom);
		\draw[color=blue] ($(bottom) + (-0.2, 0)$) -- +(0.4, 0);
	\end{tikzpicture}
	\hspace{1em} and \hspace{1em}
	\begin{tikzpicture}[x=8mm, y=8mm, every node/.style={inner sep=2pt}]
		\newcommand{\sca}{1}
		\node[m morphism] (d) {$\delta$};
		\node[m morphism, on grid, below left=\sca and \sca of d] (g) {$\gamma$};
		\node[n morphism, on grid, below right=\sca and 2*\sca of g] (a) {$\alpha$};
		\node[m morphism, on grid, below left=\sca and \sca of a] (b) {$\beta$};
		
		\draw[color=blue] (d) -- node[left=5pt, pos=0] {$a$} (g);
		\draw[color=grey] (d) -- node[right=3pt] {$l$} (a);
		\draw[dashed, color=red] (g) -- node[left=3pt] {$G$} (b);
		\draw[color=darkgreen] (g) -- node[above=1pt] {$j$} (a);
		\draw[color=darkgreen] (a) -- node[pos=1, right=5pt] {$c$} (b);
		
		\draw[color=blue] (d.north) -- node[left] {$b$} +(0, 0.8) coordinate (top);
		\draw[color=blue] ($(top) + (-0.2, 0)$) -- +(0.4, 0);
		\draw[color=blue] (b.south) -- node[left] {$b$} +(0, -0.8) coordinate (bottom);
		\draw[color=blue] ($(bottom) + (-0.2, 0)$) -- +(0.4, 0);
	\end{tikzpicture}
\end{center}
for $c, j\in I_{\mathcal M}, a, b\in I_{\mathcal N}, i\in I_{\mathcal C}, l\in I_{\mathcal D}$ and $\alpha, \beta, \gamma, \delta$ arbitrary.

\begin{remark} \label{rem:6j_modc}
	\cite{bm_tv+}
	The 6j symbols for bimodule categories are special cases of the 6j symbols for module functors:
	\begin{enumerate}
	\item Let $k\in I_{\mathcal M}$ and let $F:= - \opl k: \mathcal C \to \mathcal M$ be the $\mathcal C$-module functor from Example \ref{ex:modf}. If we choose $j_{Fij}^{\alpha}:=j_{ikj}^{\alpha}$ and $p_{Fij}^{\alpha}:=p_{ikj}^{\alpha}$ for $i\in I_{\mathcal C}, j\in I_{\mathcal M}$, we have
	
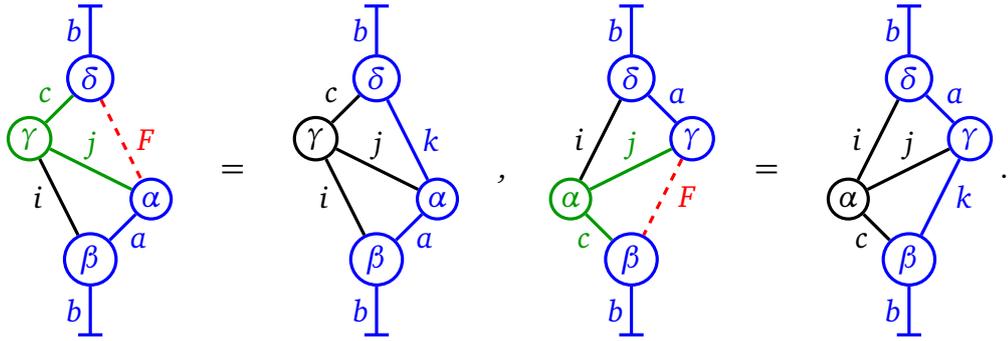
\begin{center}
	\begin{tikzpicture}[x=8mm, y=8mm, every node/.style={inner sep=2pt}]
		\newcommand{\sca}{1}
		\node[m morphism] (d) {$\delta$};
		\node[n morphism, on grid, below left=\sca and \sca of d] (g) {$\gamma$};
		\node[m morphism, on grid, below right=\sca and 2*\sca of g] (a) {$\alpha$};
		\node[m morphism, on grid, below left=\sca and \sca of a] (b) {$\beta$};
		
		\draw[color=darkgreen] (d) -- node[left=5pt, pos=0] {$c$} (g);
		\draw[dashed, color=red] (d) -- node[right=3pt] {$F$} (a);
		\draw (g) -- node[left=3pt] {$i$} (b);
		\draw[color=darkgreen] (g) -- node[above=1pt] {$j$} (a);
		\draw[color=blue] (a) -- node[pos=1, right=5pt] {$a$} (b);
		
		\draw[color=blue] (d.north) -- node[left] {$b$} +(0, 0.8) coordinate (top);
		\draw[color=blue] ($(top) + (-0.2, 0)$) -- +(0.4, 0);
		\draw[color=blue] (b.south) -- node[left] {$b$} +(0, -0.8) coordinate (bottom);
		\draw[color=blue] ($(bottom) + (-0.2, 0)$) -- +(0.4, 0);
	\end{tikzpicture}
	\hspace{1em} $=$ \hspace{1em}
	\begin{tikzpicture}[x=8mm, y=8mm, every node/.style={inner sep=2pt}, module morphism/.style={morphism, draw=blue, text=blue}]
		\newcommand{\sca}{1}
		\node[module morphism] (d) {$\delta$};
		\node[morphism, on grid, below left=\sca and \sca of d] (g) {$\gamma$};
		\node[module morphism, on grid, below right=\sca and 2*\sca of g] (a) {$\alpha$};
		\node[module morphism, on grid, below left=\sca and \sca of a] (b) {$\beta$};
		
		\draw (d) -- node[left=5pt, pos=0] {$c$} (g);
		\draw[color=blue] (d) -- node[right=3pt] {$k$} (a);
		\draw (g) -- node[left=3pt] {$i$} (b);
		\draw (g) -- node[above=1pt] {$j$} (a);
		\draw[color=blue] (a) -- node[pos=1, right=5pt] {$a$} (b);
		
		\draw[color=blue] (d.north) -- node[left] {$b$} +(0, 0.8) coordinate (top);
		\draw[color=blue] ($(top) + (-0.2, 0)$) -- +(0.4, 0);
		\draw[color=blue] (b.south) -- node[left] {$b$} +(0, -0.8) coordinate (bottom);
		\draw[color=blue] ($(bottom) + (-0.2, 0)$) -- +(0.4, 0);
	\end{tikzpicture}
	\hfil , \hfil
	\begin{tikzpicture}[x=8mm, y=8mm, every node/.style={inner sep=2pt}]
		\newcommand{\sca}{1}
		\node[m morphism] (d) {$\delta$};
		\node[m morphism, on grid, below right=\sca and \sca of d] (g) {$\gamma$};
		\node[n morphism, on grid, below left=\sca and 2*\sca of g] (a) {$\alpha$};
		\node[m morphism, on grid, below right=\sca and \sca of a] (b) {$\beta$};
		
		\draw[color=blue] (d) -- node[right=5pt, pos=0] {$a$} (g);
		\draw (d) -- node[left=3pt] {$i$} (a);
		\draw[dashed, color=red] (g) -- node[right=3pt] {$F$} (b);
		\draw[color=darkgreen] (g) -- node[above=1pt] {$j$} (a);
		\draw[color=darkgreen] (a) -- node[left=5pt, pos=1] {$c$} (b);
		
		\draw[color=blue] (d.north) -- node[left] {$b$} +(0, 0.8) coordinate (top);
		\draw[color=blue] ($(top) + (-0.2, 0)$) -- +(0.4, 0);
		\draw[color=blue] (b.south) -- node[left] {$b$} +(0, -0.8) coordinate (bottom);
		\draw[color=blue] ($(bottom) + (-0.2, 0)$) -- +(0.4, 0);
	\end{tikzpicture}
	\hspace{1em}$=$ \hspace{1em}
	\begin{tikzpicture}[x=8mm, y=8mm, every node/.style={inner sep=2pt}, module morphism/.style={morphism, draw=blue, text=blue}]
		\newcommand{\sca}{1}
		\node[module morphism] (d) {$\delta$};
		\node[module morphism, on grid, below right=\sca and \sca of d] (g) {$\gamma$};
		\node[morphism, on grid, below left=\sca and 2*\sca of g] (a) {$\alpha$};
		\node[module morphism, on grid, below right=\sca and \sca of a] (b) {$\beta$};
		\draw[color=blue] (d) -- node[right=5pt, pos=0] {$a$} (g);
		\draw (d) -- node[left=3pt] {$i$} (a);
		\draw[color=blue] (g) -- node[right=3pt] {$k$} (b);
		\draw (g) -- node[above=1pt] {$j$} (a);
		\draw (a) -- node[left=5pt, pos=1] {$c$} (b);
		
		\draw[color=blue] (d.north) -- node[left] {$b$} +(0, 0.8) coordinate (top);
		\draw[color=blue] ($(top) + (-0.2, 0)$) -- +(0.4, 0);
		\draw[color=blue] (b.south) -- node[left] {$b$} +(0, -0.8) coordinate (bottom);
		\draw[color=blue] ($(bottom) + (-0.2, 0)$) -- +(0.4, 0);
	\end{tikzpicture}
	.
	\end{center}
	\item Let $i\in I_{\mathcal M}$. Set $F := i \opr -: \mathcal D \to \mathcal M$ and choose $j_{Fjk}^{\alpha}:=j_{ijk}^{\alpha}$ and $p_{Fjk}^{\alpha}:=p_{ijk}^{\alpha}$ for $j\in I_{\mathcal D}, k\in I_{\mathcal M}$. Then
	
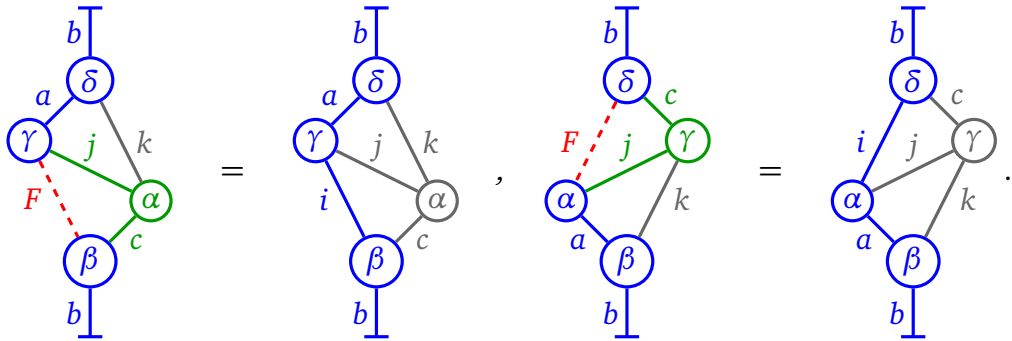
\begin{center}
	\begin{tikzpicture}[x=8mm, y=8mm, every node/.style={inner sep=2pt}]
		\newcommand{\sca}{1}
		\node[m morphism] (d) {$\delta$};
		\node[m morphism, on grid, below left=\sca and \sca of d] (g) {$\gamma$};
		\node[n morphism, on grid, below right=\sca and 2*\sca of g] (a) {$\alpha$};
		\node[m morphism, on grid, below left=\sca and \sca of a] (b) {$\beta$};
		
		\draw[color=blue] (d) -- node[left=5pt, pos=0] {$a$} (g);
		\draw[color=grey] (d) -- node[right=3pt] {$k$} (a);
		\draw[dashed, color=red] (g) -- node[left=3pt] {$F$} (b);
		\draw[color=darkgreen] (g) -- node[above=1pt] {$j$} (a);
		\draw[color=darkgreen] (a) -- node[pos=1, right=5pt] {$c$} (b);
		
		\draw[color=blue] (d.north) -- node[left] {$b$} +(0, 0.8) coordinate (top);
		\draw[color=blue] ($(top) + (-0.2, 0)$) -- +(0.4, 0);
		\draw[color=blue] (b.south) -- node[left] {$b$} +(0, -0.8) coordinate (bottom);
		\draw[color=blue] ($(bottom) + (-0.2, 0)$) -- +(0.4, 0);
	\end{tikzpicture}
	\hspace{1em} $=$ \hspace{1em}
	\begin{tikzpicture}[x=8mm, y=8mm, every node/.style={inner sep=2pt}, module morphism/.style={morphism, draw=blue, text=blue}]
		\newcommand{\sca}{1}
		\node[module morphism] (d) {$\delta$};
		\node[module morphism, on grid, below left=\sca and \sca of d] (g) {$\gamma$};
		\node[right morphism, on grid, below right=\sca and 2*\sca of g] (a) {$\alpha$};
		\node[module morphism, on grid, below left=\sca and \sca of a] (b) {$\beta$};
		\draw[color=blue] (d) -- node[left=5pt, pos=0] {$a$} (g);
		\draw[color=grey] (d) -- node[right=3pt] {$k$} (a);
		\draw[color=blue] (g) -- node[left=3pt] {$i$} (b);
		\draw[color=grey] (g) -- node[above=1pt] {$j$} (a);
		\draw[color=grey] (a) -- node[right=5pt, pos=1] {$c$} (b);
		
		\draw[color=blue] (d.north) -- node[left] {$b$} +(0, 0.8) coordinate (top);
		\draw[color=blue] ($(top) + (-0.2, 0)$) -- +(0.4, 0);
		\draw[color=blue] (b.south) -- node[left] {$b$} +(0, -0.8) coordinate (bottom);
		\draw[color=blue] ($(bottom) + (-0.2, 0)$) -- +(0.4, 0);
	\end{tikzpicture}
	\hfil , \hfil
	\begin{tikzpicture}[x=8mm, y=8mm, every node/.style={inner sep=2pt}]
		\newcommand{\sca}{1}
		\node[m morphism] (d) {$\delta$};
		\node[n morphism, on grid, below right=\sca and \sca of d] (g) {$\gamma$};
		\node[m morphism, on grid, below left=\sca and 2*\sca of g] (a) {$\alpha$};
		\node[m morphism, on grid, below right=\sca and \sca of a] (b) {$\beta$};
		
		\draw[color=darkgreen] (d) -- node[right=5pt, pos=0] {$c$} (g);
		\draw[dashed, color=red] (d) -- node[left=3pt] {$F$} (a);
		\draw[color=grey] (g) -- node[right=3pt] {$k$} (b);
		\draw[color=darkgreen] (g) -- node[above=1pt] {$j$} (a);
		\draw[color=blue] (a) -- node[left=5pt, pos=1] {$a$} (b);
		
		\draw[color=blue] (d.north) -- node[left] {$b$} +(0, 0.8) coordinate (top);
		\draw[color=blue] ($(top) + (-0.2, 0)$) -- +(0.4, 0);
		\draw[color=blue] (b.south) -- node[left] {$b$} +(0, -0.8) coordinate (bottom);
		\draw[color=blue] ($(bottom) + (-0.2, 0)$) -- +(0.4, 0);
	\end{tikzpicture}
	\hspace{1em} $=$ \hspace{1em}
	\begin{tikzpicture}[x=8mm, y=8mm, every node/.style={inner sep=2pt}, module morphism/.style={morphism, draw=blue, text=blue}]
		\newcommand{\sca}{1}
		\node[module morphism] (d) {$\delta$};
		\node[right morphism, on grid, below right=\sca and \sca of d] (g) {$\gamma$};
		\node[module morphism, on grid, below left=\sca and 2*\sca of g] (a) {$\alpha$};
		\node[module morphism, on grid, below right=\sca and \sca of a] (b) {$\beta$};
		
		\draw[color=grey] (d) -- node[right=5pt, pos=0] {$c$} (g);
		\draw[color=blue] (d) -- node[left=3pt] {$i$} (a);
		\draw[color=grey] (g) -- node[right=3pt] {$k$} (b);
		\draw[color=grey] (g) -- node[above=1pt] {$j$} (a);
		\draw[color=blue] (a) -- node[pos=1, left=5pt] {$a$} (b);
		
		\draw[color=blue] (d.north) -- node[left] {$b$} +(0, 0.8) coordinate (top);
		\draw[color=blue] ($(top) + (-0.2, 0)$) -- +(0.4, 0);
		\draw[color=blue] (b.south) -- node[left] {$b$} +(0, -0.8) coordinate (bottom);
		\draw[color=blue] ($(bottom) + (-0.2, 0)$) -- +(0.4, 0);
	\end{tikzpicture}
	.
	\end{center}
	\item Let $i \in I_{\mathcal C}$. Set $F:= i \opl -: \mathcal M \to \mathcal M$ and choose $j_{Fjk}^{\alpha} := j_{ijk}^{\alpha}, p_{Fjk}^{\alpha} := p_{ijk}^{\alpha}$ for $j, k\in I_{\mathcal M}$. Then
	
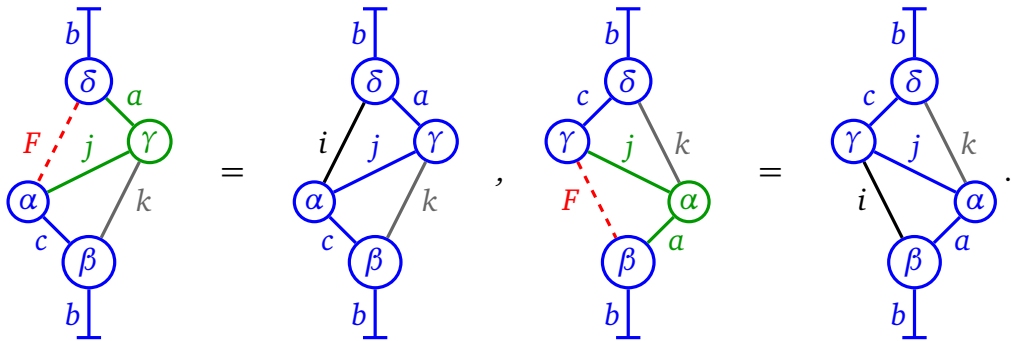
\begin{center}
		\begin{tikzpicture}[x=8mm, y=8mm, every node/.style={inner sep=2pt}]
			\newcommand{\sca}{1}
			\node[m morphism] (d) {$\delta$};
			\node[n morphism, on grid, below right=\sca and \sca of d] (g) {$\gamma$};
			\node[m morphism, on grid, below left=\sca and 2*\sca of g] (a) {$\alpha$};
			\node[m morphism, on grid, below right=\sca and \sca of a] (b) {$\beta$};
			
			\draw[color=darkgreen] (d) -- node[right=5pt, pos=0] {$a$} (g);
			\draw[dashed, color=red] (d) -- node[left=3pt] {$F$} (a);
			\draw[color=grey] (g) -- node[right=3pt] {$k$} (b);
			\draw[color=darkgreen] (g) -- node[above=1pt] {$j$} (a);
			\draw[color=blue] (a) -- node[left=5pt, pos=1] {$c$} (b);
			
			\draw[color=blue] (d.north) -- node[left] {$b$} +(0, 0.8) coordinate (top);
			\draw[color=blue] ($(top) + (-0.2, 0)$) -- +(0.4, 0);
			\draw[color=blue] (b.south) -- node[left] {$b$} +(0, -0.8) coordinate (bottom);
			\draw[color=blue] ($(bottom) + (-0.2, 0)$) -- +(0.4, 0);
		\end{tikzpicture}
		\hspace{1em} $=$ \hspace{1em}
		\begin{tikzpicture}[x=8mm, y=8mm, every node/.style={inner sep=2pt}, module morphism/.style={morphism, draw=blue, text=blue}]
			\newcommand{\sca}{1}
			\node[module morphism] (d) {$\delta$};
			\node[module morphism, on grid, below right=\sca and \sca of d] (g) {$\gamma$};
			\node[module morphism, on grid, below left=\sca and 2*\sca of g] (a) {$\alpha$};
			\node[module morphism, on grid, below right=\sca and \sca of a] (b) {$\beta$};
			
			\draw[color=blue] (d) -- node[right=5pt, pos=0] {$a$} (g);
			\draw (d) -- node[left=3pt] {$i$} (a);
			\draw[color=grey] (g) -- node[right=3pt] {$k$} (b);
			\draw[color=blue] (g) -- node[above=1pt] {$j$} (a);
			\draw[color=blue] (a) -- node[left=5pt, pos=1] {$c$} (b);
			
			\draw[color=blue] (d.north) -- node[left] {$b$} +(0, 0.8) coordinate (top);
			\draw[color=blue] ($(top) + (-0.2, 0)$) -- +(0.4, 0);
			\draw[color=blue] (b.south) -- node[left] {$b$} +(0, -0.8) coordinate (bottom);
			\draw[color=blue] ($(bottom) + (-0.2, 0)$) -- +(0.4, 0);
		\end{tikzpicture}
		\hfil , \hfil
		\begin{tikzpicture}[x=8mm, y=8mm, every node/.style={inner sep=2pt}]
			\newcommand{\sca}{1}
			\node[m morphism] (d) {$\delta$};
			\node[m morphism, on grid, below left=\sca and \sca of d] (g) {$\gamma$};
			\node[n morphism, on grid, below right=\sca and 2*\sca of g] (a) {$\alpha$};
			\node[m morphism, on grid, below left=\sca and \sca of a] (b) {$\beta$};
			
			\draw[color=blue] (d) -- node[left=5pt, pos=0] {$c$} (g);
			\draw[color=grey] (d) -- node[right=3pt] {$k$} (a);
			\draw[dashed, color=red] (g) -- node[left=3pt] {$F$} (b);
			\draw[color=darkgreen] (g) -- node[above=1pt] {$j$} (a);
			\draw[color=darkgreen] (a) -- node[pos=1, right=5pt] {$a$} (b);
			
			\draw[color=blue] (d.north) -- node[left] {$b$} +(0, 0.8) coordinate (top);
			\draw[color=blue] ($(top) + (-0.2, 0)$) -- +(0.4, 0);
			\draw[color=blue] (b.south) -- node[left] {$b$} +(0, -0.8) coordinate (bottom);
			\draw[color=blue] ($(bottom) + (-0.2, 0)$) -- +(0.4, 0);
		\end{tikzpicture}
		\hspace{1em} $=$ \hspace{1em}
		\begin{tikzpicture}[x=8mm, y=8mm, every node/.style={inner sep=2pt}, module morphism/.style={morphism, draw=blue, text=blue}]
			\newcommand{\sca}{1}
			\node[module morphism] (d) {$\delta$};
			\node[module morphism, on grid, below left=\sca and \sca of d] (g) {$\gamma$};
			\node[module morphism, on grid, below right=\sca and 2*\sca of g] (a) {$\alpha$};
			\node[module morphism, on grid, below left=\sca and \sca of a] (b) {$\beta$};
			
			\draw[color=blue] (d) -- node[left=5pt, pos=0] {$c$} (g);
			\draw[color=grey] (d) -- node[right=3pt] {$k$} (a);
			\draw (g) -- node[left=3pt] {$i$} (b);
			\draw[color=blue] (g) -- node[above=1pt] {$j$} (a);
			\draw[color=blue] (a) -- node[pos=1, right=5pt] {$a$} (b);
			
			\draw[color=blue] (d.north) -- node[left] {$b$} +(0, 0.8) coordinate (top);
			\draw[color=blue] ($(top) + (-0.2, 0)$) -- +(0.4, 0);
			\draw[color=blue] (b.south) -- node[left] {$b$} +(0, -0.8) coordinate (bottom);
			\draw[color=blue] ($(bottom) + (-0.2, 0)$) -- +(0.4, 0);
		\end{tikzpicture}
	.
	\end{center}

	Alternatively, the 6j symbols for the middle module constraint can be obtained as 6j symbols of the $\mathcal C$-left module functors $- \opr k: \mathcal M \to \mathcal M$ for $k\in I_{\mathcal D}$.
	\end{enumerate}
\end{remark}

\subsection{Properties of generalised 6j symbols} \label{ssec:bie_ell}
In \cite{barrett_westbury}, invariants of oriented closed 3-manifolds are constructed by choosing a triangulation for a 3-manifold $M$ and considering a state sum model which assigns a 6j symbol of a spherical fusion category to every tetrahedron in the triangulation. In order to show that the result is independent of the chosen triangulation, one needs to show invariance under bistellar moves (also called Pachner moves). It turns out that invariance under bistellar moves follows from certain relations of 6j symbols. Generalised 6j symbols satisfy analogous relations, as is shown in \cite{bm_tv+}. In this section, we summarise these relations, closely following \cite{bm_tv+}. 

Throughout this section, sums are to be read as follows: For Latin letters like $a$, the sum $\sum_a$ denotes the sum over all $a\in I_{\mathcal X}$, where it is clear from the context which category $\mathcal X$ is referred to. For Greek letters like $\beta$, the sum $\sum_\beta$ denotes the sum over $\beta = 1, \dots, n$, where $n\in\mathbb N$ is the largest possible number $\beta$ can take such that all 6j symbols in the subsequent expression are still defined.

\begin{lemma}
\cite[Proposition 5.2]{barrett_westbury}
The 6j symbols for a spherical fusion category $\mathcal C$ satisfy the following relations:
\begin{enumerate}
	\item orthogonality relations:
	\begin{center}
	$\sum_{a, \beta, \alpha} \dim(a) \dim(d)$
		\begin{tikzpicture}[x=8mm, y=8mm, every node/.style={inner sep=2pt}]
			\newcommand{\sca}{1}
			\node[morphism] (d) {$\delta$};
			\node[morphism, on grid, below left=\sca and \sca of d] (g) {$\gamma$};
			\node[morphism, on grid, below right=\sca and 2*\sca of g] (a) {$\alpha$};
			\node[morphism, on grid, below left=\sca and \sca of a] (b) {$\beta$};
			
			\draw (d) -- node[left=5pt, pos=0] {$c$} (g);
			\draw (d) -- node[right=3pt] {$k$} (a);
			\draw (g) -- node[left=3pt] {$i$} (b);
			\draw (g) -- node[above=1pt] {$j$} (a);
			\draw (a) -- node[pos=1, right=5pt] {$a$} (b);
			
			\draw (b.south) .. node[very near start, left=5pt] {$b$} controls +(0, -\sca) and +(0, -\sca) .. +(2.5*\sca, 0) coordinate (P);
			\draw (d.north) .. node[very near start, left=5pt] {$b$} controls +(0, \sca) and +(0, \sca) .. +(2.5*\sca, 0) -- (P);
		\end{tikzpicture}
	\hspace{1em}
		\begin{tikzpicture}[x=8mm, y=8mm, every node/.style={inner sep=2pt}]
			\newcommand{\sca}{1}
			\node[morphism] (d) {$\beta$};
			\node[morphism, on grid, below right=\sca and \sca of d] (g) {$\alpha$};
			\node[morphism, on grid, below left=\sca and 2*\sca of g] (a) {$\epsilon$};
			\node[morphism, on grid, below right=\sca and \sca of a] (b) {$\lambda$};
			
			\draw (d) -- node[right=5pt, pos=0] {$a$} (g);
			\draw (d) -- node[left=3pt] {$i$} (a);
			\draw (g) -- node[right=3pt] {$k$} (b);
			\draw (g) -- node[above=1pt] {$j$} (a);
			\draw (a) -- node[left=5pt, pos=1] {$d$} (b);
			
			\draw (b.south) .. node[very near start, left=5pt] {$b$} controls +(0, -\sca) and +(0, -\sca) .. +(2.5*\sca, 0) coordinate (P);
			\draw (d.north) .. node[very near start, left=5pt] {$b$} controls +(0, \sca) and +(0, \sca) .. +(2.5*\sca, 0) -- (P);
		\end{tikzpicture}
	$ = \delta_{c,d}\, \delta_{\epsilon, \gamma} \, \delta_{\delta, \lambda}$
	\end{center}
	\item Biedenharn-Elliott relations:
	\newcommand{\arcwidth}{1.7}
	\begin{center}
	$\sum_{\delta}$
		\begin{tikzpicture}[x=8mm, y=8mm, every node/.style={inner sep=2pt}]
			\newcommand{\sca}{1}
			\node[morphism] (d) {$\delta$};
			\node[morphism, on grid, below left=\sca and \sca of d] (g) {$\gamma$};
			\node[morphism, on grid, below right=\sca and 2*\sca of g] (a) {$\alpha$};
			\node[morphism, on grid, below left=\sca and \sca of a] (b) {$\beta$};
			
			\draw (d) -- node[left=5pt, pos=0] {$c$} (g);
			\draw (d) -- node[right=3pt] {$k$} (a);
			\draw (g) -- node[left=3pt] {$i$} (b);
			\draw (g) -- node[above=1pt] {$j$} (a);
			\draw (a) -- node[pos=1, right=5pt] {$a$} (b);
			
			\draw (b.south) .. node[very near start, left=5pt] {$b$} controls +(0, -\sca) and +(0, -\sca) .. +(\arcwidth*\sca, 0) coordinate (P);
			\draw (d.north) .. node[very near start, left=5pt] {$b$} controls +(0, \sca) and +(0, \sca) .. +(\arcwidth*\sca, 0) -- (P);
		\end{tikzpicture}
	\hspace{1em}
		\begin{tikzpicture}[x=8mm, y=8mm, every node/.style={inner sep=2pt}]
			\newcommand{\sca}{1}
			\node[morphism] (d) {$\epsilon$};
			\node[morphism, on grid, below left=\sca and \sca of d] (g) {$\lambda$};
			\node[morphism, on grid, below right=\sca and 2*\sca of g] (a) {$\mu$};
			\node[morphism, on grid, below left=\sca and \sca of a] (b) {$\delta$};
			
			\draw (d) -- node[left=5pt, pos=0] {$d$} (g);
			\draw (d) -- node[right=3pt] {$n$} (a);
			\draw (g) -- node[left=3pt] {$c$} (b);
			\draw (g) -- node[above=1pt] {$m$} (a);
			\draw (a) -- node[pos=1, right=5pt] {$k$} (b);
			
			\draw (b.south) .. node[very near start, left=5pt] {$b$} controls +(0, -\sca) and +(0, -\sca) .. +(\arcwidth*\sca, 0) coordinate (P);
			\draw (d.north) .. node[very near start, left=5pt] {$b$} controls +(0, \sca) and +(0, \sca) .. +(\arcwidth*\sca, 0) -- (P);
		\end{tikzpicture}
\end{center}

\begin{center}
	$= \sum_{f, \nu, \rho, \sigma} \dim(f)$
		\begin{tikzpicture}[x=8mm, y=8mm, every node/.style={inner sep=2pt}]
			\newcommand{\sca}{1}
			\node[morphism] (d) {$\epsilon$};
			\node[morphism, on grid, below left=\sca and \sca of d] (g) {$\nu$};
			\node[morphism, on grid, below right=\sca and 2*\sca of g] (a) {$\rho$};
			\node[morphism, on grid, below left=\sca and \sca of a] (b) {$\beta$};
			
			\draw (d) -- node[left=5pt, pos=0] {$d$} (g);
			\draw (d) -- node[right=3pt] {$n$} (a);
			\draw (g) -- node[left=3pt] {$i$} (b);
			\draw (g) -- node[above=1pt] {$f$} (a);
			\draw (a) -- node[pos=1, right=5pt] {$a$} (b);
			
			\draw (b.south) .. node[very near start, left=5pt] {$b$} controls +(0, -\sca) and +(0, -\sca) .. +(\arcwidth*\sca, 0) coordinate (P);
			\draw (d.north) .. node[very near start, left=5pt] {$b$} controls +(0, \sca) and +(0, \sca) .. +(\arcwidth*\sca, 0) -- (P);
		\end{tikzpicture}
	\hspace{1em}
		\begin{tikzpicture}[x=8mm, y=8mm, every node/.style={inner sep=2pt}]
			\newcommand{\sca}{1}
			\node[morphism] (d) {$\lambda$};
			\node[morphism, on grid, below left=\sca and \sca of d] (g) {$\gamma$};
			\node[morphism, on grid, below right=\sca and 2*\sca of g] (a) {$\sigma$};
			\node[morphism, on grid, below left=\sca and \sca of a] (b) {$\nu$};
			
			\draw (d) -- node[left=5pt, pos=0] {$c$} (g);
			\draw (d) -- node[right=3pt] {$m$} (a);
			\draw (g) -- node[left=3pt] {$i$} (b);
			\draw (g) -- node[above=1pt] {$j$} (a);
			\draw (a) -- node[pos=1, right=5pt] {$f$} (b);
			
			\draw (b.south) .. node[very near start, left=5pt] {$d$} controls +(0, -\sca) and +(0, -\sca) .. +(\arcwidth*\sca, 0) coordinate (P);
			\draw (d.north) .. node[very near start, left=5pt] {$d$} controls +(0, \sca) and +(0, \sca) .. +(\arcwidth*\sca, 0) -- (P);
		\end{tikzpicture}
	\hspace{1em}
		\begin{tikzpicture}[x=8mm, y=8mm, every node/.style={inner sep=2pt}]
			\newcommand{\sca}{1}
			\node[morphism] (d) {$\rho$};
			\node[morphism, on grid, below left=\sca and \sca of d] (g) {$\sigma$};
			\node[morphism, on grid, below right=\sca and 2*\sca of g] (a) {$\mu$};
			\node[morphism, on grid, below left=\sca and \sca of a] (b) {$\alpha$};
			
			\draw (d) -- node[left=5pt, pos=0] {$f$} (g);
			\draw (d) -- node[right=3pt] {$n$} (a);
			\draw (g) -- node[left=3pt] {$j$} (b);
			\draw (g) -- node[above=1pt] {$m$} (a);
			\draw (a) -- node[pos=1, right=5pt] {$k$} (b);
			
			\draw (b.south) .. node[very near start, left=5pt] {$a$} controls +(0, -\sca) and +(0, -\sca) .. +(\arcwidth*\sca, 0) coordinate (P);
			\draw (d.north) .. node[very near start, left=5pt] {$a$} controls +(0, \sca) and +(0, \sca) .. +(\arcwidth*\sca, 0) -- (P);
		\end{tikzpicture}
	\end{center}

\end{enumerate}
\end{lemma}

The orthogonality relation follows from the fact that $a$ and $a^{-1}$ are inverses of each other, while the Biedenharn-Elliott relation is a consequence of the pentagon axiom.

We now cite from \cite{bm_tv+} the generalisations of the orthogonality and Biedenharn-Elliott relations for the generalised 6j symbols defined in Section \ref{ssec:6j_def}.

\begin{lemma} \label{lem:rel_6j_modf}
	\cite{bm_tv+}
	The 6j symbols for module functors satisfy the following relations:
	\begin{enumerate}
		\item orthogonality relations:
		\begin{center}
		$\sum_{a, \alpha, \beta} \dim(a) \dim(d)$
		\begin{tikzpicture}[x=8mm, y=8mm, every node/.style={inner sep=2pt}]
			\newcommand{\sca}{1}
			\node[m morphism] (d) {$\delta$};
			\node[n morphism, on grid, below left=\sca and \sca of d] (g) {$\gamma$};
			\node[m morphism, on grid, below right=\sca and 2*\sca of g] (a) {$\alpha$};
			\node[m morphism, on grid, below left=\sca and \sca of a] (b) {$\beta$};
			
			\draw[color=darkgreen] (d) -- node[left=5pt, pos=0] {$c$} (g);
			\draw[dashed, color=red] (d) -- node[right=3pt] {$F$} (a);
			\draw (g) -- node[left=3pt] {$i$} (b);
			\draw[color=darkgreen] (g) -- node[above=1pt] {$j$} (a);
			\draw[color=blue] (a) -- node[pos=1, right=5pt] {$a$} (b);
			
			\draw[color=blue] (d.north) -- node[left] {$b$} +(0, 0.8) coordinate (top);
			\draw[color=blue] ($(top) + (-0.2, 0)$) -- +(0.4, 0);
			\draw[color=blue] (b.south) -- node[left] {$b$} +(0, -0.8) coordinate (bottom);
			\draw[color=blue] ($(bottom) + (-0.2, 0)$) -- +(0.4, 0);
		\end{tikzpicture}
		\hspace{1em}
		\begin{tikzpicture}[x=8mm, y=8mm, every node/.style={inner sep=2pt}]
			\newcommand{\sca}{1}
			\node[m morphism] (d) {$\beta$};
			\node[m morphism, on grid, below right=\sca and \sca of d] (g) {$\alpha$};
			\node[n morphism, on grid, below left=\sca and 2*\sca of g] (a) {$\epsilon$};
			\node[m morphism, on grid, below right=\sca and \sca of a] (b) {$\lambda$};
			
			\draw[color=blue] (d) -- node[right=5pt, pos=0] {$a$} (g);
			\draw (d) -- node[left=3pt] {$i$} (a);
			\draw[dashed, color=red] (g) -- node[right=3pt] {$F$} (b);
			\draw[color=darkgreen] (g) -- node[above=1pt] {$j$} (a);
			\draw[color=darkgreen] (a) -- node[left=5pt, pos=1] {$d$} (b);
			
			\draw[color=blue] (d.north) -- node[left] {$b$} +(0, 0.8) coordinate (top);
			\draw[color=blue] ($(top) + (-0.2, 0)$) -- +(0.4, 0);
			\draw[color=blue] (b.south) -- node[left] {$b$} +(0, -0.8) coordinate (bottom);
			\draw[color=blue] ($(bottom) + (-0.2, 0)$) -- +(0.4, 0);
		\end{tikzpicture}
		$=\delta_{c, d}\, \delta_{\gamma, \epsilon} \, \delta_{\delta, \lambda}$
		\end{center}

		\begin{center}
			$\sum_{c, \gamma, \delta} \dim(c) \dim(d)$
			\begin{tikzpicture}[x=8mm, y=8mm, every node/.style={inner sep=2pt}]
				\newcommand{\sca}{1}
				\node[m morphism] (d) {$\delta$};
				\node[n morphism, on grid, below left=\sca and \sca of d] (g) {$\gamma$};
				\node[m morphism, on grid, below right=\sca and 2*\sca of g] (a) {$\alpha$};
				\node[m morphism, on grid, below left=\sca and \sca of a] (b) {$\beta$};
				
				\draw[color=darkgreen] (d) -- node[left=5pt, pos=0] {$c$} (g);
				\draw[dashed, color=red] (d) -- node[right=3pt] {$F$} (a);
				\draw (g) -- node[left=3pt] {$i$} (b);
				\draw[color=darkgreen] (g) -- node[above=1pt] {$j$} (a);
				\draw[color=blue] (a) -- node[pos=1, right=5pt] {$a$} (b);
				
				\draw[color=blue] (d.north) -- node[left] {$b$} +(0, 0.8) coordinate (top);
				\draw[color=blue] ($(top) + (-0.2, 0)$) -- +(0.4, 0);
				\draw[color=blue] (b.south) -- node[left] {$b$} +(0, -0.8) coordinate (bottom);
				\draw[color=blue] ($(bottom) + (-0.2, 0)$) -- +(0.4, 0);
			\end{tikzpicture}
			\hspace{1em}
			\begin{tikzpicture}[x=8mm, y=8mm, every node/.style={inner sep=2pt}]
				\newcommand{\sca}{1}
				\node[m morphism] (d) {$\epsilon$};
				\node[m morphism, on grid, below right=\sca and \sca of d] (g) {$\lambda$};
				\node[n morphism, on grid, below left=\sca and 2*\sca of g] (a) {$\gamma$};
				\node[m morphism, on grid, below right=\sca and \sca of a] (b) {$\delta$};
				
				\draw[color=blue] (d) -- node[right=5pt, pos=0] {$d$} (g);
				\draw (d) -- node[left=3pt] {$i$} (a);
				\draw[dashed, color=red] (g) -- node[right=3pt] {$F$} (b);
				\draw[color=darkgreen] (g) -- node[above=1pt] {$j$} (a);
				\draw[color=darkgreen] (a) -- node[left=5pt, pos=1] {$c$} (b);
				
				\draw[color=blue] (d.north) -- node[left] {$b$} +(0, 0.8) coordinate (top);
				\draw[color=blue] ($(top) + (-0.2, 0)$) -- +(0.4, 0);
				\draw[color=blue] (b.south) -- node[left] {$b$} +(0, -0.8) coordinate (bottom);
				\draw[color=blue] ($(bottom) + (-0.2, 0)$) -- +(0.4, 0);
			\end{tikzpicture}
			$=\delta_{a, d}\, \delta_{\beta, \epsilon} \, \delta_{\alpha, \lambda}$
		\end{center}
		
		\item Biedenharn-Elliott relations:
		\begin{center}
			$\sum_{\delta}$
			\begin{tikzpicture}[x=8mm, y=8mm, every node/.style={inner sep=2pt}]
				\newcommand{\sca}{1}
				\node[m morphism] (d) {$\delta$};
				\node[morphism, on grid, below left=\sca and \sca of d] (g) {$\gamma$};
				\node[m morphism, on grid, below right=\sca and 2*\sca of g] (a) {$\alpha$};
				\node[m morphism, on grid, below left=\sca and \sca of a] (b) {$\beta$};
				
				\draw (d) -- node[left=5pt, pos=0] {$c$} (g);
				\draw[color=blue] (d) -- node[right=3pt] {$k$} (a);
				\draw (g) -- node[left=3pt] {$i$} (b);
				\draw (g) -- node[above=1pt] {$j$} (a);
				\draw[color=blue] (a) -- node[pos=1, right=5pt] {$a$} (b);
				
				\draw[color=blue] (d.north) -- node[left] {$b$} +(0, 0.8) coordinate (top);
				\draw[color=blue] ($(top) + (-0.2, 0)$) -- +(0.4, 0);
				\draw[color=blue] (b.south) -- node[left] {$b$} +(0, -0.8) coordinate (bottom);
				\draw[color=blue] ($(bottom) + (-0.2, 0)$) -- +(0.4, 0);
			\end{tikzpicture}
		\hspace{1em}
			\begin{tikzpicture}[x=8mm, y=8mm, every node/.style={inner sep=2pt}]
				\newcommand{\sca}{1}
				\node[m morphism] (d) {$\epsilon$};
				\node[n morphism, on grid, below left=\sca and \sca of d] (g) {$\lambda$};
				\node[m morphism, on grid, below right=\sca and 2*\sca of g] (a) {$\mu$};
				\node[m morphism, on grid, below left=\sca and \sca of a] (b) {$\delta$};
				
				\draw[color=darkgreen] (d) -- node[left=5pt, pos=0] {$d$} (g);
				\draw[dashed, color=red] (d) -- node[right=3pt] {$F$} (a);
				\draw (g) -- node[left=3pt] {$c$} (b);
				\draw[color=darkgreen] (g) -- node[above=1pt] {$l$} (a);
				\draw[color=blue] (a) -- node[pos=1, right=5pt] {$k$} (b);
				
				\draw[color=blue] (d.north) -- node[left] {$b$} +(0, 0.8) coordinate (top);
				\draw[color=blue] ($(top) + (-0.2, 0)$) -- +(0.4, 0);
				\draw[color=blue] (b.south) -- node[left] {$b$} +(0, -0.8) coordinate (bottom);
				\draw[color=blue] ($(bottom) + (-0.2, 0)$) -- +(0.4, 0);
			\end{tikzpicture}
		\end{center}
	
		\begin{center}
			\hfill
			$=\sum_{m, \nu, \rho, \sigma} \dim(m)$
			\begin{tikzpicture}[x=8mm, y=8mm, every node/.style={inner sep=2pt}]
				\newcommand{\sca}{1}
				\node[m morphism] (d) {$\epsilon$};
				\node[n morphism, on grid, below left=\sca and \sca of d] (g) {$\nu$};
				\node[m morphism, on grid, below right=\sca and 2*\sca of g] (a) {$\rho$};
				\node[m morphism, on grid, below left=\sca and \sca of a] (b) {$\beta$};
				
				\draw[color=darkgreen] (d) -- node[left=5pt, pos=0] {$d$} (g);
				\draw[dashed, color=red] (d) -- node[right=3pt] {$F$} (a);
				\draw (g) -- node[left=3pt] {$i$} (b);
				\draw[color=darkgreen] (g) -- node[above=1pt] {$m$} (a);
				\draw[color=blue] (a) -- node[pos=1, right=5pt] {$a$} (b);
				
				\draw[color=blue] (d.north) -- node[left] {$b$} +(0, 0.8) coordinate (top);
				\draw[color=blue] ($(top) + (-0.2, 0)$) -- +(0.4, 0);
				\draw[color=blue] (b.south) -- node[left] {$b$} +(0, -0.8) coordinate (bottom);
				\draw[color=blue] ($(bottom) + (-0.2, 0)$) -- +(0.4, 0);
			\end{tikzpicture}
		\hspace{1em}
			\begin{tikzpicture}[x=8mm, y=8mm, every node/.style={inner sep=2pt}]
				\newcommand{\sca}{1}
				\node[n morphism] (d) {$\lambda$};
				\node[morphism, on grid, below left=\sca and \sca of d] (g) {$\gamma$};
				\node[n morphism, on grid, below right=\sca and 2*\sca of g] (a) {$\sigma$};
				\node[n morphism, on grid, below left=\sca and \sca of a] (b) {$\nu$};
				
				\draw (d) -- node[left=5pt, pos=0] {$c$} (g);
				\draw[color=darkgreen] (d) -- node[right=3pt] {$l$} (a);
				\draw (g) -- node[left=3pt] {$i$} (b);
				\draw (g) -- node[above=1pt] {$j$} (a);
				\draw[color=darkgreen] (a) -- node[pos=1, right=5pt] {$m$} (b);
				
				\draw[color=darkgreen] (d.north) -- node[left] {$d$} +(0, 0.8) coordinate (top);
				\draw[color=darkgreen] ($(top) + (-0.2, 0)$) -- +(0.4, 0);
				\draw[color=darkgreen] (b.south) -- node[left] {$d$} +(0, -0.8) coordinate (bottom);
				\draw[color=darkgreen] ($(bottom) + (-0.2, 0)$) -- +(0.4, 0);
			\end{tikzpicture}
		\hspace{1em}
			\begin{tikzpicture}[x=8mm, y=8mm, every node/.style={inner sep=2pt}]
				\newcommand{\sca}{1}
				\node[m morphism] (d) {$\rho$};
				\node[n morphism, on grid, below left=\sca and \sca of d] (g) {$\sigma$};
				\node[m morphism, on grid, below right=\sca and 2*\sca of g] (a) {$\mu$};
				\node[m morphism, on grid, below left=\sca and \sca of a] (b) {$\alpha$};
				
				\draw[color=darkgreen] (d) -- node[left=5pt, pos=0] {$m$} (g);
				\draw[dashed, color=red] (d) -- node[right=3pt] {$F$} (a);
				\draw (g) -- node[left=3pt] {$j$} (b);
				\draw[color=darkgreen] (g) -- node[above=1pt] {$l$} (a);
				\draw[color=blue] (a) -- node[pos=1, right=5pt] {$k$} (b);
				
				\draw[color=blue] (d.north) -- node[left] {$a$} +(0, 0.8) coordinate (top);
				\draw[color=blue] ($(top) + (-0.2, 0)$) -- +(0.4, 0);
				\draw[color=blue] (b.south) -- node[left] {$a$} +(0, -0.8) coordinate (bottom);
				\draw[color=blue] ($(bottom) + (-0.2, 0)$) -- +(0.4, 0);
			\end{tikzpicture}
		{\normalfont \hfill \eqlabel{eq:ber_modf}}
		\end{center}
	
	\end{enumerate}
\end{lemma}

The orthogonality relations follow from the fact that $s$ and $s^{-1}$ are inverses of each other, whereas the Biedenharn-Elliott relations are a consequence of the pentagon axiom for module functors (\ref{diag:modf}).

The 6j symbols for bimodule categories satisfy analogous relations. Their orthogonality relations follow because $m$ and $m^{-1}$, $n$ and $n^{-1}$ as well as $b$ and $b^{-1}$, respectively, are inverses of each other. The Biedenharn-Elliott relations for the 6j symbols of bimodule categories are a consequence of the four pentagon axioms for bimodule categories, see Definition \ref{def:bimodc}. Alternatively, these relations can be obtained by combining Lemma \ref{lem:rel_6j_modf} with Remark \ref{rem:6j_modc}.
\subsection{Generalised 6j symbols over $\vectgo$}
We now consider the case $\mathcal C=\mathrm{Vec}_G^{\omega_G, \kappa_G}, \mathcal D=\mathrm{Vec}_H^{\omega_H, \kappa_H}$.

\subsubsection{6j symbols for spherical fusion categories}
We assume that $\mathcal C$ and $\mathcal D$ are spherical fusion categories, so $G$ and $H$ are finite groups and, by Lemma \ref{lem:vecgok_sph}, $\kappa_G$ and $\kappa_H$ only take values in $\{-1, 1\}$.

Choose $I_{\mathcal C} = \{ \delta^g \mid g\in G \} \cong G$ and $I_{\mathcal D} = \{ \delta^h \mid h\in H\} \cong H$. Since $\delta^i \otimes \delta^j = \delta^{ij} \in I_{\mathcal C}$ is simple for all $i, j\in G$, we can choose as rescaled inclusions and projections for $i\otimes j$
\[ j_{ijk}^1 := \ident_k:\ \delta^k \to \delta^i \otimes \delta^j, \quad p_{ijk}^1 := \kappa_G^{-1}(k)\, \ident_k = \kappa_G(k)\, \ident_k:\ \delta^i \otimes \delta^j \to \delta^k \] 
for $i, j, k\in G$ with $ij = k$. Note that $j_{ijk}^1$ and $p_{ijk}^1$ are undefined unless $ij=k$. Clearly, other choices of rescaled inclusions and projections are possible.

We can now calculate the 6j symbols associated to spherical fusion categories. Let $a, b, c, i, j, k\in G$. Then the 6j symbols below are only defined if $\alpha = \beta = \gamma = \delta = 1$ and $b=ck, c=ij, a=jk$ and, as a consequence, $b=ia$. We calculate

\begin{center}
	\begin{tikzpicture}[x=8mm, y=8mm, every node/.style={inner sep=2pt}]
		\newcommand{\sca}{1}
		\node[morphism] (d) {$\delta$};
		\node[morphism, on grid, below left=\sca and \sca of d] (g) {$\gamma$};
		\node[morphism, on grid, below right=\sca and 2*\sca of g] (a) {$\alpha$};
		\node[morphism, on grid, below left=\sca and \sca of a] (b) {$\beta$};
		\draw (d) -- node[left=5pt, pos=0] {$c$} (g);
		\draw (d) -- node[right=3pt] {$k$} (a);
		\draw (g) -- node[left=3pt] {$i$} (b);
		\draw (g) -- node[above=1pt] {$j$} (a);
		\draw (a) -- node[pos=1, right=5pt] {$a$} (b);
		\draw (b.south) .. node[very near start, left=5pt] {$b$} controls +(0, -\sca) and +(0, -\sca) .. +(2.5*\sca, 0) coordinate (P);
		\draw (d.north) .. node[very near start, left=5pt] {$b$} controls +(0, \sca) and +(0, \sca) .. +(2.5*\sca, 0) -- (P);
	\end{tikzpicture}
	$\ \eqdiag \ \trace\left( 
	p_{iab}^{\beta} \circ (\ident_i \otimes p_{jka}^{\alpha}) \circ a_{i, j, k} \circ (j_{ijc}^{\gamma} \otimes \ident_k) \circ j_{ckb}^{\delta}
	\right)$
\end{center}
\[ = \dim(b) \ p_{iab}^{\beta} \circ (\ident_i \otimes p_{jka}^{\alpha}) \circ a_{i, j, k} \circ (j_{ijc}^{\gamma} \otimes \ident_k) \circ j_{ckb}^{\delta} = \kappa_G(a) \, \omega_G(i, j, k), \]

\begin{center}
	\begin{tikzpicture}[x=8mm, y=8mm, every node/.style={inner sep=2pt}]
		\newcommand{\sca}{1}
		\node[morphism] (d) {$\delta$};
		\node[morphism, on grid, below right=\sca and \sca of d] (g) {$\gamma$};
		\node[morphism, on grid, below left=\sca and 2*\sca of g] (a) {$\alpha$};
		\node[morphism, on grid, below right=\sca and \sca of a] (b) {$\beta$};
		\draw (d) -- node[right=5pt, pos=0] {$a$} (g);
		\draw (d) -- node[left=3pt] {$i$} (a);
		\draw (g) -- node[right=3pt] {$k$} (b);
		\draw (g) -- node[above=1pt] {$j$} (a);
		\draw (a) -- node[left=5pt, pos=1] {$c$} (b);
		\draw (b.south) .. node[very near start, left=5pt] {$b$} controls +(0, -\sca) and +(0, -\sca) .. +(2.5*\sca, 0) coordinate (P);
		\draw (d.north) .. node[very near start, left=5pt] {$b$} controls +(0, \sca) and +(0, \sca) .. +(2.5*\sca, 0) -- (P);
	\end{tikzpicture}
	$\ \eqdiag \ \trace\left( 
	p_{ckb}^{\beta} \circ (p_{ijc}^{\alpha}\otimes \ident_k) \circ a_{i, j, k}^{-1} \circ (\ident_i \otimes j_{jka}^{\gamma}) \circ j_{iab}^{\delta}
	\right)$
\end{center}
\[ =\dim(b) \ p_{ckb}^{\beta} \circ (p_{ijc}^{\alpha}\otimes \ident_k) \circ a_{i, j, k}^{-1} \circ (\ident_i \otimes j_{jka}^{\gamma}) \circ j_{iab}^{\delta} = \kappa_G(c) \, \omega_G^{-1}(i, j, k). \]
Here we identified both $\Hom_{\vectgo}(\delta^1, \delta^1)$ and $\Hom_{\vectgo}(\delta^b, \delta^b)$ with $\mathbb F$.
The orthogonality relation is obviously satisfied and a direct calculation shows that the Biedenharn-Elliott relation follows from the fact that $\omega$ is a cocycle and thus satisfies (\ref{eq:3cocycle}).

\subsubsection{6j symbols for bimodule categories}
\newcommand{\bimodc}{\mathcal B(X, \Psi, \Phi, \Omega)}
Now let $\bimodc$ be a finite semisimple $(\mathcal C, \mathcal D)$-bimodule category as defined in Section \ref{ssec:vecgo_bimodc}. This implies that $X$ is finite. In Section \ref{ssec:vecgo_bimodc}, we proved $\mathcal C \boxtimes \mathcal D^{rev} = \mathrm{Vec}_{G\times H}^{\omega}$, where
\[ \omega((g_1, h_1), (g_2, h_2), (g_3, h_3)) = \omega_G(g_1, g_2, g_3) \, \omega_H^{-1}(h_3^{-1}, h_2^{-1}, h_1^{-1}). \]
Let $\mathcal M(X, \Gamma)$ be the $\mathrm{Vec}_{G\times H}^{\omega}$-module category that is associated to $\bimodc$ by Proposition \ref{pro:vecgo_bimodc}:
\[ \Gamma((g_1, h_1), (g_2, h_2), x) = \Psi(g_1, g_2, (h_2^{-1}h_1^{-1}) \opl x) \, \Phi(h_1, h_2, x) \, \Omega(g_1, h_2, h_1^{-1} \opl x), \]
By definition, a $(\mathcal C, \mathcal D)$-bimodule trace on $\bimodc$ is the same as a $\mathcal C \boxtimes \mathcal D^{rev}$-module trace on $\mathcal M(X, \Gamma)$.

The 6j symbols for a bimodule category are only defined if $\mathcal M(X, \Gamma)$ admits a module trace, i.e., satisfies the condition from Lemma \ref{lem:vecgo_mod_tr} for the $\kappa$ defined by (\ref{eq:kappa_deli_prod}). We will therefore assume from here on that there exists a module trace $\theta$ on $\mathcal M(X, \Gamma)$ and write $\tilde \kappa_X(x):= \dim^{\theta}(x)$ for all $x\in X$. Note that $\theta$ is uniquely determined by $\tilde\kappa_X$. We may assume $\tilde\kappa_X(x) \in \{1, -1\}$ for all $x\in X$, as shown in Section \ref{ssec:vecgo_modc}.

Let $c, i, j\in G, a, b, k\in X$. The 6j symbols below are only defined if $\alpha=\beta=\gamma=\delta=1$ and $b = c\opl k, c=ij, a=j\opl k$. This implies $i\opl a= b$. We can calculate
\begin{tabular}{l}
	\begin{tikzpicture}[x=8mm, y=8mm, every node/.style={inner sep=2pt}, module morphism/.style={morphism, draw=blue, text=blue}]
		\newcommand{\sca}{1}
		\node[module morphism] (d) {$\delta$};
		\node[morphism, on grid, below left=\sca and \sca of d] (g) {$\gamma$};
		\node[module morphism, on grid, below right=\sca and 2*\sca of g] (a) {$\alpha$};
		\node[module morphism, on grid, below left=\sca and \sca of a] (b) {$\beta$};
		
		\draw (d) -- node[left=5pt, pos=0] {$c$} (g);
		\draw[color=blue] (d) -- node[right=3pt] {$k$} (a);
		\draw (g) -- node[left=3pt] {$i$} (b);
		\draw (g) -- node[above=1pt] {$j$} (a);
		\draw[color=blue] (a) -- node[pos=1, right=5pt] {$a$} (b);
		
		\draw[color=blue] (d.north) -- node[left] {$b$} +(0, 0.8) coordinate (top);
		\draw[color=blue] ($(top) + (-0.2, 0)$) -- +(0.4, 0);
		\draw[color=blue] (b.south) -- node[left] {$b$} +(0, -0.8) coordinate (bottom);
		\draw[color=blue] ($(bottom) + (-0.2, 0)$) -- +(0.4, 0);
	\end{tikzpicture}
	$\ \eqdiag \ \theta_b\left( p_{iab}^{\beta} \circ (\ident_i \opl p_{jka}^{\alpha}) \circ m_{i, j, k} \circ (j_{ijc}^{\gamma} \opl \ident_k) \circ j_{ckb}^{\delta} \right) = \tilde \kappa_X(a) \, \Psi(i, j, b),$
\\
	\begin{tikzpicture}[x=8mm, y=8mm, every node/.style={inner sep=2pt}, module morphism/.style={morphism, draw=blue, text=blue}]
		\newcommand{\sca}{1}
		\node[module morphism] (d) {$\delta$};
		\node[module morphism, on grid, below right=\sca and \sca of d] (g) {$\gamma$};
		\node[morphism, on grid, below left=\sca and 2*\sca of g] (a) {$\alpha$};
		\node[module morphism, on grid, below right=\sca and \sca of a] (b) {$\beta$};
		\draw[color=blue] (d) -- node[right=5pt, pos=0] {$a$} (g);
		\draw (d) -- node[left=3pt] {$i$} (a);
		\draw[color=blue] (g) -- node[right=3pt] {$k$} (b);
		\draw (g) -- node[above=1pt] {$j$} (a);
		\draw (a) -- node[left=5pt, pos=1] {$c$} (b);
		
		\draw[color=blue] (d.north) -- node[left] {$b$} +(0, 0.8) coordinate (top);
		\draw[color=blue] ($(top) + (-0.2, 0)$) -- node[above=2pt] {} +(0.4, 0);
		\draw[color=blue] (b.south) -- node[left] {$b$} +(0, -0.8) coordinate (bottom);
		\draw[color=blue] ($(bottom) + (-0.2, 0)$) -- +(0.4, 0);
	\end{tikzpicture}
	$\ \eqdiag \ \theta_b\left( p_{ckb}^{\beta} \circ (p_{ijc}^{\alpha} \opl \ident_k) \circ m_{i, j, k}^{-1} \circ (\ident_i \opl j_{jka}^{\gamma}) \circ j_{iab}^{\delta} \right) = \kappa_G(c) \, \Psi^{-1}(i, j, b).$
\end{tabular}

Similarly, for $c, j, k\in H, a, b, i\in X$, the 6j symbols below are only defined if $\alpha=\beta=\gamma=\delta=1$ and $b = i \opr c, a=i\opr j, c=jk$ and hence $b = a \opr k$. We have

\begin{center}
	\begin{tikzpicture}[x=8mm, y=8mm, every node/.style={inner sep=2pt}, module morphism/.style={morphism, draw=blue, text=blue}]
		\newcommand{\sca}{1}
		\node[module morphism] (d) {$\delta$};
		\node[module morphism, on grid, below left=\sca and \sca of d] (g) {$\gamma$};
		\node[right morphism, on grid, below right=\sca and 2*\sca of g] (a) {$\alpha$};
		\node[module morphism, on grid, below left=\sca and \sca of a] (b) {$\beta$};
		\draw[color=blue] (d) -- node[left=5pt, pos=0] {$a$} (g);
		\draw[color=grey] (d) -- node[right=3pt] {$k$} (a);
		\draw[color=blue] (g) -- node[left=3pt] {$i$} (b);
		\draw[color=grey] (g) -- node[above=1pt] {$j$} (a);
		\draw[color=grey] (a) -- node[right=5pt, pos=1] {$c$} (b);
		
		\draw[color=blue] (d.north) -- node[left] {$b$} +(0, 0.8) coordinate (top);
		\draw[color=blue] ($(top) + (-0.2, 0)$) -- +(0.4, 0);
		\draw[color=blue] (b.south) -- node[left] {$b$} +(0, -0.8) coordinate (bottom);
		\draw[color=blue] ($(bottom) + (-0.2, 0)$) -- +(0.4, 0);
	\end{tikzpicture}
	$=\kappa_H(c)\, \Phi^{-1}(k^{-1}, j^{-1}, b)\, ,$ \hfil
	\begin{tikzpicture}[x=8mm, y=8mm, every node/.style={inner sep=2pt}, module morphism/.style={morphism, draw=blue, text=blue}]
		\newcommand{\sca}{1}
		\node[module morphism] (d) {$\delta$};
		\node[right morphism, on grid, below right=\sca and \sca of d] (g) {$\gamma$};
		\node[module morphism, on grid, below left=\sca and 2*\sca of g] (a) {$\alpha$};
		\node[module morphism, on grid, below right=\sca and \sca of a] (b) {$\beta$};
		
		\draw[color=grey] (d) -- node[right=5pt, pos=0] {$c$} (g);
		\draw[color=blue] (d) -- node[left=3pt] {$i$} (a);
		\draw[color=grey] (g) -- node[right=3pt] {$k$} (b);
		\draw[color=grey] (g) -- node[above=1pt] {$j$} (a);
		\draw[color=blue] (a) -- node[pos=1, left=5pt] {$a$} (b);
		
		\draw[color=blue] (d.north) -- node[left] {$b$} +(0, 0.8) coordinate (top);
		\draw[color=blue] ($(top) + (-0.2, 0)$) -- +(0.4, 0);
		\draw[color=blue] (b.south) -- node[left] {$b$} +(0, -0.8) coordinate (bottom);
		\draw[color=blue] ($(bottom) + (-0.2, 0)$) -- +(0.4, 0);
	\end{tikzpicture}
	$=\tilde\kappa_X(a) \, \Phi(k^{-1}, j^{-1}, b).$
\end{center}

Now let $i\in G, k\in H, a, b, c, j\in X$. The 6j symbols for the middle module constraint below are only defined if $\alpha=\beta=\gamma=\delta=1$ and $b = c\opr k, c=i\opl j, a=j\opr k$ and thus $b=i\opl a$. As before, we calculate

\begin{center}
	\begin{tikzpicture}[x=8mm, y=8mm, every node/.style={inner sep=2pt}, module morphism/.style={morphism, draw=blue, text=blue}]
		\newcommand{\sca}{1}
		\node[module morphism] (d) {$\delta$};
		\node[module morphism, on grid, below left=\sca and \sca of d] (g) {$\gamma$};
		\node[module morphism, on grid, below right=\sca and 2*\sca of g] (a) {$\alpha$};
		\node[module morphism, on grid, below left=\sca and \sca of a] (b) {$\beta$};
		
		\draw[color=blue] (d) -- node[left=5pt, pos=0] {$c$} (g);
		\draw[color=grey] (d) -- node[right=3pt] {$k$} (a);
		\draw (g) -- node[left=3pt] {$i$} (b);
		\draw[color=blue] (g) -- node[above=1pt] {$j$} (a);
		\draw[color=blue] (a) -- node[pos=1, right=5pt] {$a$} (b);
		
		\draw[color=blue] (d.north) -- node[left] {$b$} +(0, 0.8) coordinate (top);
		\draw[color=blue] ($(top) + (-0.2, 0)$) -- +(0.4, 0);
		\draw[color=blue] (b.south) -- node[left] {$b$} +(0, -0.8) coordinate (bottom);
		\draw[color=blue] ($(bottom) + (-0.2, 0)$) -- +(0.4, 0);
	\end{tikzpicture}
	$= \tilde\kappa_X(a)\, \Omega(i, k^{-1}, b)\, ,$ \hfil
	\begin{tikzpicture}[x=8mm, y=8mm, every node/.style={inner sep=2pt}, module morphism/.style={morphism, draw=blue, text=blue}]
		\newcommand{\sca}{1}
		\node[module morphism] (d) {$\delta$};
		\node[module morphism, on grid, below right=\sca and \sca of d] (g) {$\gamma$};
		\node[module morphism, on grid, below left=\sca and 2*\sca of g] (a) {$\alpha$};
		\node[module morphism, on grid, below right=\sca and \sca of a] (b) {$\beta$};
		
		\draw[color=blue] (d) -- node[right=5pt, pos=0] {$a$} (g);
		\draw (d) -- node[left=3pt] {$i$} (a);
		\draw[color=grey] (g) -- node[right=3pt] {$k$} (b);
		\draw[color=blue] (g) -- node[above=1pt] {$j$} (a);
		\draw[color=blue] (a) -- node[left=5pt, pos=1] {$c$} (b);
		
		\draw[color=blue] (d.north) -- node[left] {$b$} +(0, 0.8) coordinate (top);
		\draw[color=blue] ($(top) + (-0.2, 0)$) -- +(0.4, 0);
		\draw[color=blue] (b.south) -- node[left] {$b$} +(0, -0.8) coordinate (bottom);
		\draw[color=blue] ($(bottom) + (-0.2, 0)$) -- +(0.4, 0);
	\end{tikzpicture}
	$=\tilde\kappa_X(c)\, \Omega^{-1}(i, k^{-1}, b).$
\end{center}

\subsubsection{6j symbols for module functors}
We now determine the 6j symbols associated to $\vectgo$-module functors. Let $\mathcal M:=\mathcal B(X, \Psi_X, \Phi_X, \Omega_X)$ and $\mathcal N:=\mathcal B(Y, \Psi_Y, \Phi_Y, \Omega_Y)$ be finite semisimple $(\mathcal C, \mathcal D)$-bimodule categories and let $(F, s): \mathcal M\to \mathcal N$ be a $\mathcal C$-module functor. Let $i\in X$. Denote by $(\overline j_{Fij}^{\alpha}, \overline p_{Fij}^{\alpha})_{j, \alpha}$ the (unrescaled) inclusions and projections for $F(i)$ that were chosen in Section \ref{ssec:vecgo_modf}. We can then define rescaled inclusions and projections for $F(i)$ as
\[ j_{Fij}^{\alpha} := \overline j_{Fij}^{\alpha}, \quad p_{Fij}^{\alpha} := \tilde \kappa_Y(j) \, \overline p_{Fij}^{\alpha} \qquad (j\in Y, \alpha=1, \dots, m_{ij}^F). \]
Let $i\in G, c, j\in X, a, b\in Y$. The 6j symbols for the $\mathcal C$-left module functor $F$ which are calculated below are only defined for $\beta=\gamma=1$ and $c=i \opl j, b = i \opl a$.
\begin{tabular}{l}
	\begin{tikzpicture}[x=8mm, y=8mm, every node/.style={inner sep=2pt}]
		\newcommand{\sca}{1}
		\node[m morphism] (d) {$\delta$};
		\node[n morphism, on grid, below left=\sca and \sca of d] (g) {$\gamma$};
		\node[m morphism, on grid, below right=\sca and 2*\sca of g] (a) {$\alpha$};
		\node[m morphism, on grid, below left=\sca and \sca of a] (b) {$\beta$};
		
		\draw[color=darkgreen] (d) -- node[left=5pt, pos=0] {$c$} (g);
		\draw[dashed, color=red] (d) -- node[right=3pt] {$F$} (a);
		\draw (g) -- node[left=3pt] {$i$} (b);
		\draw[color=darkgreen] (g) -- node[above=1pt] {$j$} (a);
		\draw[color=blue] (a) -- node[pos=1, right=5pt] {$a$} (b);
		
		\draw[color=blue] (d.north) -- node[left] {$b$} +(0, 0.8) coordinate (top);
		\draw[color=blue] ($(top) + (-0.2, 0)$) -- node[above=4pt] {} +(0.4, 0);
		\draw[color=blue] (b.south) -- node[left] {$b$} +(0, -0.8) coordinate (bottom);
		\draw[color=blue] ($(bottom) + (-0.2, 0)$) -- +(0.4, 0);
	\end{tikzpicture}
	$\ \eqdiag \ \theta_b\left( p_{iab}^{\beta} \circ (\ident_i \opl p_{Fja}^{\alpha}) \circ s_{i, j} \circ F(j_{ijc}^{\gamma}) \circ j_{Fcb}^{\delta} \right) = \tilde\kappa_Y(a)\, \left( A_{i, j, a}^F\right)^\alpha_\delta,$
\\
	\begin{tikzpicture}[x=8mm, y=8mm, every node/.style={inner sep=2pt}]
		\newcommand{\sca}{1}
		\node[m morphism] (d) {$\delta$};
		\node[m morphism, on grid, below right=\sca and \sca of d] (g) {$\gamma$};
		\node[n morphism, on grid, below left=\sca and 2*\sca of g] (a) {$\alpha$};
		\node[m morphism, on grid, below right=\sca and \sca of a] (b) {$\beta$};
		
		\draw[color=blue] (d) -- node[right=5pt, pos=0] {$a$} (g);
		\draw (d) -- node[left=3pt] {$i$} (a);
		\draw[dashed, color=red] (g) -- node[right=3pt] {$F$} (b);
		\draw[color=darkgreen] (g) -- node[above=1pt] {$j$} (a);
		\draw[color=darkgreen] (a) -- node[left=5pt, pos=1] {$c$} (b);
		
		\draw[color=blue] (d.north) -- node[left] {$b$} +(0, 0.8) coordinate (top);
		\draw[color=blue] ($(top) + (-0.2, 0)$) -- node[above=2pt] {} +(0.4, 0);
		\draw[color=blue] (b.south) -- node[left] {$b$} +(0, -0.8) coordinate (bottom);
		\draw[color=blue] ($(bottom) + (-0.2, 0)$) -- +(0.4, 0);
	\end{tikzpicture}
	$\ \eqdiag \ \theta_b\left( p_{Fcb}^{\beta} \circ F(p_{ijc}^{\alpha}) \circ s_{i, j}^{-1} \circ (\ident_i \opl j_{Fja}^{\gamma}) \circ j_{iab}^{\delta} \right) = \tilde\kappa_X(c)\, \left(\left( A_{i, j, a}^F \right)^{-1}\right)^{\beta}_{\gamma}$
\end{tabular}

It is straightforward to check that these 6j symbols satisfy the orthogonality relations.
We now show that the Biedenharn-Elliott relations (\ref{eq:ber_modf}) hold. The 6j symbols in (\ref{eq:ber_modf}) are only defined if $c, i, j\in G, d, l, m\in X, a, b, k\in Y$, $\alpha=\beta=\gamma=\lambda=1$ and $a=j\opl k, b=i\opl a, c=ij, d=c\opl l$. The left hand side of (\ref{eq:ber_modf}) is easily calculated to be
\[ \tilde\kappa_Y(a)\, \tilde\kappa_Y(k)\, \Psi_Y(i, j, b) \, \left( A_{c, l, k}^F \right)^{\mu}_{\epsilon}. \]
Most summands on the right hand side vanish, with the exception of those where $m=j\opl l$ and $\nu=\sigma=1$. The right hand side is thus
\begin{align*} 
	& \sum_{\rho=1}^{m_{ma}^F} \tilde\kappa_X (m) \, \tilde\kappa_Y(a)\, \tilde\kappa_X(m) \, \tilde\kappa_Y(k) \, \left(A_{i, m, a}^F \right)^{\rho}_{\epsilon} \, \Psi_X(i, j, d) \, \left( A_{j, l, k}^F \right)^{\mu}_{\rho} \\
	=\hphantom{} &\tilde\kappa_Y(a)\, \tilde\kappa_Y(k) \, \Psi_X(i, j, d) \, \left(A_{j, l, k}^F \, A_{i, m, a}^F \right)^{\mu}_{\epsilon}.
\end{align*}
The Biedenharn-Elliott relations now follow from \ref{eq:cond_A}.

Right module functors are treated analogously: If $(F, t): \mathcal M \to \mathcal N$ is a $\mathcal D$-right module functor, we use the same rescaled inclusions and projections for $F(i)$ as above for all $i\in X$. Recall from Section \ref{ssec:vecgo_modf} that $(F, t)$ is uniquely determined by a family of natural numbers $(m_{xy}^F)_{x\in X, y\in Y}$ with
\[ m_{h\opl x, h\opl y}^F = m_{xy}^F \qquad (h\in H, x\in X, y\in Y) \]
and a family of matrices $B_{h, x, y}^F \in \glmat(m_{xy}^F, \mathbb F), h\in H, x\in X, y\in Y,$ that satisfies
\[ B_{gh, x, y}^F = \Phi_X(h^{-1}, g^{-1}, (gh)^{-1}\opl x) \, \Phi_Y^{-1}(h^{-1}, g^{-1}, (gh)^{-1}\opl y) \ B_{g, x, y}^F \, B_{h, g^{-1}\opl x, g^{-1}\opl y}^F \]
for $g, h\in H, x\in X, y\in Y$. Let $l\in H, i, c\in X, a, b\in Y$. The  6j symbols below are only defined if $\gamma=\beta=1$ and $c=l^{-1} \opl i, b=l^{-1} \opl a$. In this case,

\begin{center}
	\begin{tikzpicture}[x=8mm, y=8mm, every node/.style={inner sep=2pt}]
		\newcommand{\sca}{1}
		\node[m morphism] (d) {$\delta$};
		\node[n morphism, on grid, below right=\sca and \sca of d] (g) {$\gamma$};
		\node[m morphism, on grid, below left=\sca and 2*\sca of g] (a) {$\alpha$};
		\node[m morphism, on grid, below right=\sca and \sca of a] (b) {$\beta$};
		
		\draw[color=darkgreen] (d) -- node[right=5pt, pos=0] {$c$} (g);
		\draw[dashed, color=red] (d) -- node[left=3pt] {$F$} (a);
		\draw[color=grey] (g) -- node[right=3pt] {$l$} (b);
		\draw[color=darkgreen] (g) -- node[above=1pt] {$i$} (a);
		\draw[color=blue] (a) -- node[left=5pt, pos=1] {$a$} (b);
		
		\draw[color=blue] (d.north) -- node[left] {$b$} +(0, 0.8) coordinate (top);
		\draw[color=blue] ($(top) + (-0.2, 0)$) -- +(0.4, 0);
		\draw[color=blue] (b.south) -- node[left] {$b$} +(0, -0.8) coordinate (bottom);
		\draw[color=blue] ($(bottom) + (-0.2, 0)$) -- +(0.4, 0);
	\end{tikzpicture}
	$=\tilde\kappa_Y(a)\, \left( B_{l, i, a}^F \right)^\alpha_\delta\, ,$
	\hfil
	\begin{tikzpicture}[x=8mm, y=8mm, every node/.style={inner sep=2pt}]
		\newcommand{\sca}{1}
		\node[m morphism] (d) {$\delta$};
		\node[m morphism, on grid, below left=\sca and \sca of d] (g) {$\gamma$};
		\node[n morphism, on grid, below right=\sca and 2*\sca of g] (a) {$\alpha$};
		\node[m morphism, on grid, below left=\sca and \sca of a] (b) {$\beta$};
		
		\draw[color=blue] (d) -- node[left=5pt, pos=0] {$a$} (g);
		\draw[color=grey] (d) -- node[right=3pt] {$l$} (a);
		\draw[dashed, color=red] (g) -- node[left=3pt] {$F$} (b);
		\draw[color=darkgreen] (g) -- node[above=1pt] {$i$} (a);
		\draw[color=darkgreen] (a) -- node[pos=1, right=5pt] {$c$} (b);
		
		\draw[color=blue] (d.north) -- node[left] {$b$} +(0, 0.8) coordinate (top);
		\draw[color=blue] ($(top) + (-0.2, 0)$) -- +(0.4, 0);
		\draw[color=blue] (b.south) -- node[left] {$b$} +(0, -0.8) coordinate (bottom);
		\draw[color=blue] ($(bottom) + (-0.2, 0)$) -- +(0.4, 0);
	\end{tikzpicture}
	$=\tilde\kappa_X(c)\, \left( \left(B_{l, i, a} \right)^{-1} \right)^\beta_\gamma.$
\end{center}

From the formulas for the 6j symbols of module functors it becomes obvious that the description of $\vectgo$-module functors in terms of matrices that we introduced in Section \ref{ssec:vecgo_modf} can be viewed, up to a sign, as a description in terms of 6j symbols. This is not surprising because the 6j symbols uniquely determine the coherence datum of the module functor.

The above formulas for 6j symbols over $\vectgo$ can be used to calculate state sum models for 3-manifolds with defects as introduced in \cite{bm_tv+}. This extends Dijkgraf-Witten theory, and the data required for these models are relatively easy to handle, at least if we restrict ourselves to manageable module functors. For instance, one could only consider module functors which preserve simple objects. These are classified by certain $G$-equivariant maps and 1-cochains, see Lemma \ref{lem:vecgo_coch_modf} and Remark \ref{rem:vecgo_coch_modf}.

\section*{Acknowledgements}
I wrote the present work as my Master's thesis under the supervision of Professor Catherine Meusburger. It was submitted at Friedrich-Alexander-Universität in Erlangen on the 7th of October 2021.

\clearpage
\urlstyle{same}

\newcommand{\doi}[1]{\href{https://www.doi.org/#1}{#1}}

\addcontentsline{toc}{section}{References}

\end{document}